\documentclass[11pt,leqno,addressasfootnote,noinfoline]{imsart}

\RequirePackage[OT1]{fontenc}
\RequirePackage{amsthm,amsmath,amsfonts}

\usepackage{color}
\usepackage{amsmath}
\usepackage{amssymb}
\usepackage{amsthm}
\usepackage{amsfonts}
\usepackage{verbatim}

\usepackage{graphicx}
 \usepackage[left=2.6cm, right=2.6cm, top=2.6cm, bottom= 2.6cm]{geometry}

\numberwithin{equation}{section}

\newcommand{\N}{\mathbb{N}}
\newcommand{\D}{\mathbb{D}}
\newcommand{\Z}{\mathbb{Z}}
\newcommand{\R}{\mathbb{R}}

\newcommand{\cM}{\mathcal{M}}
\newcommand{\cD}{\mathcal{D}}
\newcommand{\cF}{\mathcal{F}}
\newcommand{\cS}{\mathcal{S}}

\newcommand{\cH}{\mathcal{H}}
\newcommand{\cB}{\mathcal{B}}
\newcommand{\cU}{\mathcal{U}}
\def\restrict#1{\raise-.5ex\hbox{\ensuremath|}_{#1}}

\newcommand\restr[2]{{
  \left.\kern-\nulldelimiterspace 
  #1 
  \vphantom{\big|} 
  \right|_{#2} 
  }}

\newtheorem{theorem}{Theorem}[section]
\newtheorem{lemma}[theorem]{Lemma}
\newtheorem{proposition}[theorem]{Proposition}

\newtheorem{remark}[theorem]{Remark}
\newtheorem{definition}[theorem]{Definition}
\begin{document}

\begin{frontmatter}

\title{The density of the $(\alpha,\beta)$-superprocess and singular solutions to a fractional non-linear PDE}

\begin{aug}
\author{\fnms{\hspace{4 mm} Thomas}
  \snm{Hughes}\thanksref{t1}\corref{}\ead[label=e1]{}} 
\ead[label=e2]{hughes@math.ubc.ca}
\thankstext{t1}{Supported by an NSERC CGS-D Scholarship and Li Tze Fong Memorial Fellowship.} 
\affiliation{
The University of British Columbia\thanksmark{t1} \thanksmark{t2}}

\address{Department of Mathematics, University of British Columbia; 
\printead{e2}}

\end{aug}
\begin{abstract}
We consider the density $X_t(x)$ of the critical $(\alpha,\beta)$-superprocess in $\R^d$ with $\alpha \in (0,2)$ and $\beta < \frac \alpha d$. Our starting point is a recent result from PDE \cite{CVW2016} which implies the following dichotomy: if $x \in \R^d$ is fixed and $\beta \leq \beta^*(\alpha) := \frac{\alpha}{d+\alpha}$, then $X_t(x) > 0$ a.s. on $\{X_t \neq 0 \}$; otherwise, the probability that $X_t(x)$ is positive when conditioned on $\{X_t \neq 0 \}$ has power law decay. We strengthen this and prove probabilistically that if $\beta< \beta^*(\alpha)$ and the density is continuous, which holds if and only if $d=1$ and $\alpha > 1+\beta$, then $X_t(x) > 0$ for all $x \in \R$ a.s. on $\{X_t \neq 0\}$.

The above complements a classical superprocess result that if $X_t$ is non-zero, then it charges every open set almost surely. We unify and extend these results by giving close to sharp conditions on a measure $\mu$ such that $\mu(X_t) := \int X_t(x) \mu(dx) > 0$ a.s. on $\{X_t \neq 0 \}$. Our characterization is based on the size of $\text{supp}(\mu)$, in the sense of Hausdorff measure and dimension. For $s \in [0,d]$, if $\beta \leq \beta^*(\alpha,s) = \frac{\alpha}{d-s+\alpha}$ and $\text{supp}(\mu)$ has positive $x^s$-Hausdorff measure, then $\mu(X_t) >0$ a.s. on $\{X_t \neq 0\}$; and when $\beta > \beta^*(\alpha,s)$, if $\mu$ satisfies a uniform lower density condition which implies $\text{dim}(\text{supp}(\mu)) < s$, then $P(\mu(X_t) = 0 \, | \, X_t \neq 0 ) > 0$.

Our methods also give new results for the fractional PDE which is dual to the $(\alpha,\beta)$-superprocess, i.e.
\[ \partial_t u(t,x) = \Delta_\alpha u(t,x) -  u(t,x)^{1+\beta}\]
with domain $(t,x) \in (0,\infty) \times \R^d$, where $\Delta_\alpha = -(-\Delta)^{\frac \alpha 2}$ is the fractional Laplacian. The initial trace of a solution $u_t(x)$ (see \cite{CV2019}) is a pair $(\cS,\nu)$, where the singular set $\cS$ is a closed set around which local integrals of $u_t(x)$ diverge as $t \downarrow 0$, and $\nu$ is a Radon measure which gives the limiting behaviour of $u_t(x)$ on $\cS^c$ as $t \downarrow 0$. For $\beta < \frac \alpha d$ we characterize the problem of existence of solutions with initial trace $(\cS,0)$ in terms of a parameter called the saturation dimension, $d_{\text{sat}} = d + \alpha(1 - \beta^{-1})$. For $\cS \neq \R^d$ with $\text{dim}(\cS) > d_{\text{sat}}$ (and in some cases with $\text{dim}(\cS) = d_{\text{sat}}$) we prove that no such solution exists. When $\text{dim}(\cS) < d_{\text{sat}}$ and $\cS$ is the compact support of a measure satisfying a uniform lower density condition, we prove that a solution exists.
\end{abstract}

 


\end{frontmatter}

\pagestyle{plain}
\footskip=25pt

\section{Introduction and statement of results}
In this work we study some path properties of the $(\alpha,\beta,d)$-superprocess. For parameters $\alpha \in (0,2)$, $\beta \in(0,1]$, $d \in \N$, the $(\alpha,\beta,d)$-superprocess, or simply the $(\alpha,\beta)$-superprocess when the dimension is fixed, is a strong Markov process taking values in $\cM_F(\R^d)$, the space of finite measures on $\R^d$ equipped with the topology of weak convergence. We will denote it by $X = (X_t : t \geq 0)$, so that $X_t \in \cM_F(\R^d)$. The spatial Markov process associated to $X$ is a symmetric $\alpha$-stable process in $\R^d$ and, for $\beta \in (0,1)$, the branching mechanism is that of a continuous state branching process with $(1+\beta)$-stable branching. When $\beta = 1$ the superprocess is binary branching and the associated continuous state branching process is Feller's branching diffusion.

The paths of $X$ live in $\D([0,\infty),\cM_F(\R^d))$, the space of c\`adl\`ag paths in $\cM_F(\R^d)$. The Markov transition kernel of $X$ is defined by its Laplace functional, which is characterized via a dual relationship with a fractional non-linear evolution equation. For a general probability space realizing $X$ with initial value $X_0 \in \cM_F(\R^d)$, we will write $(\Omega, \cF, P^X_{X_0})$, and denote the associated expectation by $E^X_{X_0}$. Let $\cB^+_b(\R^d) = \cB^+_b$ denote the space of bounded, measurable functions on $\R^d$. Then for every $\phi \in \cB^+_b$,
\begin{equation} \label{e_Lap}
E^X_{X_0}(\exp(-X_t(\phi))) = \exp (-X_0(u^\phi_t)),
\end{equation}
where $X_t(\phi) = \int \phi(x) X_t(dx)$, and $u^\phi_t(x)$ is the unique solution of the evolution equation
\begin{equation} \label{e_evol}
u_t(x) = S_t \phi (x) - \int_0^t S_{t-s}(u^{1+\beta}_s)(x) \,ds
\end{equation}
for $(t,x) \in Q := (0,\infty) \times \R^d$. $(S_t)_{t\geq0}$ denotes the transition semigroup of the isotropic (symmetric) $\alpha$-stable process in $\R^d$. (See Theorem 4.4.1 of \cite{Dawson} for the existence and uniqueness of solutions to \eqref{e_evol} and the derivation of \eqref{e_Lap}.) The generator of the $\alpha$-stable process is the fractional Laplacian $\Delta_\alpha = -(-\Delta)^{\frac \alpha 2}$. We will use the probabilistic convention, so that $\Delta_\alpha$ corresponds to the $\alpha$-stable process; the actual ``fraction" associated to this operator, and the parameter commonly used in the partial differential equation (PDE) literature, is then $\frac \alpha 2 \in (0,1)$. In the singular integral formulation, $\Delta_\alpha$ is defined as
\begin{equation}
\Delta_\alpha f(x) := \lim_{\epsilon \to 0}  -a_{\alpha,d}\int_{\R^d} \frac{f(x) - f(y)}{|x-y|^{d+ \alpha}} \chi_\epsilon(|x-y|)\,dy \nonumber
\end{equation}
for a constant $a_{\alpha,d} > 0$, where
\begin{equation}
\chi_\epsilon(r) = \begin{cases}1 &\text{ if } r > \epsilon, \\ 0 &\text{ if } r \in [0,\epsilon]. \end{cases} \nonumber
\end{equation}
The integral equation \eqref{e_evol} corresponds to the fractional PDE
\begin{equation} \label{e_pde}
\partial_t u = \Delta_\alpha u - u^{1+\beta}.
\end{equation}
When they exist, solutions of \eqref{e_pde} and \eqref{e_evol} generally coincide. In Section~\ref{s_duality} we define weak solutions to \eqref{e_pde} (see Definition~\ref{def_weaksol_IVP}) and make this correspondence rigorous in certain cases. In addition to studying solutions of \eqref{e_pde} (or \eqref{e_evol}) as a means of proving properties of $X$, we also prove novel results concerning the existence and non-existence of solutions to \eqref{e_pde} with very singular initial conditions. 

We are interested in the density of $X_t$. It is a classical result of Fleischmann \cite{F1988} that $X_t$ is absolutely continuous if and only if $\beta < \frac \alpha d$. This work is concerned only with absolutely continuous case, and we restrict to it now.\\

\noindent \textbf{Assumption.} For the remainder of this work, we assume $\beta < \frac \alpha  d$.\\

Under this assumption, $X_t$ has a density $X_t(x)$, so that $X_t(dx) = X_t(x) dx$. A priori, a generic density is only defined up to Lebesgue-null sets. Much of this work concerns the behaviour of the density on such sets, so this is not sufficient. In particular, we need the object
\begin{equation} \label{e_muXt}
\mu(X_t) := \int_{\R^d} X_t(x) \mu(dx)
\end{equation}
to be well-defined for $\mu \in \cM_F(\R^d)$. Let $B(x,r)$ denote the closed ball of radius $r>0$ around $x \in \R^d$. We define
\begin{equation} \label{e_Xepsilon}
X^\epsilon_t(x) = \frac{X_t(B(x,\epsilon))}{|B(x,\epsilon)|},
\end{equation}
where $|A|$ denotes the Lebesgue measure of $A \subset \R^d$. By the Lebesgue Differentiation Theorem, $X^\epsilon_t(x)$ converges as $\epsilon \downarrow 0$ for Lebesgue-a.e. $x \in \R^d$, and consequently
\begin{equation} \label{e_density_liminf}
X_t(x) := \liminf_{\epsilon \downarrow 0} X^\epsilon_t(x)
\end{equation}
is a density for $X_t$. In fact, more holds.
\begin{lemma} For fixed $x \in \R^d$, $X^\epsilon_t(x) \to X_t(x)$ $P^X_{X_0}$-a.s as $\epsilon \downarrow 0$. Moreover, for $\mu \in \cM_F(\R^d)$, $X^\epsilon_t(x) \to X_t(x)$ for $\mu$-a.e. $x$ almost surely and $\mu(X^\epsilon_t) \to \mu(X_t)$ in $L^1(P^X_{X_0})$ as $\epsilon \downarrow 0$.
\end{lemma}
The above is in fact an abridged version of Lemma~\ref{lemma_density_convergence}, which is proved in Section~\ref{s_density}. Alongside this, as we discuss in Section~\ref{s_duality}, the evolution equation has a unique solution with initial condition given by a finite measure. Consequently, both sides of \eqref{e_Lap} are well-defined when $\phi$ is replaced with a finite measure. The measure case can then be summarized as follows: (see Lemma~\ref{lemma_measure_duality} for a precise statement) for $\mu \in \cM_F(\R^d)$, there is a unique solution $u^\mu_t(x)$ to \eqref{e_pde} on $Q$ such that $u^\mu_t \to \mu$ weakly in the sense of measures as $t \downarrow 0$. The solution $u^\mu_t$ satisfies
\begin{equation} \label{e_Lapmeasure}
E^X_{X_0}(\exp(-\mu(X_t))) = \exp (-X_0(u^\mu_t)).
\end{equation}
In the above, $\mu(X_t)$ is defined by \eqref{e_muXt} with the density $X_t(x)$ from \eqref{e_density_liminf}.

Another approach to specifying a canonical version of the density is via a Green's function representation for $X_t$, which is the method used by Fleischmann, Mytnik and Wachtel in \cite{FMW2010} and other works. This version of the density is given by
\begin{equation} \label{e_density_GFR}
X_t(x) = X_0 * p_t (x) + \int_{(0,t] \times \R^d} p_{t-s}(y-x) \, M(d(s,y)).
\end{equation}
In the above, $p_t$ is the transition density of the symmetric $\alpha$-stable process (see Section~\ref{s_transitiondensities}) and the measure $M$ is a compensated stable martingale measure associated to $X$. Although we do not show it here, the density above will agree with the version we use.

The canonical measure associated with the $(\alpha,\beta)$-superprocess, which we denote $\N_0$, is a \linebreak$\sigma$-finite measure supported on $\D([0,\infty),\cM_F(\R^d))$, the space of c\`adl\`ag $\cM_F(\R^d)$-valued paths. In the construction of superprocesses as scaling limits of discrete spatial branching models, one takes the number of individuals in the population to infinity while their masses are simultaneously scaled to $0$. The canonical measure $\N_0$ is then the ``law" of the superprocess when started with a single (infinitessimal) ancestor at the origin. Likewise, $\N_x$ describes the superprocess descending from an ancestor located at $x \in \R^d$. In Section~\ref{s_cluster} we describe the relationship between canonical measure and the superprocess with law $P^X_{X_0}$. The following formula will be used quite frequently: for any $x \in  \R^d$, for  $t>0$,
\begin{equation} \label{e_canon_survive}
\N_x( X_t \neq 0 ) = U_t := \left( \frac{1}{\beta t} \right)^{\frac 1 \beta}.
\end{equation}
(This is shown, for example, in (5.4.2) of \cite{DawsonInfDiv1991}.) We remark that $U_t$ is the maximal solution on $(0,\infty)$ to the ODE $u' = - u^{1+\beta}$. By \eqref{e_canon_survive}, when considering $X_t$ under $\N_x$ for some fixed $t>0$, one is essentially working with a finite measure. This formula is a consequence of the fact that the canonical measure also shares a close relationship with the dual evolution equation. If $u^\phi_t(x)$ is as in \eqref{e_evol} for $\phi \in \cB^+_b$, then
\begin{equation} \label{e_Lapcanon}
\N_x( 1 - \exp(-(X_t,\phi))) = u^\phi_t(x).
\end{equation}
As with $P_{X_0}^X$, the above relationship can be generalized to include measures when $X_t$ has a density. If $u^\mu_t$ is the unique solution to \eqref{e_pde} with initial condition $\mu \in \cM_F(\R^d)$, (see Lemma~\ref{lemma_measure_duality} for details) then
\begin{equation} \label{e_Lapcanonmeasure}
\N_x(1 - \exp(-\mu(X_t))) = u^\mu_t(x).
\end{equation}

We now motivate our results with the statement of two theorems. The first is a fundamental result about superprocesses associated to $\alpha$-stable spatial motions. The result is due to Perkins and Evans, whose proof of the $\beta = 1$ case appeared in \cite{EP1991}. The proof for $\beta < 1$ appears in a more recent work of Li and Zhou \cite{LZ2008}. We define $\text{supp}(\mu)$ to be the closed support of $\mu \in \cM_F(\R^d)$ and will often refer to it as the support of $\mu$. \\

\noindent \textbf{Theorem A.} [Evans, Perkins (1991); Li, Zhou (2008)] \emph{For $X_0 \in \cM_F(\R^d)$ and $t>0$, 
\[P_{X_0}^X(\text{supp}(X_t) = \R^d \text{ or } \emptyset) = 1.\]
Similarly, $\N_0( \text{supp}(X_t) = \R^d \, | \, X_t \neq 0 ) = 1$.}\\

The superprocesses we study are critical and therefore go extinct almost surely. That is, \linebreak$P_{X_0}^X(X_t \neq 0 \, \text{ for all } t>0) = 0$. (This is analogous to the almost sure extinction of critical branching processes, whose spatial analogues have scaling limits given by superprocesses.) Thus, the above statement can be understood to say that, conditioned on non-extinction at time $t>0$, (i.e. conditioned on $\{X_t \neq 0\}$) $\text{supp}(X_t) = \R^d$ a.s.

Theorem A is sometimes called \textit{instantaneous propagation.} This is because, regardless of the choice of initial measure, $X_t$ has mass ``everywhere" in $\R^d$. So for $X_0 = \delta_x$, varying $x$ over $\R^d$ has no influence on the support of $X_t$, which is $\R^d$ almost surely on $\{X_t \neq 0 \}$ for any choice of $x$. The condition $\text{supp}(X_t) = \R^d$ is equivalent to: $X_t(U) > 0$ for every open set $U \subset \R^d$, where $X_t(U) = (X_t,1_U)$ and $1_U$ is the indicator function of $U$. So for open $U$, $X_t > 0$ implies $X_t(U)>0$ almost surely. For $\lambda>0$, consider now $u^{\lambda 1_U}$, which is the solution to
\[\partial_t u = \Delta_\alpha u -  u^{1+\beta}, \,\, u_0 = \lambda 1_U. \]
By \eqref{e_Lapcanon},
\[\N_x(1-\exp(-\lambda X_t(U))) = -u^{\lambda 1_U}_t(x).\]
Taking $\lambda \to \infty$ with standard limiting arguments, we observe that
\[\lim_{\lambda \to \infty} \N_x(1-\exp(-\lambda X_t(U))) = \N_x(X_t(U) > 0).\]
By instantaneous propagation, the right hand side is equal to $\N_x(X_t \neq 0)$, which, by translation invariance, is equal for any choice of $x$. The implication is that
\[\lim_{\lambda \to \infty} u^{\lambda 1_U}_t(x) = \N_0(X_t \neq 0)\]
for all $x \in \R^d$. Hence $u^{\infty 1_U}_t = \lim_{\lambda \to \infty} u^{\lambda 1_U}_t$ is constant in space for each $t>0$. In particular, by \eqref{e_canon_survive}, for all $x \in \R^d$ we have $u^{\infty 1_U}_t(x) = U_t$.

\begin{remark} Using the same argument, it follows that $\N_0(X_t(\phi) > 0 ) = U_t$ for any measurable function $\phi \geq 0$ which is positive on a set of positive Lebesgue measure. In particular, $\lim_{\lambda \to \infty} u^{\lambda \phi}_t = U_t$ for any such function. These probabilistic results, which we believe are not widely known in the PDE community, imply the non-existence of positive solutions to \eqref{e_pde} whose singular sets (see \eqref{e_inittrace_singular}) have positive Lebesgue measure.
\end{remark}

The other result motivating ours is a recent result in the PDE literature which, to our knowledge, is not yet widely known in the probability community. It is due to Chen, V\'eron, and Wang \cite{CVW2016}. We first give the result originally stated for PDE, then interpret probabilistically. Let $u^{\lambda}_t$ denote the solution to \eqref{e_pde} with initial data $\lambda \delta_0$ and let $u^\infty_t = \lim_{\lambda \to \infty} u^{\lambda}_t$. 

Let $\beta^*(\alpha) = \frac{\alpha}{d+\alpha}$. We will generally view $d$ as fixed and therefore omit the dependence of $\beta^*(\alpha)$ on $d$. \\

\noindent \textbf{Theorem B.} [Chen, V\'eron, Wang (2017); Chen, V\'eron (2019)] Let $t>0$ and $x \in \R^d$.

\noindent\emph{(a) Let $\beta \leq \beta^*(\alpha)$. Then $u^\infty_t = U_t$.}

\noindent \emph{(b) Let $\beta^*(\alpha) < \beta < \frac \alpha d$. Then $u^\infty_t$ satisfies
\begin{equation}
 C_1\frac{t^{-\frac 1 \beta}}{1+ |t^{-\frac 1 \alpha}x|^{d+\alpha}} \leq u^\infty_t(x) \leq C_2\frac{t^{-\frac  1 \beta} \log(e+|t^{-\frac 1 \alpha} x|)}{1+ |t^{-\frac 1 \alpha}x|^{d+\alpha}}
\end{equation}
for constants $0<C_1<C_2$.} \\

We introduce some terminology: if a solution to \eqref{e_pde}, or a limit of solutions is equal to $U_t$, then we will call the solution (or limit) \textit{flat}, because these solutions are constant for fixed $t>0$. Otherwise it is \textit{non-flat}. Thus $u^\infty_t$ is flat when $\beta \leq \beta^*(\alpha)$ and non-flat when $\beta^*(\alpha) < \beta < \frac \alpha d$.

With the exception of the case that $\beta = \beta^*(\alpha)$, this result was proved in \cite{CVW2016}, while the $\beta = \beta^*(\alpha)$ case was proved by the first two authors of that paper in \cite{CV2019}. Rather surprisingly, depending on the parameters $(\alpha,\beta,d)$, $u^\infty_t(x)$ is either flat or has almost the same asymptotic decay as the heat kernel associated to $\Delta_\alpha$. We interpret the above probabilistically. From \eqref{e_Lapcanonmeasure} with $\mu = \lambda \delta_x$, we have
\begin{equation}
\N_0(1 - \exp (-\lambda X_t(x))) = u^{\lambda}_t(x), \nonumber
\end{equation}
where $X_t(x)$ is the density of $X_t$ at $x$. Taking $\lambda \to \infty$ yields 
\begin{equation} \label{e_pointpdeinf}
\N_0(X_t(x) > 0) = u^{\infty}_t(x),
\end{equation}
from which the following is an immediate consequence of Theorem B.

\begin{theorem} \label{thm_pointdichotomy} Fix $t >0$. The following hold under $\N_0$ and $P_{X_0}^X$ for $X_0 \in \cM_F(\R^d)$. 

\noindent(a) Let $\beta \leq \beta^*(\alpha)$. Then for fixed $x \in \R^d$, $X_t(x) >0$ almost surely on $\{X_t \neq 0 \}$. In particular, $X_t(x) >0$ for Lebesgue-a.e. $x \in \R^d$ almost surely on $\{X_t \neq 0 \}$.

\noindent(b) Let $\beta^*(\alpha) < \beta < \alpha / d$. Then $\{x : X_t(x) > 0\}$ has finite Lebesgue measure almost surely. Moreover, we have
\begin{equation} \label{e_pointchargebounds}
 C_1 \frac{t^{-\frac  1 \beta}}{1+ |t^{-\frac 1 \alpha}x|^{d+\alpha}} \leq \N_0( X_t(x) > 0 ) \leq C_2\frac{t^{-\frac  1 \beta} \log(e+|t^{-\frac 1 \alpha} x|)}{1+ |t^{-\frac 1 \alpha}x|^{d+\alpha}}
\end{equation}
for constants $0<C_1<C_2$.
\end{theorem}

Part (a) follows from Theorem B(a) as follows: by \eqref{e_canon_survive} and \eqref{e_pointpdeinf}, when $\beta \leq \beta^*(\alpha)$ we have
\[ \N_0(X_t(x) > 0 ) = \N_0(X_t \neq 0 ).\]
Since $\{X_t(x) > 0 \} \subseteq \{X_t \neq 0 \}$, it follows that $\N_x(X_t(x) = 0 \, | \, X_t \neq 0 ) = 0$ and Fubini's theorem implies that $X_t(x) > 0$ for Lebesgue-a.e. $x$ almost surely under $\N_0(\cdot \, | \, X_t \neq 0 )$. In part (b), the fact that $\{x: X_t(x) > 0 \}$ has finite Lebesgue measure under $\N_0$ follows from integrability of the upper bound in \eqref{e_pointchargebounds} and Fubini's theorem. For $P^X_{X_0}$, the result then follows from the cluster representation for the superprocess (see Section~\ref{s_cluster}).

In keeping with the terminology of instantaneous propagation for the behaviour from Theorem A, we propose to call \textit{strong instantaneous propagation} the property described for $\beta \leq \beta^*(\alpha)$. Both results say, in some sense, that $X_t$ has mass ``everywhere." Instantaneous propagation describes this on the level of mass on open sets, whereas strong instantaneous propagation concerns the density at a fixed point.

A priori, taken as an immediate consequence of Theorem B, we lack a probabilistic intuition for this result, as Theorem B is proved using analytical methods. In Section~\ref{s_points} we give a probabilistic proof of part (a). The arguments there are prototypical of other arguments used later on to prove some of our other results.\\

\noindent \textbf{Open Problem.} Give a probabilistic proof of Theorem~\ref{thm_pointdichotomy}(b).\\

We now strengthen one of the conclusions of Theorem~\ref{thm_pointdichotomy} when a continuous version of the density $X_t(\cdot)$ exists. The starting point for this is another dichotomy for the $(\alpha,\beta)$-superprocess, which was proved by Fleischmann, Mytnik and Wachtel \cite{FMW2010}.\\ 

\noindent \textbf{Theorem C.} [Fleischmann, Mytnik, Wachtel (2010)] \emph{Fix $t>0$ and consider $X_t$ under $P^X_{X_0}$ with $X_0 \in \cM_F(\R^d)$. \\
(a) If $d = 1$ and $\alpha > 1 + \beta$, then there is a version of the density $X_t(\cdot)$ which is locally $\eta$-H\"older continuous for all $\eta < \frac{\alpha}{1+\beta} - 1$. }

\noindent \emph{(b) If $d >1$, or $d=1$ and $\alpha \leq 1+\beta$, then $\| 1_U X_t(\cdot) \|_\infty = \infty$ for all open sets $U$ almost surely on $\{X_t \neq 0 \}$.} \\

We will refer to the parameter regime $d=1$, $\alpha > 1+ \beta$ as the \textit{continuous case}. In this case it is understood that we will always work with the continuous version of the density $x \to X_t(x)$, and it is easy to see that this version agrees with the version defined in \eqref{e_density_liminf}. The case from Theorem C(b) is then the \textit{discontinuous case}. For $d=1$, in the parameter regime with $\alpha > 1+\beta$ and $\beta < \beta^*(\alpha)$, the density enjoys both continuity and strong instantaneous propagation, and we are able to show the following strict positivity result.

\begin{theorem} \label{thm_positiveEverwhere} Let $d=1$, $\alpha > 1 + \beta$ and $\beta < \beta^*(\alpha)$. Then for $t>0$,
\[ X_t(x) > 0 \,\, \text{ for all $x \in \R$ almost surely on $\{X_t \neq 0 \}$} \]
under both $P_{X_0}^X$ and $\N_0$.
\end{theorem}
\begin{remark} Strong instantaneous propagation still holds when $\beta = \beta^*(\alpha)$, but the proof of the above does not, and this case is left open. Of course, $\{x : X_t(x) > 0\}$ has finite Lebesgue measure when $\beta > \beta^*(\alpha)$ by Theorem~\ref{thm_pointdichotomy}(b). \end{remark}

This is an interesting property and is perhaps best understood in the context of similar results for non-negative solutions to stochastic PDE (SPDE). Consider first the stochastic heat equation
\begin{equation} \label{e_SHE}
 \partial_t Y_t(x) = \Delta Y_t(x) + Y_t(x)^\gamma \dot{W}(t,x)
\end{equation}
on $(0,\infty) \times \R$, where $\dot{W}(t,x)$ is space-time white Gaussian noise and $\gamma >0$. It was shown by Mueller in \cite{M1991} that if $\gamma \geq 1$, then a non-negative solution $Y_t(x)$ satisfies $Y_t(x) > 0$ for all $(t,x) \in (0,\infty) \times \R$. On the other hand, Perkins and Mueller \cite{MP1992} have shown that when $\gamma < 1$, if $Y_0$ has compact support then $Y_t$ has compact support for all $t>0$ a.s. The case $\gamma = \frac 1 2$ corresponds to the density of super-Brownian motion, in which case the result was originally due to Iscoe \cite{I1998}.

In the SPDE associated to the $(\alpha,\beta)$-superprocess with $\beta<1$ in dimension one, the diffusion term is replaced with fractional diffusion and the noise is stable rather than Gaussian. In particular, the density $X_t(x)$ of the $(\alpha,\beta)$-superprocess solves the SPDE
\begin{equation}\label{e_stable_spde} \partial_t X_t(x) = \Delta_\alpha X_t(x) + X_{t^-}(x)^{\frac{1}{1+\beta}} \dot{L}(t,x),\end{equation}
where $\dot{L}(t,x)$ is a spectrally positive space-time stable noise of index $1+\beta$. The Green's function representation \eqref{e_density_GFR} established in \cite{FMW2010} can be viewed as a mild form of solution but only applies at fixed times. Alternatively, weak solutions to the SPDE obtained by replacing $\Delta_\alpha$ with $\Delta$ in \eqref{e_stable_spde} were studied in \cite{M2002}. The key results of that work (Propositions 4.1 and 5.1) could be generalized to construct weak solutions to \eqref{e_stable_spde}.

In this context, Theorem~\ref{thm_positiveEverwhere} is a (fixed time) strict positivity result for a fractional SPDE with stable noise, and it is the first such result that we are aware of. Furthermore, because \linebreak$\{x : X_t(x) >0\}$ has finite Lebesgue measure when $\beta > \beta^*(\alpha)$, it is apparent that the interplay of fractional diffusion and stable noise with non-Lipschitz coefficients lead to non-trivial behaviour which is not seen in Gaussian SPDE like \eqref{e_SHE} (c.f. the results from \cite{M1991} and \cite{MP1992} discussed above).

In order to best understand the results that follow, it will be useful to view Theorem A and Theorem~\ref{thm_pointdichotomy} from a unified perspective. Let $m_U$ denote the Lebesgue measure restricted to $U \subset \R^d$. Theorem A states that if $U$ is open (so $m_U$ is non-zero), $\int X_t(x) m_U(dx) > 0$ a.s. on $\{X_t \neq 0 \}$. Thus $X_t(\cdot)$ has mass everywhere in $\R^d$ at a \textit{macroscopic} level--the level of open sets. In terms of the PDE \eqref{e_pde}, the equivalent statement is that $\lim_{\lambda \to \infty} u^{\lambda m_U}_t(x)$ is equal to the flat solution $U_t$. The Lebesgue measure on an open set $U$ is, locally speaking, the most spread out measure on $\R^d$. The least spread out measure, that is the most concentrated measure, is $\delta_{x}$ for some $x \in \R^d$. Theorem~\ref{thm_pointdichotomy} states that $X_t(\cdot)$ integrated against the measure $\delta_{x}$, i.e. $X_t(x)$, is positive a.s. on $\{X_t \neq 0 \}$ \textit{if and only if} $\beta \leq \frac{\alpha}{d+\alpha}$, again with an equivalent interpretation that $\lim_{\lambda \to \infty} u^{\lambda \delta_x} = U_t$. In this case, $X_t(\cdot)$ has mass everywhere at a \textit{microscopic} level--at a single point. Taken together, these two results describe the almost sure behaviour of $X_t(\cdot)$ on both the most concentrated and least concentrated measures on $\R^d$: $\delta_x$ and $m_U$, respectively. Our results that follow interpolate between these two extremes to describe the almost sure behaviour of $\mu(X_t)$ for general measures $\mu$. We are therefore also able to answer some questions about flatness and non-flatness of solutions to \eqref{e_pde} that are the limits as $\lambda \to \infty$ of solutions with initial condition $\lambda \mu$. Indeed, a general principle underlying this work can be summarized as follows:

\begin{center}\emph{$\mu(X_t) > 0$ almost surely on $\{X_t \neq 0\}$ if and only if $\lim_{\lambda \to \infty} u^{\lambda \mu}_t = U_t$.}\end{center}

A condition which quantifies the size of a measure, in the sense of the size of its support, is the following \textit{mass distribution property}. A measure $\mu \in \cM_F(\R^d)$ satisfies condition (F1) with parameter $s \in [0,d]$, or simply (F1)-$s$, if:
\begin{itemize}
\item[] (F1)-$s$ \hspace{2 mm} For some constant $\overline{C} > 0$, for all $x \in \R^d$ and $r>0$,
\begin{equation} 
\mu(B(x,r)) \leq \overline{C} r^s. \nonumber 
\end{equation}
\end{itemize}
This condition means that $\mu$ is spread out in the sense that its support is large, i.e. at least $s$-dimensional. In particular, if $\mu$ satisfies (F1)-$s$, then $\text{dim}(\text{supp}(\mu)) \geq s$, where $\text{dim(A)}$ denotes the Hausdorff dimension of $A \subset \R^d$. (See Frostman's Lemma below.)

We view $\{X_t(x) > 0 \}$ as the event that $X_t(\cdot)$ charges the measure $\delta_x$. Our results give a partial answer the question: \textit{which measures does $X_t(\cdot)$ charge almost surely?} 

\begin{definition} For $s \in [0,d]$, let $\beta^*(\alpha,s) = \frac{\alpha}{(d-s) + \alpha}$. \end{definition}
We remark that $\beta^*(\alpha,0) = \beta^*(\alpha)$, the critical parameter from Theorem~\ref{thm_pointdichotomy}. Furthermore, $\beta^*(\alpha,d) = 1$,  which is the critical parameter implicit in Theorem~A, since the $(\alpha,\beta)$-superprocess has instantaneous propagation for all $\beta \leq 1$. Recall from \eqref{e_Lapcanonmeasure} that
\[ \N_x (1-\exp(-\lambda \mu(X_t))) = u^{\lambda \mu}_t(x). \]
As $\lambda \to \infty$ the left hand side increases to $\N_x(\mu(X_t) > 0) < \infty$. 
\begin{definition}
Let $u^{\infty \mu}_t(x) = \lim_{\lambda \to \infty} u^{\lambda \mu}_t(x)$.
\end{definition}
We therefore have
\begin{equation} \label{e_uinf_canonrep}
\N_x(\mu(X_t) > 0 ) = u^{\infty \mu}_t(x).
\end{equation}
Since $\{\mu(X_t) > 0 \} \subseteq \{X_t \neq 0 \}$,  \eqref{e_canon_survive} and \eqref{e_uinf_canonrep} imply that
\begin{equation} \label{e_uinf_bd}
\sup_x u^{\infty \mu}_t(x) \leq U_t < \infty.
\end{equation}

The results that follow are fixed time results. For Theorems~\ref{thm_flatFrost}, \ref{thm_flatDim}, \ref{thm_nonflat}(a) and \ref{thm_cont}, and the discussion of these results, we fix $t>0$.

\begin{theorem} \label{thm_flatFrost} Suppose that $\mu \in \cM_F(\R^d)$ satisfies (F1)-$s$ for some $s \in [0,d]$ and let $\beta \leq \beta^*(\alpha,s)$.

\noindent (a) $\mu(X_t) > 0$ a.s. on $\{X_t \neq 0 \}$ under $\N_0$ and $P_{X_0}^X$. Equivalently, we have $u^{\infty \mu}_t = U_t$. 

\noindent (b) For any $\nu \in \cM_F(\R^d)$ with $\text{supp}(\mu) \subseteq \text{supp}(\nu)$, $\nu(X_t) > 0$ a.s. on $\{X_t \neq 0\}$ under $\N_0$ and $P^{X}_{X_0}$, and $u^{\infty \nu}_t = U_t$.
\end{theorem}

Thus if $\mu$ is spread out in the sense of (F1)-$s$ and $\beta \leq \beta^*(\alpha,s)$, the density charges $\mu$ almost surely when $X_t$ is conditioned on survival. What is more surprising is part (b), which states that it is the closed support of a measure $\nu$ rather than its particular properties which ensure that $\nu(X_t) >0$ almost surely on $\{X_t \neq 0\}$. We next show how the above leads to a more general result using Frostman's Lemma. For $\cS \subset \R^d$, let $\cH^s(\cS)$ denote the $x^s$-Hausdorff measure of $\cS$ and recall that $\text{dim}(\cS)$ denotes the Hausdorff dimension of $\cS$. Let $\cM_F(\cS)$ denote the space of finite measures $\mu$ with $\text{supp}(\mu) \subseteq \cS$.\\

\noindent \textbf{Frostman's Lemma.}\emph{ Suppose that $\cS \subset \R^d$ is Borel. Then $\cH^s(\cS) > 0$ if and only if $\,\exists \,\mu \in \cM_F(\cS)$ satisfying (F1)-$s$. }\\

See Theorem 8.8 in \cite{Mattila} for a proof. The following is an immediate consequence of Theorem~\ref{thm_flatFrost} and Frostman's Lemma.

\begin{theorem} \label{thm_flatDim} Let $s \in [0,d]$ and suppose $\beta \leq \beta^*(\alpha,s)$. If $\mu \in \cM_F(\R^d)$ satisfies $\cH^s(\text{supp}(\mu)) > 0$, then $\mu(X_t) > 0$ a.s. on $\{X_t \neq 0 \}$ under $P^X_{X_0}$ and $\N_0$, and $u^{\infty \mu}_t = U_t$.
\end{theorem}

Recall that if $\text{dim}(\text{supp}(\mu)) = s$, then $\cH^{s'}(\text{supp}(\mu)) = +\infty$ for all $s' < s$. In particular, when $\beta < \beta^*(\alpha,s)$ the conclusions above hold when $\text{dim}(\text{supp}(\mu)) \geq s$. The assertions in Theorems~\ref{thm_flatFrost} and \ref{thm_flatDim} that $u^{\infty \mu}_t = U_t$ complement Theorem G of \cite{CV2019}, which proves a similar result when $s$ is an integer and the set (in this setting $\text{supp}(\mu)$) is a line or hyperplane. Their result is stated in the language of initial traces, which we discuss later. Our critical parameter $\beta^*(\alpha, s)$ agrees with the critical parameter of their result. 

Observe that $\beta^*(\alpha,\alpha) = \frac \alpha d$. Since $\beta^*(\alpha,s)$ is increasing in $s$, it follows that, if $\beta < \frac \alpha d$, then $\beta < \beta^*(\alpha, s)$ for all $s \geq \alpha$. The first condition is required for the existence of the density, and hence if $s \geq \alpha$, $\beta < \beta^*(\alpha,s)$ holds over the entire parameter set in which we are interested. In other words, whenever the density exists, almost surely it charges any $\mu \in \cM_F(\R^d)$ such that $\text{dim}(\text{supp}(\mu)) \geq \alpha$. (This also complements an observation made in \cite{CV2019}.)

The results of Theorems~\ref{thm_flatFrost} and \ref{thm_flatDim} are sharp. When $\beta > \beta^*(\alpha,s)$, under complimentary assumptions on the measure $\mu$, $u^{\infty \mu}_t$ is non-flat and $\mu(X_t) = 0$ with positive probability on \linebreak$\{X_t \neq 0\}$. In order to state this result precisely we introduce the condition (F2). We say that $\mu \in \cM_F(\R^d)$ satisfies property (F2) with parameter $s \in [0,d]$, or (F2)-$s$, if:
\begin{itemize}
\item[] (F2)-$s$ \hspace{2 mm} For some constant $\underline{C}>0$, for all $x \in \text{supp}(\mu)$ and $r\in(0,1]$,
\begin{equation}
\mu(B(x,r)) \geq \underline{C} r^s. \nonumber
\end{equation}
\end{itemize}
In contrast with (F1), which implies that the mass of $\mu$ is spread out in a certain sense, (F2) tells us that the mass of $\mu$ is not \textit{too} spread out in that same sense. In particular, (F2)-$s$ implies that $\text{dim}(\text{supp}(\mu)) \leq s$ (for example, see Section 8.7 of \cite{Heinonen}).

Let $d(x,\cS) = \inf_{y \in \cS}|x-y|$ denote the distance between $x \in \R^d$ and a set $\cS \subset \R^d$. The restriction $s < \alpha$ in the following is because this is required for $(\beta^*(\alpha,s), \frac \alpha d)$ to be non-empty.

\begin{theorem} \label{thm_nonflat} Suppose that $\mu \in \cM_F(\R^d)$ satisfies (F2)-$s$ for some $s \in [0,\alpha)$ and has compact support $\text{supp}(\mu) = \cS$. Let $ \beta^*(\alpha,s) < \beta < \frac \alpha d$. \\
\noindent(a) For $x \in \R^d$ and $X_0 \in \cM_F(\R^d)$, $\N_x(\mu(X_t) = 0 \, | \, X_t \neq 0 ) >0$ and $P^X_{X_0}(\mu(X_t) = 0 \, | \, X_t \neq 0 ) > 0$. \\
\noindent(b) $\N_x(\mu(X_t) > 0 ) = u^{\infty \mu}_t(x)$ satisfies the following: there are constants $C_{\ref{e_thmnonflat_upperbound}} > c_{\ref{e_thmnonflat_upperbound}} >0$ such that for all $(t,x) \in Q$,
\begin{equation} \label{e_thmnonflat_upperbound}
\frac{c_{\ref{e_thmnonflat_upperbound}} t^{-\frac 1 \beta}}{1+|t^{-\frac 1 \alpha}d(x,\cS)|^{d+\alpha}}\leq \N_x(\mu(X_t) > 0) \leq C_{\ref{e_thmnonflat_upperbound}} \left[t^{-\frac 1 \beta - \frac s \alpha} \vee t^{-\frac 1 \beta}\right] \frac{ \log(e+t^{-\frac 1 \alpha}d(x,\cS))}{1+ |t^{-\frac 1 \alpha}d(x,\cS)|^{d+\alpha}}.
\end{equation}
In particular, $\N_x(\mu(X_t) > 0)$ vanishes uniformly on $\{x : d(x,\cS) \geq \rho \}$ as $t\downarrow 0$, for all $\rho > 0$.\\
\noindent(c) If, in addition, $\mu$ satisfies (F1)-$s$, then there is a constant $c_{\ref{e_thmnonflat_lowerbound}} >0$ such that for all $x \in \R^d$ and $t \in (0,1]$,
\begin{equation} \label{e_thmnonflat_lowerbound}
\N_x(\mu(X_t) > 0 ) \geq c_{\ref{e_thmnonflat_lowerbound}} \left[\frac{  t^{-\frac 1 \beta - \frac s \alpha}  }{1+|t^{-\frac 1 \alpha} [d(x,\cS) + \text{diam}(\cS)]|^{d+\alpha}} \right].
\end{equation}
\noindent(d) For any $\nu \in \cM_F(\cS)$, the conclusions of part (a) hold when $\mu$ is replaced with $\nu$. Furthermore, for constants $C_{\ref{e_thmnonflat_upperbound_nu}} > c_{\ref{e_thmnonflat_upperbound_nu}} >0$ we have
\begin{equation} \label{e_thmnonflat_upperbound_nu}
\frac{c_{\ref{e_thmnonflat_upperbound_nu}} t^{-\frac 1 \beta}}{1+|t^{-\frac 1 \alpha}d(x,\text{supp}(\nu))|^{d+\alpha}}\leq \N_x(\nu(X_t) > 0) \leq C_{\ref{e_thmnonflat_upperbound_nu}} \left[t^{-\frac 1 \beta - \frac s \alpha} \vee t^{-\frac 1 \beta}\right] \frac{ \log(e+t^{-\frac 1 \alpha}d(x,\cS))}{1+ |t^{-\frac 1 \alpha}d(x,\cS)|^{d+\alpha}}.
\end{equation}
\end{theorem}

Theorems~\ref{thm_flatDim} and \ref{thm_nonflat} give a sharp picture of the behaviour of $X_t(\cdot)$ when integrated against an $s$-dimensional measure. Restrict to the event $\{X_t \neq 0 \}$. Then we have that when $\mu$ is \textit{at least} $s$-dimensional (i.e. $\cH^s(\text{supp}(\mu))>0$) and $\beta \leq \beta^*(\alpha,s)$, $X_t(\cdot)$ charges $\mu$ a.s., and when $\mu$ is \textit{at most} $s$-dimensional (by this we mean $\mu$ satisfies (F2)-$s$) and $\beta > \beta^*(\alpha,s)$, the probability that $X_t(\cdot)$ charges $\mu$ has spatial decay similar to an $\alpha$-stable heat kernel. Theorem A and Theorem~\ref{thm_pointdichotomy} are the special cases of $s=d$ and $s=0$, respectively; our results cover $s \in [0,d]$. Furthermore, for $s \in [0,\alpha)$ there is a non-trivial transition as we vary $\beta$ over the critical value $\beta^*(\alpha,s)$. For $s \in [\alpha,d]$ the density charges any $s$-dimensional set or measure almost surely for all $\beta < \frac \alpha d$. 

In the flat case, by Theorem~\ref{thm_flatFrost}(b) and Frostman's Lemma we were able to generalize the results from (F1)-$s$ measures to general measures whose supports have positive $\cH^s$-measure (Theorem~\ref{thm_flatDim}). In the non-flat case, our most general result is Theorem~\ref{thm_nonflat}(d), which holds for all $\nu$ whose support is contained in that of an (F2)-$s$ measure and is therefore at most $s$-dimensional. It is an attractive open problem to show that this can be extended to measures supported on general compact sets of dimension at most $s$. With the exception of the critical case $\beta = \beta^*(\alpha,s)$, this would provide a complete characterization of $\N_x(\mu(X_t) > 0 )$ for compactly supported $\mu$.

\begin{remark} For the special case of $\beta = 1$, which requires $d=1$ and $\alpha \in(1,2)$ in order for the density to exist, we observe that $\beta^*(\alpha,s) < 1$ for all $s \in[0,1)$. By Theorem~\ref{thm_nonflat}, this implies that for any $\mu \in \cM_F(\R^d)$ with $\text{dim}(\text{supp}(\mu)) < 1$, the $(\alpha,1)$-superprocess in $\R^1$ fails to charge $\mu$ with positive probability
 on $\{X_t \neq 0\}$. \end{remark}

To obtain matching upper and lower bounds on $\N_x(\mu(X_t) > 0 )$ as $t \downarrow 0$ in the non-flat case requires $\mu$ to satisfy both (F1)-$s$ and (F2)-$s$ (see Theorem~\ref{thm_nonflat}(c)). This is a strong condition which is sometimes called Ahlfors-David regularity. Still, the properties (F1)-$s$ and (F2)-$s$ are satisfied by many measures which are, in an appropriate sense, uniform over some $s$-dimensional set $\cS$. First consider the case where $s$ is an integer. If $\cS$ is a rectifiable curve in $\R^d$ and $\mu$ is its length measure, then $\mu$ satisfies (F1)-$1$ and (F2)-$1$. If $\cS$ is a surface, or two-dimensional manifold, then its surface measure $\mu$ satisfies (F1)-$2$ and (F2)-$2$. (In view of the discussion following Theorem~\ref{thm_flatFrost}, this implies that for every $\beta < \frac \alpha d$, if $\mu$ is the surface or volume measure for a manifold of dimension at least two, then $\mu(X_t)>0$ a.s. when conditioned on survival). For non-integer $s$, it is known that if $\cS$ is a self-similar Cantor set of dimension $s$, then its uniform measure $\mu$ satisfies (F1)-$s$ and (F2)-$s$. (For example, see the discussion in Section 3 in \cite{MS2009} or Section 1.2 of \cite{DavidSemmes}.) Our results can be applied directly to all of the examples above.

Many random sets support measures satisfying (F1) and/or (F2). To illustrate our dichotomy for the $(\alpha,\beta)$-superprocess we consider the range of an independent fractional Brownian motion with Hurst parameter $H \in (0,1)$ in $\R^d$. We denote this process by $(B_t)_{t \geq 0}$ and let $R(B) = \linebreak \{B_t : t \in [0,1]\}$. In order to attain both sides of the dichotomy we assume that $H > d^{-1} \vee \alpha^{-1}$, which requires $d \geq 2$ and $\alpha \in (1,2)$. (As we discuss below, $\text{dim}(R(B)) = H^{-1} \wedge d$. Our assumption guarantees that $\text{dim}(R(B)) = H^{-1} < \alpha$, and we have $\beta^*(\alpha,H^{-1}) < \frac \alpha d$.) The paths of $B_t$ are $\eta$-H\"older continuous for $\eta \in (0,H)$, (for example see Proposition 1.6 of \cite{Nourdin}) from which it can be easily shown that the measure $\mu_B$ defined by $\mu_B(A) = \int_0^1 1_A(B_t) \,dt$ satisfies (F2)-$s$ almost surely (with a random constant $\underline{C}(\omega)$) for all $s < H^{-1}$. Hence if $\beta > \beta^*(\alpha,H^{-1})$, by Theorem~\ref{thm_nonflat}, $\mu_B(X_t) = 0$ with positive probability on $\{X_t \neq 0 \}$. On the other hand, $\text{dim}(R(B)) = H^{-1} \wedge d$. This is a special case of Theorem 2.1 of \cite{X1995}. Given our assumption on $H$, $\text{dim}(R(B)) = H^{-1} < \alpha$. Hence by Theorem~\ref{thm_flatDim}, (and because $\text{supp}(\mu_B) = R(B)$) if $\beta < \beta^*(\alpha,H^{-1})$ then $\mu_B(X_t) > 0$ a.s. on $\{X_t \neq 0 \}$.

Finally, we specialize the flatness and non-flatness results to the continuous case, which we recall from Theorem C is when $d=1$ and $\alpha > 1+\beta$. It is an elementary consequence of continuity of $X_t(\cdot)$ that $\mu(X_t) > 0$ if and only if $X_t(x)>0$ for some $x \in \text{supp}(\mu)$. Theorem~\ref{thm_flatDim} and \ref{thm_nonflat} then imply the following. (Part (b) is  mainly included for contrast.)

\begin{theorem} \label{thm_cont} Let $d=1$, $\alpha > 1+\beta$ and $s \in [0,d]$.

\noindent (a) Suppose $\beta \in ( \beta^*(\alpha,s), \frac \alpha d)$ (which implies $s < \alpha$) and let $\cS = \text{supp}(\mu)$ for a compactly supported measure $\mu \in \cM_F(\R^d)$ satisfying (F2)-$s$. Then $X_t(x) = 0$ for all $x \in \cS$ with positive probability on $\{X_t \neq 0 \}$ under $P^X_{X_0}$ and $\N_0$. Moreover, \eqref{e_thmnonflat_upperbound} holds when $\N_x(\mu(X_t) > 0 )$ is replaced with \linebreak$\N_x(\{X_t(x) > 0 \text{ for some } x \in \cS\})$, and if $\mu$ satisfies (F1)-$s$, then so does \eqref{e_thmnonflat_lowerbound}.

\noindent (b) Suppose $\beta \leq \beta^*(\alpha,s)$ and let $\cS \subset \R^d$ be Borel with $\cH^s(\cS) > 0$. Then almost surely on $\{X_t \neq 0 \}$ there is a point $x\in \cS$ such that $X_t(x) >0$.
\end{theorem}

Thus far we have generally viewed $\alpha$ as fixed and $\beta$ as a variable parameter, with the critical $\beta$ (i.e. $\beta^*(\alpha,s)$) depending on $s$, the dimension of $\cS$ or $\text{supp}(\mu)$. We would also like to draw attention to another perspective. We define the \textit{saturation dimension} associated to the parameters $(\alpha,\beta,d)$.
\begin{definition}\label{def_dsat}
For fixed parameters $(\alpha,\beta,d)$, the \textit{saturation dimension} $d_{\text{sat}}$ is
\begin{equation}
d_{\text{sat}} = d_{\text{sat}}(\alpha,\beta,d) = d + \alpha - \frac{\alpha}{\beta}.
\end{equation}
\end{definition}
Thus $d_{\text{sat}}$ is simply the value of $s$ for which $\beta = \beta^*(\alpha,s)$, and therefore it is the critical dimension of the set or measure as pertains to this problem. It is the maximum dimension of a set which the density can fail to charge with positive probability: if  $\cH^{d_{\text{sat}}}(\cS) >0$, implying $\text{dim}(\cS) \geq d_{\text{sat}}$, then the behaviour of $X_t$ is trivial on $\cS$ the sense that $\mu(X_t) >0$ a.s. on $\{X_t \neq 0 \}$ for every $\mu \in \cM_F(\R^d)$ with $\text{supp}(\mu) = \cS$. The fact that $d_{\text{sat}} \leq d$ corresponds exactly to instantaneous propagation, i.e. Theorem A. On the other hand, strong instantaneous propagation (i.e. the conclusions of Theorem~\ref{thm_pointdichotomy}(a)) corresponds to parameters for which $d_{\text{sat}} = 0$ (where if $d_{\text{sat}} < 0$ we simply define it to be $0$). A summary of some of the results above in these terms is as follows:

\begin{itemize}
\item \textit{If $\mu \in \cM_F(\R^d)$ with $\cH^{d_{\text{sat}}}(\text{supp}(\mu)) >0$, then $X_t(\cdot)$ charges $\mu$ almost surely on $\{X_t \neq 0 \}$.
\item If $\cS = \text{supp}(\mu)$ for some $\mu \in \cM_F(\R^d)$ which satisfies (F2)-$s$ with $s < d_{\text{sat}}$ (implying \linebreak$\text{dim}(\cS) < d_{\text{sat}}$), then for every $\nu \in \cM_F(\cS)$, $\nu(X_t) = 0$ with positive probability on $\{X_t \neq 0 \}$.}
\end{itemize}

We now turn our attention to the PDE \eqref{e_pde} and state some new results concerning the initial trace theory for positive solutions to this equation. We emphasize that all of our results apply only when $\beta < \frac \alpha d$, which in the PDE literature is sometimes called the \textit{subcritical absorption} regime. Following \cite{CV2019}, we define the initial trace. For $A \subseteq \R^d$, let $C_c(A)$ denote the space of continuous, compactly supported functions on $A$. A positive solution $u$ to \eqref{e_pde} has \textit{initial trace} $(\cS,\nu)$,  where $\cS \subset \R^d$ is closed and $\nu$ is a Radon measure satisfying $\nu(\cS) = 0$, if
\begin{itemize} 
\item For all $\xi \in  C_c(\cS^c)$,	
\begin{equation} \label{e_inittrace_measure}
\lim_{t \to 0} \int \xi(x) u_t(x) dx = \int \xi d\nu.
\end{equation}
\item For every $z \in \cS$ and $\rho > 0$,
\begin{equation} \label{e_inittrace_singular}
\lim_{t \to 0} \int_{B(z,\rho)} u_t(x) dx = + \infty.
\end{equation}
\end{itemize}
The set $\cS$ is called the \textit{singular set} of $u$, whereas $\cS^c$ is called the \textit{regular set}. Theorems A and B of \cite{CV2019} give general conditions under which a positive solution to \eqref{e_pde} can be associated to an initial trace $(\cS,\nu)$. Our results concern the converse problem of determining if there exists a solution with a given initial trace. 

The non-fractional analogue of \eqref{e_pde} is 
\begin{equation} \label{e_nonfracpde}
\partial_t u = \Delta u -  u^p.
\end{equation} 
The initial trace theory for \eqref{e_nonfracpde} is well understood. For an analytic approach which covers a wide range of parameters $p$ (including $p = 1+\beta$ for $\beta \in (0,1]$), see Marcus and V\'eron \cite{MV1999}. Because of the dual relationship with super-Brownian motion, the problem is also amenable to probabilistic analysis. For the case $p=2$, Le Gall \cite{LG1996} characterized the positive solutions to \eqref{e_nonfracpde} using the Brownian snake.\footnote{Probabilistic approaches using superprocesses have been very effective in studying the elliptic equation related to \eqref{e_nonfracpde}, $\Delta u = u^p$ with $p \in (1,2]$. The $p=2$ case has been studied by Le Gall in \cite{LG1995, LG1997}, again using the Brownian snake. The work of Dynkin and Kuznetsov (e.g. \cite{DK1996, DK1998}) gives results covering $p \in (1,2]$, which is the entire range for which the superprocess approach applies.}

The theory for the fractional equation \eqref{e_pde} is more recent and is far from complete. Theorem~B hints at this, as it shows that existence of a solution with initial trace $(\{0\},0)$ (sometimes called a \textit{very singular solution}) depends on the values of the parameters. A one-point singular set is the smallest non-trivial singular set, and so a consequence of the theorem is that when $\beta \leq \beta^*(\alpha)$, the singular initial trace problem is trivial in the sense that if $\cS$ is non-empty, the solution equals or exceeds $U_t$ (see also Theorem F of \cite{CV2019}). This issue does not arise for solutions to \eqref{e_nonfracpde} with $p \in (1,2]$ because the associated superprocess has compact support which is localized near its initial conditions, and hence the event that the range of the superprocess does not intersect a given closed set (or that the superprocess does not charge a measure supported on that set in an appropriate sense) always has positive probability.



Our results about flatness and non-flatness of $u^{\infty \mu}_t$ allow us to advance the theory of existence for solutions to \eqref{e_pde} with a given initial trace. At this stage we are largely concerned with existence and we do not consider the regularity of solutions. For this reason we study weak solutions. A precise definition of a weak solution to \eqref{e_pde} with initial trace $(\cS,\nu)$ is given in Section~\ref{s_trace}. In that section we also restrict our attention to solutions $u(t,x)$ which are bounded above by $U_t$, i.e. $u(t,x) \leq U_t$. We define $\cU = \cU(Q)$ by
\begin{equation} \label{def_calU}
 \cU = \{u: Q \to [0,\infty) : u_t \leq U_t \,\,\, \forall \, t>0\}.
\end{equation}
By duality, in particular \eqref{e_canon_survive}, \eqref{e_Lapcanon} and \eqref{e_Lapcanonmeasure}, $\cU$ includes all solutions which admit a probabilistic representation in terms of the $(\alpha,\beta)$-superprocess, which includes all solutions obtained as the limit of solutions with initial data in $\cM_F(\R^d)$ or $\cB^+_b$. See also Theorem D of \cite{CV2019}, which proves (for classical solutions) that solutions satisfying a mild integrability property are bounded by $U_t$ and hence belong to $\cU$. It remains unresolved if there exist positive solutions to \eqref{e_pde} which do not belong to $\cU$; because they cannot be obtained as the limits of solutions with truncated function-valued initial conditions, it is unclear how one might construct one.

\begin{theorem} \label{thm_inittrace} (a) \emph{(Non-existence)} Let $\cS \subset \R^d$ be closed with $\cH^{d_{\text{sat}}}(\cS) > 0$. If the singular set of a solution $u$ to \eqref{e_pde} in \hspace{1 mm}$\cU$ contains $\cS$, then $u_t = U_t$. 
In particular, if $\cS \neq \R^d$ there is no solution to \eqref{e_pde} in \hspace{1 mm}$\cU$ with initial trace $(\cS,\nu)$ for any Radon measure $\nu$.

\noindent (b) \emph{(Existence)} Suppose that $\cS \subset \R^d$ is compact and there exists $\mu  \in \cM_F(\R^d)$ satisfying (F2)-$s$ with $s < d_{\text{sat}}$ such that $\cS = \text{supp}(\mu)$. Then there exists a weak solution to \eqref{e_pde} with initial trace $(\cS,0)$. The solution is given by $\lim_{\lambda \to \infty} u^{\lambda \mu}_t(x) = \N_x(\mu(X_t) > 0)$ and satisfies \eqref{e_thmnonflat_upperbound}. 
\end{theorem}

The problem of existence for solutions with initial trace $(\cS,0)$ is therefore characterized by the saturation dimension. This generalizes Theorem~B(a) (and Theorem F of \cite{CV2019}), which can be viewed as the special case of the above when $d_{\text{sat}} = 0$. When $d_{\text{sat}} = 0$ there is no transition: all non-empty singular sets fall into the non-existence regime.\\



\noindent \textbf{Organization of the paper.} The rest of the paper is organized as follows. Section~\ref{s_prelim} provides background information and preliminary results; the key subsections are Section~\ref{s_density}, which covers the density of $X_t$ and $\mu(X_t)$, and Section~\ref{s_duality}, which discusses solutions to \eqref{e_pde} with measure-valued initial data and extends the dual relationship to include finite measures. In Section~\ref{s_points} we give a probabilistic proof of Theorem~\ref{thm_pointdichotomy}(a), and in Section~\ref{s_pos} we prove Theorem~\ref{thm_positiveEverwhere}. Sections~\ref{s_flat} and \ref{s_nonflat} cover, respectively, flatness and non-flatness of $u^{\infty \mu}_t$ for general measures $\mu$, in particular Theorems~\ref{thm_flatFrost} and \ref{thm_nonflat}. Finally, we discuss solutions to \eqref{e_pde} with non-empty singular sets and prove Theorem~\ref{thm_inittrace} in Section~\ref{s_trace}.

\section{Preliminaries} \label{s_prelim}

\subsection{Transition densities}  \label{s_transitiondensities}
We denote by $p_t(x)$ the fundamental solution (or heat kernel) to the fractional heat equation on $\R^d$. That is, $p_t(x)$ is the solution to 
\begin{equation} 
\partial_t u =  \Delta_\alpha u, \,\,\,\, u_0 = \delta_0 \nonumber
\end{equation}
on $Q= (0,\infty) \times \R^d$. The semigroup $(S_t)_{t \geq 0}$ which we have already introduced is a convolution semigroup with kernel $p_t$, i.e.
\[S_t\phi(x) = \phi * p_t(x) = \int \phi(y) p_t(x-y) dy.\]
We note that for $\mu \in \cM_F(\R^d)$, $S_t \mu (x)$ can be defined in the same way with no difficulty. The kernel $p_t$ is radial and radially decreasing. In a slight abuse of notation, for $\rho>0$ we will sometimes write $p_t(\rho)$ to mean $p_t(x)$, where $|x| = \rho$.

$(S_t)_{t \geq 0}$ is the transition semigroup of the symmetric $\alpha$-stable process, so we have the following: if $W$ is a symmetric $\alpha$-stable process with law and expectation $P^W_x$ and $E^W_x$, respectively, when started at $x$, then for an appropriate class of functions (e.g. bounded and measurable),
\[ E^W_x( \phi(W_t)) = S_t \phi(x).\]
In particular, $p_t$ is also the transition density of $W$, that is
\[P^W_x(W_t \in dy) = p_t(y-x) dy.\]
The symmetric $\alpha$-stable process is self-similar, which is reflected by the scaling property for $p_t$:
\begin{equation} 
p_t(x) = t^{-\frac d \alpha} p_1(t^{-\frac 1 \alpha}x).\nonumber
\end{equation}
Finally we recall the asymptotic decay of the transition density. There are universal constants $0<c_{\ref{heatkernelbds}}<C_{\ref{heatkernelbds}}<\infty$ such that
\begin{equation} \label{heatkernelbds}
c_{\ref{heatkernelbds}}\, \frac{t^{-\frac d \alpha} }{1+|t^{-1/\alpha}x|^{d+\alpha}}\leq p_t(x) \leq C_{\ref{heatkernelbds}} \,\frac{ t^{-\frac d \alpha}}{1+|t^{-1/\alpha} x|^{d+\alpha}}.
\end{equation}
It is occasionally useful to write the above as
\begin{equation} \label{heatkernelbds2}
c_{\ref{heatkernelbds}} \left( t^{-\frac d \alpha} \wedge\frac{t}{|x|^{d+\alpha}} \right) \leq p_t(x) \leq C_{\ref{heatkernelbds}}\left( t^{-\frac d \alpha} \wedge\frac{t}{|x|^{d+\alpha}} \right),
\end{equation}
where in order to do so one may have to adjust the constants. Without loss of generality we will fix $c_{\ref{heatkernelbds}}$ and $C_{\ref{heatkernelbds}}$ so that both bounds hold.

\subsection{The density of the $(\alpha,\beta)$-superprocess} \label{s_density}
In this section we give an overview of the density of $X_t$. The main purpose is to show that we can take a version of the density which is defined almost surely at any fixed point $x \in \R^d$ and that with this version the quantity $\mu(X_t)$ from \eqref{e_muXt} is well-defined almost surely. We also discuss the regularity properties of the density in the continuous regime.

Recall from \eqref{e_Xepsilon} and \eqref{e_density_liminf} that for $x \in \R^d$ and $\epsilon > 0$,
\[X^\epsilon_t(x) = \frac{X_t(B(x,\epsilon))}{|B(x,\epsilon)|}\]
and 
\[ X_t(x) := \liminf_{\epsilon \downarrow 0} X^\epsilon_t(x).\]
Since $X_t$ is absolutely continuous, $x \to X_t(x)$ is a density for it by the Lebesgue differentiation theorem (for example see Theorem~3.21 of \cite{Folland}). Because we are interested in the behaviour of the density at fixed points and more generally on Lebesgue null sets, we require more than the standard a.e.-convergence of $X^\epsilon_t(\cdot)$ to $X_t(\cdot)$. The desired conditions are shown to hold in the next lemma. 

For $\psi : \R^d \to \R$ and $\epsilon > 0$, let $\psi_\epsilon(x) := \epsilon^{-d} \psi(\epsilon^{-1}x)$.

\begin{lemma} \label{lemma_density_convergence} Let $X_0 \in \cM_F(\R^d)$ and $t>0$ and consider $X_t$ under $P^X_{X_0}$.

\noindent(a) For every $x \in \R^d$, $X^\epsilon_t(x) \to X_t(x)$ a.s. and in $L^1$ as $\epsilon \downarrow 0$. 

\noindent(b) For every $\mu \in \cM_F(\R^d)$, $X^\epsilon_t(x) \to X_t(x)$ for $\mu$-a.e. $x$ almost surely and $\mu(X^\epsilon_t) \to \mu(X_t)$ in $L^1$, where $\mu(X_t) = \int X_t(x) \mu(dx)$. Moreover, we have
\begin{equation} \label{e_meanmeasure}
 E_{X_0}^X(\mu(X_t)) = \mu(S_t X_0).
\end{equation}

\noindent(c) If $\psi:\R^d \to [0,\infty)$ satisfies $\int \psi = 1$ and $\sup_x |\psi(x)|(1+|x|)^{d+\delta} < \infty$ for some $\delta > 0$, then the conclusions of part (a) and (b) hold when $X^\epsilon_t$ is replaced with $X_t * \psi_\epsilon$.
\end{lemma}

We prove this lemma at the end of the section. Part (c) is included because it is sometimes useful to approximate $X_t(x)$ using convolutions with general kernels. A particularly useful example is when $\psi = p_1$, in which case $\psi_{\epsilon^{\frac 1\alpha}} = p_\epsilon$.

We now turn our attention to the H\"older regularity of the density in the continuous regime, i.e. $d = 1$ and $\alpha > 1+\beta$ (c.f. Theorem C). The following was proved in \cite{FMW2010}.\\

\noindent \textbf{Theorem D.} [Fleischmann, Mytnik and Wachtel (2010)] \emph{Let $d=1$ and $\alpha > 1+\beta$. Let $X_0 \in \cM_F(\R^d)$ and $t>0$. Under $P^X_{X_0}$,  there is a continuous version $X_t(\cdot)$ of the density such that for every $\eta < \eta_c =  \frac{\alpha}{1+\beta} - 1$, $X_t(\cdot)$ is locally H\"older of index $\eta$, i.e.
\[\sup_{x_1,x_2 \in K ,x_1 \neq x_2} \frac{|X_t(x_1) - X_t(x_2)|}{|x_1 - x_2|^\eta} < \infty \,\,\,\, \text{ for all compact $K \subset \R$.} \] Furthermore, the value $\eta_c$ is optimal in that, for any $\eta > \eta_c$, with probability one, for any open $U \subseteq \R$, 
\[\sup_{x_1,x_2 \in U ,x_1 \neq x_2} \frac{|X_t(x_1) - X_t(x_2)|}{|x_1 - x_2|^\eta} = \infty \,\,\,\, \text{ whenever $X_t \neq 0$.} \]}\\

As we have noted in the introduction, when $d=1$ and $\alpha > 1+\beta$, the density we define in \eqref{e_density_liminf} is the same as the continuous version. The statement in the above implies the following. Let $\eta < \frac{\alpha}{1+\beta} - 1$ and $X_0 \in \cM_F(\R)$. Then
\begin{align} \label{e_localmodulus}
&P^X_{X_0}\text{-a.s., for all compact $K \subset \R$, there exists $C(K,\eta,\omega) >0$ such that for all $x,y\in K$,} \nonumber
\\& \hspace{30 mm} |X_t(x) - X_t(y)| \leq C(K,\eta,\omega) |x -y|^\eta .
\end{align}

We conclude the section with the proof of Lemma~\ref{lemma_density_convergence}. The proof uses a classical result concerning absolute continuity of the laws of superprocesses. As with Theorem A, the result originates in \cite{EP1991}, where it was proved for binary branching superprocesses. The proof for $(1+\beta)$-stable branching superprocesses appears in \cite{LZ2008}. We do not state the result in full generality, but refer the reader to Theorems 1.1 and 2.2, respectively, of \cite{EP1991} and \cite{LZ2008}. For any $X_0, \hat{X}_0 \in \cM_F(\R^d)$ and $t>0$, we have
\begin{align} \label{mutual_absolute_continuity}
\text{The laws of $X_t$ under $P^X_{X_0}$ and $P^X_{\hat{X}_0}$ are mutually absolutely continuous.}
\end{align}
We use the above when $\hat{X}_0$ is a translation of $X_0$. In particular, for $x \in \R^d$ suppose that $\hat{X}_0 = X_0+x$, where for a measure $\mu$ we define $\mu + x$ by $( \mu+x)(A) = \int 1_A(y-x)\mu(dy)$. Applying \eqref{mutual_absolute_continuity}, the laws of $X_t$ under $P^X_{X_0}$ and $P^X_{X_0+x}$ are mutually absolutely continuous. On the other hand, by translation invariance of the superprocess, the law of $X_t$ under $P^X_{X_0+x}$ is equal to the law of $ X_t - x$ under $P^X_{X_0}$. Consequently, we have that for $x \in \R^d$ and $t>0$,
\begin{align} \label{mutual_absolute_continuity_translation}
\text{The laws of $X_t$ and $X_t + x$ are mutually absolutely continuous under $P^X_{X_0}$.}
\end{align}

We will also use the following moment bound.
\begin{lemma} \label{lemma_Lp_bound}Let $0<\theta < \beta$ and $\phi \geq 0$. Then
\begin{equation}
E^X_{X_0}(X_t(\phi)^{1+\theta}) \leq 1 + C \left[ \int_0^t X_0(S_{t-s}((S_s \phi)^{1+\beta})) ds + X_0((S_t\phi)^{1+\beta}) \right] \nonumber
\end{equation}
for a constant $C = C(\alpha,d,\theta) > 0$.
\end{lemma}

The $\alpha = 2$ case of this result appears as Lemma 2.1 of \cite{MP2003}. The same argument used in \cite{MP2003} also works in the $\alpha \in (0,2)$ when one replaces the heat semigroup with the fractional heat semigroup, so we omit the proof and continue with the proof of Lemma~\ref{lemma_density_convergence}. The arguments used to bound moments below are also borrowed from \cite{MP2003}.

\begin{proof}[Proof of Lemma~\ref{lemma_density_convergence}] Fix $t>0$ and $X_0 \in \cM_F(\R^d)$. Let $\psi:\R^d\to [0,\infty)$ satisfy $\int \psi = 1$ and $\sup_x |\psi(x)|(1+|x|)^{d+\delta} < \infty$ for some $\delta >0$ and recall that $\psi_\epsilon(x) = \epsilon^{-d}\psi(\epsilon^{-1}x)$. These conditions include the case where $\psi$ is the normalized indicator function of the unit ball. A priori this is the case we are interested in, but we carry out the analysis for the general case as this is required for part (c).

Let $X^{\epsilon,\psi}_t(x) = X_t * \psi_\epsilon(x)$. Since $X_t$ is absolutely continuous almost surely, we have
\begin{equation} \label{e_ae_convergence1}
\text{$P^X_{X_0}$-almost surely, $X_t^{\epsilon,\psi}(x)$ converges for Lebesgue-a.e. $x\in\R^d$ as $\epsilon \downarrow 0$.}
\end{equation}
If $\psi$ is the normalized indicator function of the unit ball, the above is simply Lebesgue's differentation theorem. Otherwise, the result follows from Theorem 8.15 of \cite{Folland}.

We now show that the convergence holds pointwise almost surely. Suppose there exists $x_0 \in \R^d$ such that $P^X_{X_0}(\{X_t^{\epsilon,\psi}(x_0) \text{ does not converge as }  \epsilon \downarrow 0\}) > 0$. By \eqref{mutual_absolute_continuity_translation} this implies that \linebreak$P^X_{X_0}(\{X_t^{\epsilon,\psi}(x) \text{ does not converge as } \epsilon \downarrow 0\}) > 0$ for all $x \in \R^d$. However, by Fubini's theorem this implies that the set of points for which $X_t^{\epsilon,\psi}(x)$ does not converge has positive Lebesgue measure a.s., which contradicts \eqref{e_ae_convergence1}. Hence we must have that
\begin{equation} \label{e_ae_convergence2}
\text{$X_t^{\epsilon,\psi}(x)$ converges $P^X_{X_0}$-a.s. as $\epsilon \downarrow 0$ for all $x\in \R^d$. }
\end{equation}
Thus we have shown the almost sure convergence stated in part (a). Before proving the $L^1$ convergence and part (b), we show how absolute continuity also implies the a.s. convergence in part (c). When $\psi$ is the normalized indicator function of a unit ball, we denote $X^{\epsilon,\psi}_t(x)$ by $X^\epsilon_t(x)$. Let $\psi$ be any other function satisfying the conditions as above. If there exists $x_0 \in \R^d$ such that $P^X_{X_0}(\liminf_{\epsilon \downarrow 0} X_t^\epsilon(x_0) \neq \liminf_{\epsilon \downarrow 0} X_t^{\epsilon,\psi}(x_0))>0$, then \eqref{mutual_absolute_continuity_translation} implies that the probability is positive when we replace $x_0$ with $x$ for all $x \in \R^d$. By Fubini's theorem, $\liminf_{\epsilon \downarrow 0} X_t^\epsilon(x) \neq \liminf_{\epsilon \downarrow 0} X_t^{\epsilon,\psi}(x)$ on a set of positive Lebesgue measure almost surely. Since $\liminf_{\epsilon \downarrow 0} X_t^\epsilon(\cdot)$ and $\liminf_{\epsilon \downarrow 0} X_t^{\epsilon,\psi}(\cdot)$ are both densities for $X_t$, this is a contradiction. Hence $\liminf_{\epsilon \downarrow 0} X_t^\epsilon(x) = \liminf_{\epsilon \downarrow 0} X_t^{\epsilon,\psi}(x)$ a.s. for every $x \in \R^d$. The left hand side is simply equal to $X_t(x)$, and so for all $\psi$ satisfying the stated conditions, \eqref{e_ae_convergence2} can be improved to
\begin{equation}
\text{$X_t^{\epsilon,\psi}(x) \to X_t(x)$ $P^X_{X_0}$-a.s. as $\epsilon \downarrow 0$ for all $x\in \R^d$, } \nonumber
\end{equation}
which proves that the almost sure convergence in (a) holds for all $\psi$ from (c). Furthermore, the above implies that the limit of $X^{\epsilon,\psi}_t(x)$ does not depend on the choice of $\psi$. Given this, we can work with general $\psi$ and prove the remaining claims of (a) and (b) simultaneously with (c).

In order to establish that $X_t^{\epsilon,\psi}(x) \to X_t(x)$ in $L^1$, we will show $L^p$-boundedness of the quantity $X_t^{\epsilon,\psi}(x)$ for $p = 1 + \theta$ with $\theta  <\beta$. To do so we apply Lemma~\ref{lemma_Lp_bound}. Let $\psi_\epsilon^x(\cdot) = \psi_\epsilon(\cdot - x)$ and note that $X_t^{\epsilon,\psi}(x) = X_t(\psi_\epsilon^x)$. For $\theta < \beta$, we have
\begin{align} \label{e_Lpbd1}
E^X_{X_0}((X_t^{\epsilon,\psi}(x))^{1+\theta}) \leq 1 + C \left[ \int_0^t X_0(S_{t-s}((S_s \psi_\epsilon^x)^{1+\beta})) ds + X_0((S_t \psi^x_{\epsilon})^{1+\beta}) \right]. 
\end{align}
Since $\int \psi^x_{\epsilon} = 1$, we can view $\psi_\epsilon^x$ as the density of a random variable, which we will denote $Y$. Thus for $s>0$,
\begin{align} \label{e_Lpbd11}
(S_s \psi_\epsilon^x)(z)^{1+\beta} = \left[\int p_s(y-z) \psi_\epsilon^x(y) dy \right]^{1+\beta} &= E^Y(p_s(Y-z))^{1+\beta} \nonumber \\ 
&\leq E^Y(p_s(Y-z)^{1+\beta})\nonumber \\
& = \int p_s(y-z)^{1+\beta} \psi_\epsilon^x(y) dy,
\end{align}
We also have
\[p_s(y-z)^{1+\beta} = p_s(y-z) p_s(y-z)^\beta \leq C p_s(y-z) s^{-\frac{d\beta}{\alpha}}\]
by \eqref{heatkernelbds2}. Hence by \eqref{e_Lpbd11} and the above,
\begin{align} \label{e_Lpbd2}
X_0(S_{t-s}((S_s \psi_\epsilon^x)^{1+\beta})) &\leq Cs^{-\frac{d\beta}{\alpha}} \iiint p_{t-s}(z-w) p_s(y-z) \psi_\epsilon^x(y)\, dy\, dz\, X_0(dw) \nonumber
\\ &= Cs^{-\frac{d\beta}{\alpha}} \iint p_t(w-y) \psi_\epsilon^x(y) \, dy \, X_0(dw) \nonumber
\\ &\leq C X_0(1) s^{-\frac{d\beta}{\alpha}} t^{-\frac d \alpha}. 
\end{align}
Since $\beta < \frac \alpha d$, the power of $s$ is integrable over $s \in [0,t]$, which implies that the first term in the square brackets in \eqref{e_Lpbd1} is uniformly bounded in $\epsilon \in (0,1]$. The second term can be handled the same way and is in fact easier to control. Thus $E^X_{X_0}((X_t^{\epsilon,\psi})^{1+\theta})$ is uniformly bounded for $\epsilon \in (0,1]$, which implies that the family $\{X_t^{\epsilon,\psi}(x) : \epsilon \in (0,1]\}$ is uniformly integrable. Since $X_t^{\epsilon,\psi}(x)$ converges almost surely to $X_t(x)$, this implies that the convergence is also in $L^1(P^X_{X_0})$.

We now fix $\mu \in \cM_F(\R^d)$ and show $L^1$ convergence of $\mu(X_t^{\epsilon,\psi})$ to $\mu(X_t)$. Note that the bound in \eqref{e_Lpbd2} does not depend on $x$. Thus, rather than consider $L^{1+\theta}$-boundedness of $X_t(\psi^x_{\epsilon})$ for fixed $x$ with respect to $P^X_{X_0}$, we can fix $\mu \in \cM_F(\R^d)$ and consider $L^{1+\theta}$-boundedness of $(\omega, x) \to X_t(\psi^x_{\epsilon})(\omega)$ with respect to $P^X_{X_0} \otimes \mu$. Because the bounds are uniform in $x$ and $\mu$ is finite, the argument requires no modification. Since $X_t(\psi^x_\epsilon) \to X_t(x)$ a.s. for every $x$, it follows that $X_t(\psi^x_\epsilon) \to X_t(x)$ in $L^1(P^X_{X_0} \otimes \mu)$ (as a function of $(\omega,x)$). This implies that $\mu(X^{\epsilon,\psi}_t) \to \mu(X_t)$ in $L^1$.

It remains to show the moment formula \eqref{e_meanmeasure}. First, let $\phi : \R^d \to [0,\infty)$ be bounded and measurable. The mean measure formula for superprocesses gives
\begin{equation} \label{e_mm1}
 E^X_{X_0}(X_t(\phi)) = X_0(S_t \phi).
\end{equation}
Now fix $\mu \in \cM_F(\R^d)$. We have
\begin{align}
\mu(X^{\epsilon,\psi}_t) = \iint \psi_\epsilon(x-y) X_t(dy)  \mu(dx) = \iint \mu_\epsilon(x) X_t(dx) = X_t(\mu_\epsilon), \nonumber
\end{align}
where $\mu_\epsilon = \mu * \hat{\psi}_\epsilon$ with $\hat{\psi}_\epsilon(\cdot) = \psi_\epsilon(-\cdot)$. Since $\mu_\epsilon$ is a bounded and measurable function, by \eqref{e_mm1} the above implies
\begin{equation}
E^X_{X_0}(\mu(X^{\epsilon,\psi}_t)) = X_0(S_t \mu_\epsilon). \nonumber
\end{equation}
Because $\mu(X^{\epsilon,\psi}_t) \to \mu(X_t)$ in $L^1(P^X_{X_0})$, the left hand side converges to $E^X_{X_0}(\mu(X_t))$ as $\epsilon \downarrow 0$. We have $S_t \mu_\epsilon \leq \mu(1) t^{-\frac d \alpha}$ for all $\epsilon > 0$, so the right hand side converges to $X_0(S_t \mu)$ by Dominated Convergence. Since $X_0(S_t\mu) = \mu(S_t X_0)$, this proves \eqref{e_meanmeasure} and the proof is complete.
\end{proof}

\subsection{The fractional PDE and $\mu(X_t)$}  \label{s_duality}
In this section we extend the duality between $X_t$ and solutions to the evolution equation \eqref{e_evol}. Recall from \eqref{e_Lap} that in its basic form, the dual relationship states that $E^X_{X_0}(\exp(-X_t(\phi)) = \exp(-X_0(u^\phi_t))$, where for bounded and measurable $\phi \geq 0$, $u^\phi_t$ is the unique solution to \eqref{e_evol} with initial data $\phi$. The purpose of this section is to extend this relationship to allow $\phi$ to be replaced with a finite measure when $\beta < \frac \alpha d$. We also introduce weak solutions to \eqref{e_pde} and some of their properties.

The integral equation \eqref{e_evol} is a mild form of the PDE \eqref{e_pde}. We will work with weak solutions. Recall that $Q = (0,\infty) \times \R^d$. Let $C^{1,2}_c(Q)$ denote the space of compactly supported functions on $Q$ which are once and twice continuously differentiable in time and space, respectively. For $T>0$, let $Q_T = (0,T] \times \R^d$ and $\bar{Q}_T = [0,T]\times \R^d$, and let $C^{1,2}_b(\bar{Q}_T)$ denote the space of bounded functions on $\bar{Q}_T$ with bounded, continuous derivatives up to order one in time and order two in space. For $p \geq 1$, we let $L^p_{\text{loc}}(Q)$ denote the space of functions $\phi$ such that $\int_K |\phi|^p < \infty$ for every compact $K \subset Q$.

\begin{definition} \label{def_weaksol} A function $u: Q \to \R$ is a weak solution to \eqref{e_pde} if $(t,x) \to u(t,x)$ is continuous, $u \in L^{1+\beta}_{\text{loc}}(Q)$, and 
\begin{equation} \label{e_weaksolibp_compact}
\int_{Q} (u(t,x)[-\partial_t \xi(t,x) - \Delta_\alpha \xi(t,x)]) +  u(t,x)^{1+\beta} \xi(t,x) \,dx  dt = 0 
\end{equation}
for all $\xi \in C^{1,2}_c(Q)$.
\end{definition}
For measure-valued initial data, the PDE problem of interest is
\begin{equation} \label{e_pde_measureproblem}
\begin{cases} &\partial_t u = \Delta_\alpha u -  u^{1+\beta} \hspace{4 mm} \text{ for $(t,x) \in Q$},
\\ & u_0 = \mu \end{cases}
\end{equation}
for $\mu \in \cM_F(\R^d)$. In the following, recall that by convergence in $\cM_F(\R^d)$ we mean weak convergence of measures.
\begin{definition} \label{def_weaksol_IVP}
For $\mu \in \cM_F(\R^d)$, we say that $u : Q \to \R^+$ is a weak solution to \eqref{e_pde_measureproblem} if it is a weak solution to \eqref{e_pde}, $u \in L^1(Q_T) \cap L^{1+\beta}(Q_T)$ for all $T>0$, and $u_t \to \mu$ in $\cM_F(\R^d)$ as $t \downarrow 0$.
\end{definition}

\begin{proposition} \label{prop_weaksolexist} For $\mu \in \cM_F(\R^d)$, there exists a unique weak solution $u:Q \to [0,\infty)$ to \eqref{e_pde_measureproblem}. Moreover, for $T>0$ we have
\begin{align} \label{weaksol_ibp}
&\int_{Q_T} (u(t,x)[-\partial_t \xi(t,x) - \Delta_\alpha \xi(t,x)]) + u(t,x)^{1+\beta} \xi(t,x) \,dx  dt \nonumber
\\ &\hspace{5 mm}= \int_{\R^d} \xi(0,x) \mu(dx)  - \int_{\R^d} u(T,x) \xi(T,x) dx
\end{align}
for all $\xi \in C^{1,2}_b(\bar{Q}_T)$.
\end{proposition}
We will denote the unique solution to \eqref{e_pde_measureproblem} by $u^\mu_t(x)$ or $u^\mu(t,x)$. The above is proved in Theorem 1.1 of \cite{CVW2016}. The authors of that work use a slightly different definition which incorporates \eqref{weaksol_ibp}. However, a short argument which we omit shows that given a solution in the sense of Definition~\ref{def_weaksol_IVP}, \eqref{weaksol_ibp} holds for all $\xi \in C^{1,2}_b(\bar{Q}_T)$. The definition of a solution used in \cite{CVW2016} also does not require continuity, but it can be verified that $u^\mu_t(x)$ is jointly continuous, for example using its correspondence with the solution to the integral equation given in Lemma~\ref{lemma_integ_pde}(b) below. (The proof of Lemma~\ref{lemma_integ_pde}(b) requires only a small modification if one does not assume a priori that $u^\mu_t(x)$ is continuous.) The following stability result holds as a consequence of Theorem 1.1 of \cite{CVW2016}.

\begin{lemma} \label{lemma_stability} (a) If $\mu \in \cM_F(\R^d)$ and $(\mu_n)_{n \geq 1}$ is a sequence of measures such that $\mu_n \to \mu$ in $\cM_F(\R^d)$, then $u^{\mu_n}_t(x) \to u^\mu_t(x)$ locally uniformly in $Q$.

\noindent (b) The map $\mu \to u^\mu_t$ is increasing.
\end{lemma}

Solutions of \eqref{e_pde_measureproblem} are bounded above by solutions of the homogeneous fractional heat equation with the same initial data. Again by Theorem 1.1 of \cite{CVW2016}, we have
\begin{equation} \label{e_heatcomparison}
u^\mu_t(x) \leq S_t \mu(x) \text{    for $(t,x) \in Q$,}
\end{equation}
where $(t,x) \to S_t \mu(x)$ is the unique solution to $\partial_t v = \Delta_\alpha v$ on $Q$ with $v_0 = \mu$ (see Theorems 3.1 and 5.1 of \cite{BSV2017}). 

\begin{lemma}\label{lemma_integ_pde} Let $\mu \in \cM_F(\R^d)$. (a) The integral equation
\begin{equation}\label{e_integ_measure}
u_t(x) = S_t \mu (x) - \int_0^t S_{t-s}( u_s^{1+\beta} )(x) \, ds, \,\,\,\,\, (t,x) \in Q
\end{equation}
has a unique, non-negative, jointly continuous solution in $L^1(Q_T) \cap L^{1+\beta}(Q_T)$ for all $T>0$. 

\noindent (b) The weak solution to \eqref{e_pde_measureproblem} with initial data $\mu$ and the solution to \eqref{e_integ_measure} are equal. \end{lemma}

We will therefore use the notation $u^{\mu}_t(x)$ and $u^\mu(t,x)$ to refer to the unique solution to \eqref{e_pde_measureproblem} and \eqref{e_integ_measure}. With the exception of continuity, part (a) of the above is proved in Lemma A.2 of \cite{F1988}, and continuity can be shown by a direct calculation which we omit.

\begin{proof}[Proof of Lemma~\ref{lemma_integ_pde}(b)]
Let $\mu \in \cM_F(\R^d)$ and let $u(t,x) = u^\mu(t,x)$, the unique weak solution to \eqref{e_pde_measureproblem}. For $T>0$, $x_0 \in \R^d$ and $n \in \N$, we take $\xi(t,x) = p_{T+n^{-1}-t}(x-x_0) \in C^{1,2}_b(Q)$ in \eqref{weaksol_ibp}. Since $p_t$ solves $\partial_t p_t(x) = \Delta_\alpha p_t(x)$, we have
\begin{align} \label{e_weaksol1}
&  \int_{Q_T}u(t,x)^{1+\beta} p_{T+n^{-1} - t}(x-x_0) \,dx dt  \nonumber
\\ &\hspace{5 mm}= \int_{\R^d} p_{T+n^{-1}}(x-x_0) \mu(dx)  - \int_{\R^d} u(T,x) p_{n^{-1}}(x-x_0) dx.
\end{align}
For all $n \geq 1$, $\| p_{T + n^{-1}}\|_\infty \leq C T^{-\frac d \alpha}$ by \eqref{heatkernelbds2}, and so by bounded convergence the first term in the second line converges to $S_T \mu (x_0)$. The second term converges to $u(T,x_0)$ by continuity of $u$. Now consider the first line. Let $\epsilon > 0$. For $n \geq 1$ and $t \leq T-\epsilon$, $\|p_{T +n^{-1} - t}\|_\infty \leq C\epsilon^{-\frac d \alpha}$. Since $u \in L^{1+\beta}(Q_T)$, we can apply Dominated Convergence to obtain that
\[ \lim_{n \to \infty} \int_0^{T-\epsilon} \int_{\R^d}u(t,x)^{1+\beta} p_{T+n^{-1} - t}(x-x_0)\, dx dt = \int_0^{T-\epsilon} \int_{\R^d} u(t,x)^{1+\beta} p_{T - t}(x-x_0)\, dx dt.\]
On the other hand, by \eqref{e_heatcomparison} there is a constant $K>0$ such that $u(t,x)^{1+\beta} \leq K$ for all $x \in \R^d$ and $t \in [T-\epsilon,T]$. Consequently, we have
\[ \lim_{\epsilon \downarrow 0} \sup_{n \geq 1} \left| \int_{T-\epsilon}^T \int_{\R^d} u(t,x)^{1+\beta} p_{T+n^{-1} - t}(x-x_0) \,dxdt \right| = 0.\]
Combining everything above, we take $n\to \infty$ in \eqref{e_weaksol1} to conclude that for all $(T,x_0) \in Q$,
\[u(T,x_0) = S_T \mu (x_0) - \int_{Q_T} u(t,x)^{1+\beta} p_{T-t}(x-x_0) \,dxdt,\]
and hence $u$ is equal to the solution of \eqref{e_integ_measure}.\end{proof} 

\begin{remark} \label{remark_integbounded} Along the same lines as the above, it can be shown that if $u(t,x)$ is a bounded weak solution to \eqref{e_pde} and $u(t,x) \to \phi(x)$ a.e. as $t\downarrow 0$ for $\phi \in \cB^+_b$, then $u(t,x) = u^\phi(t,x)$, the solution to \eqref{e_evol} with $u_0 = \phi$. This also implies that such a solution is unique and has the probabilistic representation $u(t,x) = \N_x(1-\exp(-X_t(\phi)))$ by \eqref{e_Lapcanon}.\end{remark}


We now extend the dual relationship with the $(\alpha,\beta)$-superprocess to measures. 

\begin{lemma} \label{lemma_measure_duality} Let $\mu \in \cM_F(\R^d)$. Then for $x \in \R^d$ and $t>0$, 
\begin{equation}   \label{e_Lapcanonmeasure2}
\N_x(1-\exp(-\mu(X_t))) = u^\mu_t(x),
\end{equation} 
and for $X_0 \in \cM_F(\R^d)$,
\begin{equation} \label{e_Lapmeasure2}
E^X_{X_0}(\exp (- \mu(X_t))) = \exp(-X_0(u_t^\mu)).
\end{equation}
\end{lemma}
\begin{proof}
We give the proof under $P_{X_0}^X$ and note that it follows by essentially the same argument for the canonical measure. (One can restrict to the event $\{X_t \neq 0 \}$ because $1 - \exp(-\mu(X_t)) = 0$ on $\{X_t = 0\}$, which allows us to treat $\N_x$ as a finite measure.) Fix $\mu \in \cM_F(\R^d)$ and define $\mu_n = S_{n^{-1}} \mu$. Then $\mu_n$ is smooth, bounded and positive, so by \eqref{e_Lap} we have
\begin{equation} \label{e_lapmeasureaux1}
E^X_{X_0}(\exp(-X_t(\mu_n))) = \exp(-X_0(u^{\mu_n}_t)).
\end{equation}
Since $\mu_n \to \mu$ in $\cM_F(\R^d)$ as $n \to \infty$, by Lemma~\ref{lemma_stability}(a) it follows that $u^{\mu_n} \to u^\mu$ locally uniformly. In particular, $u^{\mu_n}_t \to u^\mu_t$ pointwise. By \eqref{e_heatcomparison} and \eqref{heatkernelbds2}, we have
\[ |u^{\mu_n}_t(x)| \leq C \mu(1) t^{-\frac d \alpha} \]
for all $n \geq 1 $ and $x \in \R^d$. Hence $X_0( u^{\mu_n}_t) \to X_0( u^\mu_t)$ by Dominated Convergence, and consequently
\begin{equation} \label{e_lapmeasureaux2}
\lim_{n \to \infty} \exp(-X_0(u^{\mu_n}_t)) = \exp( - X_0( u^\mu_t)).
\end{equation}
Consider now the left hand side of \eqref{e_lapmeasureaux1}. Expanding $(X_t,\mu_n)$, we have
\[ X_t(\mu_n) = X_t(S_{n^{-1}} \mu) = \int S_{n^{-1}} \mu(x) X_t(dx) = \int S_{n^{-1}} X_t(x) \mu(dx).\]
We have used the symmetry of $S_{n^{-1}}$. Note that $S_{n^{-1}} X_t(x) = X_t * p_{n^{-1}}(x)$ is an approximation of the density which satisfies the conditions of Lemma~\ref{lemma_density_convergence}(c), and hence $X_t(\mu_n) \to \mu(X_t)$ in $L^1(P^X_{X_0})$. In particular, this implies the convergence of the left hand side of \eqref{e_lapmeasureaux1} to $E^X_{X_0}(\exp(-\mu(X_t))$. Combined with \eqref{e_lapmeasureaux2}, this implies \eqref{e_Lapmeasure2} and completes the proof.
\end{proof}

\begin{remark} \label{remark_comp_prin}Note that \eqref{e_Lapcanonmeasure2} immediately implies that $\mu \to u^\mu_t$ is increasing, which we have already stated as Lemma~\ref{lemma_stability}(b). We will refer to this monotonicity in initial conditions of solutions to \eqref{e_pde} as the \textit{comparison principle}. Given Remark~\ref{remark_integbounded}, the comparison principle also holds for bounded weak solutions with initial data in $\cB^+_b$. \end{remark}


Solutions to \eqref{e_pde} satisfy a useful scaling property. For $\lambda > 0$ and $\mu \in \cM_F(\R^d)$, one can verify directly using elementary methods and uniqueness of solutions to problem \eqref{e_pde_measureproblem} the following formula:
\begin{equation} \label{e_scale}
u^{\lambda \mu}(t,x) = \lambda^{\alpha \gamma} u^{\mu(\cdot / \lambda^{\beta \gamma})}(\lambda^{\alpha \beta \gamma}t, \lambda^{\beta \gamma}x),
\end{equation}
where $\gamma = \frac{1}{\alpha - \beta}$, and the measure $\mu(\cdot / \lambda^{\beta \gamma})$ is the dilation of $\mu$ defined by
\[ \mu(\cdot / \lambda^{\beta \gamma})(A) = \mu(A / \lambda^{\beta \gamma}) = \int 1(\lambda^{\beta \gamma}x \in A) \mu(dx)\]
for measurable $A\subseteq \R^d$. This leads to a very useful expression when we scale out the time variable $t$ to obtain an expression involving a solution at time $1$; in particular, we have
\begin{equation} \label{e_scalet}
u^{ \mu}_t(x) = t^{-\frac 1 \beta} u^{t^{\frac{\alpha - \beta}{\alpha \beta}}\mu(\cdot /t^{-\frac 1 \alpha})}(1,t^{-\frac 1 \alpha}x).
\end{equation}
It follows that $u^{\infty \mu}_t(x) = \lim_{\lambda \to \infty} u^{\lambda \mu}_t(x)$ satisfies
\begin{equation} \label{e_scaletinf}
u^{ \infty \mu}_t(x) = t^{-\frac 1 \beta} u^{\infty \mu(\cdot /t^{-\frac 1 \alpha})}(1,t^{-\frac 1 \alpha}x).
\end{equation}

\subsection{A Feynman-Kac formula} \label{s_FK}
We now state a Feynman-Kac formula for some functions related to solutions of \eqref{e_evol}. First, for $\phi,\psi \in \cB^+_b$, we formally define
\begin{equation} \label{e_eps_deriv}
z^{\phi,\psi}(t,x) =  \frac{\partial}{\partial \epsilon} \restr{u^{\phi + \epsilon \psi}(t,x)}{\epsilon = 0} = \lim_{\epsilon \downarrow 0} \frac{u^{\phi + \epsilon \psi}(t,x) - u^\phi(t,x) }{\epsilon},
\end{equation}
where $u^\phi(t,x)$ is taken to be the solution to \eqref{e_evol} with $u_0 = \phi$. By \eqref{e_Lapcanon} and the Dominated Convergence Theorem, it follows that  the derivative with respect to $\epsilon$ exists and
\[ z^{\phi,\psi}(t,x) =  \N_x( X_t(\psi) \exp(-X_t(\phi ))). \]
Recall that $E^W_x$ denotes the expectation associated to an $\alpha$-stable process $W_t$ with $W_0 = x$.
\begin{lemma}\label{lemma_FK} (a) For $\phi,\psi \in \cB^+_b$, 
\begin{equation} \label{e_FKgen}
z^{\phi,\psi}(t,x) = E^W_x\left( \psi(W_t) \exp \left( - (1+\beta)\int_0^t u^{\phi}_s(W_s)^\beta ds \right) \right).
\end{equation}
\noindent(b) For $\phi \in \cB^+_b$ and $\lambda > 0$, $\frac{\partial}{\partial \lambda} u^{\lambda \phi}_t(x)$ exists for all $(t,x) \in Q$ and
\begin{equation} \label{e_FKlambda}
\frac{\partial}{\partial \lambda}u^{\lambda \phi}_t(x) = E^W_x\left( \phi(W_t) \exp \left( - (1+\beta) \int_0^t u^{\lambda \phi}_t(W_s)^\beta ds \right) \right)
\end{equation}
In particular, for $(t,x) \in Q$ and $\Lambda >0 $ we have
\begin{equation} \label{e_FK_FTC}
u^{\Lambda \phi}_t(x) = \int_0^\Lambda \frac{\partial}{\partial \lambda}u^{\lambda \phi}_t(x)\, d\lambda.
\end{equation} 
\end{lemma}
A proof of part (a) is implicit in the proof of Theorem 6.3.1 in \cite{DawsonInfDiv1991}, where the representation \eqref{e_FKgen} is established by showing that $z^{\phi,\psi}(t,x)$ solves a certain linear evolution equation. To see that (b) follows from (a), it suffices to consider $z^{\lambda \phi, \phi}(t,x)$ for $\phi \in \cB^+_b$ and $\lambda >0$.


\subsection{Cluster decompositions} \label{s_cluster}
The connection between the $(\alpha,\beta)$-superprocess and its canonical measure is via a cluster representation in which the superprocess is given by a Poisson superposition of clusters whose intensity is canonical measure. To make this precise, for $X_0 \in \cM_F(\R^d)$ we define
\[ \N_{X_0} (\cdot) = \int \N_x(\cdot) X_0(dx).\]
Let $\Xi(\cdot)$ be a Poisson point process on $\D([0,\infty), \cM_F(\R^d))$ with intensity $\N_{X_0}$. Then the process
\[ X_t(\cdot) = \begin{cases} \int \nu_t(\cdot) \Xi(d\nu) &\text{ if } t>0, 
\\ X_0(\cdot) &\text{ if } t=0 \end{cases} \]
is an $(\alpha,\beta)$-superproces with law $P^X_{X_0}$. This is a consequence of Theorem 4.2.1 of \cite{DLG2005}. For fixed $t>0$ the above implies that
\begin{equation} \label{clusterrep}
X_t \stackrel{\cD}{=} \sum_{i =1}^{N} X^i_{t},
\end{equation}
where $\stackrel{\cD}{=}$ indicates equality of distribution. In the above, $N$ is a Poisson random variable with mean $X_0(1)(\beta t)^{-\frac 1 \beta}$ and the $X_t^i$ are iid random measures with distribution $\N_{X_0}(X_t \in \cdot \, | \, X_t \neq 0)$. This representation gives us a convenient way to compare path properties under $P^X_{X_0}$ and $\N_x$. In particular, suppose we realize $X_t$ under $P^X_{\delta_x}$ via \eqref{clusterrep}. Since the probability that $N=1$ is positive, we can condition on this event, and it follows that
\begin{equation} \label{clustercanon}
\N_x(X_t \in \cdot \, | \, X_t \neq 0) = P^X_{\delta_x}(X_t \in \cdot\, | \, N = 1). 
\end{equation}
Consequently, for measurable $A \subset \cM_F(\R^d)$ we have
\begin{equation} \label{clustertransfer}
\text{If } P^X_{\delta_x}(X_t \in A \, | \, X_t \neq 0 ) = 1, \text{ then } \N_x( X_t \in A \, | \, X_t \neq0 ) = 1.	
\end{equation}
 
\section{The density at a fixed point} \label{s_points}
In the introduction, we noted that Theorem~\ref{thm_pointdichotomy} is equivalent to Theorem B, which is proved analytically. In this section we give a probabilistic proof of Theorem~\ref{thm_pointdichotomy}(a), and therefore of Theorem B(a). We define $u^\lambda_t(x)$ as the solution to \eqref{e_pde_measureproblem} with initial measure $\lambda \delta_0$. Then by translation invariance of the equation \eqref{e_pde} and \eqref{e_Lapmeasure2}, for $x\in\R^d$ we have
\begin{equation}  \label{e_pointLap}
E^X_{X_0}(\exp(-\lambda X_t(x))) = \exp \left(-\int u^\lambda_t(y-x) X_0(dy) \right)
\end{equation}
We define $u^\infty_t = \lim_{\lambda \to \infty} u^\lambda_t$ and observe that, by taking $\lambda \to \infty$ in \eqref{e_pointLap},
\begin{equation} \label{e_pointLap_inf}
P_{X_0}^X(X_t(x) = 0) = \exp \left(-\int u^\infty_t(y-x) X_0(dy)\right).
\end{equation}
The main purpose of this section is to show the following.
\begin{proposition} \label{prop_pointflat} If $\beta \leq \beta^*(\alpha) = \frac{\alpha}{d+\alpha}$, then for fixed $x \in \R^d$, $X_t(x) > 0$ a.s. on $\{X_t \neq 0\}$ under $P^X_{X_0}$ and $\N_0$. Hence $u^\infty_t = U_t$. 
\end{proposition}

In fact, this result is a consequence of the more general Theorem~\ref{thm_flatFrost}, which we prove in Section~\ref{s_flat}. However, we state and prove it separately for a few reasons. First, while Theorem~\ref{thm_flatFrost} (and several other results) concern the behaviour of $\mu(X_t)$ for certain families of measures, the measure $\delta_x$ is of particular interest because it corresponds to the density at a fixed point. The other reason is that, while the method used to prove Proposition~\ref{prop_pointflat} and Theorem~\ref{thm_flatFrost} is largely the same, the proof of the latter involves technicalities that do not arise in the former. We therefore opt to include the simpler proof in the case which is particularly interesting.

Recall that we will sometimes write $u^\lambda(t,x)$ to denote $u^\lambda_t(x)$. We note the particular form that the scaling relationship \eqref{e_scalet} takes for the family $u^\lambda(t,x)$. Since $\delta_0(\cdot / r) = \delta_0$, it follows that
\begin{equation} \label{e_scalepoint}
u^\lambda(t,x) = t^{-\frac 1 \beta} u^{t^{\frac{\alpha - \beta}{\alpha \beta}} \lambda}(1,t^{-\frac 1 \alpha}x)
\end{equation}
Consequently, $u^\infty(t,x)$ satisfies
\begin{equation} \label{e_infscalet_point}
u^\infty(t,x) = t^{-\frac 1 \beta} u^\infty(1, t^{-\frac 1 \alpha} x).
\end{equation}

The following lemma gives a lower bound for $u^\lambda_1$ (for $\lambda \geq 1$) in terms of the heat kernel $p_1$ of the symmetric $\alpha$-stable process and holds for all $0 < \beta < \frac \alpha d$. The statement of this result for $u^\infty_1$ appeared in \cite{CVW2016}, where it was called Lemma 5.3. Here we give a probabilistic proof.

\begin{lemma} \label{lemma_heatkernellowerbd}
There is a constant $c_{\ref{lemma_heatkernellowerbd}} = c_{\ref{lemma_heatkernellowerbd}}(\alpha , \beta, d) > 0 $ such that for all $\lambda \geq 1$,
\[ u^\lambda_1(x) \geq c_{\ref{lemma_heatkernellowerbd}} p_1(x).\]
In particular, $u^\infty_1(x) = \N_0(X_1(x) > 0 ) \geq c_{\ref{lemma_heatkernellowerbd}} p_1(x)$.
\end{lemma}

\begin{proof}
For $\lambda >0$ and $\epsilon > 0$, consider $u^{\lambda p_\epsilon}(t,x)$, the unique solution to \eqref{e_evol} with $u_0 = \lambda p_\epsilon$. By Lemma~\ref{lemma_FK}(b), we have
\begin{equation}
\frac{\partial}{\partial \lambda} u^{\lambda p_\epsilon}(1,x) = E^W_x \left( p_\epsilon(W_1) \exp \left( - (1+\beta) \int_0^1 u^{\lambda p_\epsilon}(1-s, W_s)^\beta \,ds\right) \right). \nonumber
\end{equation}
By \eqref{e_heatcomparison}, $u^{\lambda p_\epsilon}_s \leq \lambda S_s p_\epsilon = \lambda p_{s + \epsilon}$, which implies that 
\begin{align}
\frac{\partial}{\partial \lambda} u^{\lambda p_\epsilon}(1,x) &\geq E^W_x \left( p_\epsilon(W_1) \exp \left( - (1+\beta) \int_0^1 \lambda^\beta p_{1-s+\epsilon}(x + W_s)^\beta \,ds\right) \right) \nonumber
\\ &\geq E^W_x \left( p_\epsilon(W_1)\right) \exp \left( - \lambda^\beta (1+\beta)  \int_0^1  p_{1-s+\epsilon}(0)^\beta \,ds\right) . \nonumber
\end{align}
In the second line we have used the fact that $p_{1-s+\epsilon}$ is radially decreasing, hence $p_{1-s+\epsilon}(W_s) \geq p_{1-s+\epsilon}(0)$, and removed the exponential from the expectation because it no longer depends on $W$. We note that $E^W_x(p_\epsilon(W_1)) = S_1 p_\epsilon(x) = p_{1+\epsilon}(x)$ from the semigroup property. Using \eqref{heatkernelbds2} and changing variables in the integral, we then have
\begin{align}
\frac{\partial}{\partial \lambda} u^{\lambda p_\epsilon}(1,x) &\geq p_{1+\epsilon}(x) \exp \left(-\lambda^\beta (1+\beta) \int_0^1 C_{\ref{heatkernelbds}}^\beta (s+\epsilon)^{-\frac{ d \beta}{\alpha}} ds \right) \nonumber
\end{align}
Because $\beta < \frac{\alpha}{d}$, the integral remains bounded as $ \epsilon \downarrow 0$. It follows that for a constant $C>0$, for all $\epsilon \in (0,1]$,
\begin{equation}
\frac{\partial}{\partial \lambda} u^{\lambda p_\epsilon}(1,x) \geq p_{1+\epsilon}(x) \exp \left( - C \lambda^\beta \right). \nonumber
\end{equation}
In particular, \eqref{e_FK_FTC} and the above imply that
\begin{equation}
u^{p_\epsilon}(1,x) \geq c_0 p_{1+\epsilon}(x) \nonumber
\end{equation}
for a constant $c_0 > 0$. By Lemma~\ref{lemma_stability}(a), the left hand side converges to $u^1(1,x)$ as $\epsilon \downarrow 0$, and the right hand side converges to $p_1(x)$. Thus we have $u^1(1,x) \geq c_0 p_1(x)$. Since $\lambda \to u^\lambda(1,x)$ is increasing, this implies the result.
\end{proof}

By \eqref{e_infscalet_point} and Lemma~\ref{lemma_heatkernellowerbd}, one obtains that 
\begin{equation} \label{e_uinf_lwrbound}
u^\infty_t(x) \geq c_{\ref{lemma_heatkernellowerbd}}t^{-\frac 1 \beta} p_1(t^{-\frac 1 \alpha}x).
\end{equation} 
For fixed $x \neq 0$, by \eqref{heatkernelbds2} we then have
\[ \liminf_{t\downarrow 0} \frac{u^\infty_t(x)}{ t^{- \frac 1 \beta + \frac{d+\alpha}{\alpha}}} > 0.\]
It is therefore immediate that $\lim_{t\downarrow 0} u^\infty_t(x) = \infty$ for all $x \in \R^d$ when $\beta < \beta^*(\alpha)$. It does not give the same conclusion when $\beta = \beta^*(\alpha)$, and in neither case is it immediate that $u^\infty_t$ is flat.

Let $(\cF_t)_{t \geq 0}$ denote the standard right-continuous filtration associated to $X = (X_t: t \geq 0)$. For a $P_{X_0}^X$-integrable function $f(X)$, to denote its conditional expectation we will omit the sub- and superscripts and simply write $E(f(X) | \cF)$. The (one-dimensional) Markov property for $X$ is then expressed as
\[E(f(X_{t+s}) \,|\,\cF_s )(\omega) = E^X_{X_s(\omega)}(f(X_t)). \]

Because $X$ is c\`adl\`ag, $P^X_{X_0}(X_t = X_{t^+} \,\,\forall t >0 ) = 1$. The following lemma gives almost surely left continuity at a fixed time.
\begin{lemma} \label{lemma_contfix} Fix $t>0$ and $X_0 \in \cM_F(\R^d)$. Almost surely under $P^X_{X_0}$ there is no discontinuity of $s \to X_s$ at time $t$, and hence $X_s\to X_t $ in $\cM_F(\R^d)$ as $s \uparrow t$. Moreover, for any open or closed ball $B$, $\lim_{s \uparrow t} X_s(B) = X_t(B)$ almost surely.
\end{lemma}
\begin{proof}
Fix $t>0$. The claim is a consequence of Lemma 1.6 of \cite{FMW2010}. Part (a) of that lemma states that the discontinuities of $t \to X_t$ are described by a jump measure $N(d(s,x,r))$, and the form of the compensator of $N$ given in part (b) implies that a.s. there is no jump at time $t$. Hence $X_s \to X_t$ weakly as $s \uparrow t$. For an open or closed ball $B$, the fact that $X_s(B) \to X_t(B)$ as $s \uparrow t$ follows from weak convergence and the fact that $X_t(\partial B) = 0$ (by absolute continuity).
\end{proof}

\begin{proof}[Proof of Proposition~\ref{prop_pointflat}] Let $X_0 \in \cM_F(\R^d)$, $x \in \R^d$ and $t>0$. Our method is to show that $P_{X_0}^X( X_t(x) = 0,  X_t \neq 0 ) = 0$. This implies that $P_{X_0}^X(X_t(x) > 0 , X_t \neq 0) = P_{X_0}^X(X_t \neq 0)$. At the end of the proof we discuss the case for canonical measure and show that $u^\infty_t = U_t$.

Recall that $B(x,r) = \{y \in \R^d:|y-x| \leq r\}$. From Theorem A, we have $\text{supp}(X_t) = \R^d$ a.s. on $\{X_t\neq 0\}$. In particular, $P_{X_0}^X$-a.s. on $\{X_t \neq 0 \}$ we have $X_t(B(x,1)) > 0$. In other words, 
\[ P_{X_0}^X \left(\{X_t \neq 0\} \cap  \left( \cup_{n=1}^\infty \left\{ X_t(B(x,1)) \geq \frac 1 n \right\} \right)^c \,\right) = 0.\]
Thus it suffices to show that $P_{X_0}^X(\{X_t(x) = 0\} \cap A_\kappa) = 0$ for every $\kappa > 0$, where we define the event 
\begin{equation}
A_\kappa = \{X_t(B(x,1)) \geq \kappa\}. \nonumber
\end{equation} 

Let $\kappa>0$ and consider the event $A_{2\kappa}$. By Lemma~\ref{lemma_contfix}, $\lim_{s\uparrow t}X_s(B(x,1)) = X_t(B(x,1)) \geq 2\kappa$. We let $\delta_n = 2^{-n}$ and note that the previous statement implies that $X_{t-\delta_n}(B(x,1)) \geq \kappa$ for $n$ sufficiently large (depending on $\omega$). That is, for a.e. $\omega \in A_{2 \kappa}$,
\begin{equation} \label{asymp_mass}
\exists \, N = N(\omega) \text{ such that } n \geq N \Rightarrow  X_{t-\delta_n}({B(x,1)}) \geq \kappa.
\end{equation} 
Fix $\omega$ and $N$ as in \eqref{asymp_mass} and let $n \geq N$. Applying the Markov property to $X$ at time $t - \delta_n$, we obtain
\begin{align}
P(X_t(x) = 0 \, | \, \cF_{t-\delta_n})(\omega) &= P^X_{X_{t-\delta_n}(\omega)}(X_{\delta_n}(x) = 0) \nonumber
\\ &= \exp \left( -\int u_{\delta_n}^\infty(y - x) \,X_{t-\delta_n}(dy) \right), \nonumber
\end{align}
where the second equality uses \eqref{e_pointLap_inf}. We have also suppressed the dependence on $\omega$ in the last expression. We now bound above by ignoring all the mass of $X_{t-\delta_n}$ outside of $B(x,1)$. This gives
\begin{align}
P(X_t(x) = 0 \, | \, \cF_{t-\delta_n})(\omega) &\leq \exp \left( -\int_{B(x,1)} u_{\delta_n}^\infty(y - x) \,X_{1-\delta_n}(dy) \right) \nonumber
\\ &\leq \exp \left( -c_{\ref{lemma_heatkernellowerbd}}  \int_{B(x,1)} \delta_n^{-\frac 1 \beta} p_{1}(\delta_n^{-\frac 1 \alpha}(y - x)) \,X_{1-\delta_n}(dy) \right), \nonumber
\end{align}
where the second inequality uses \eqref{e_uinf_lwrbound}. Since $p_1$ is radially decreasing, the minimum value it can attain in the integral above is $p_1(2\delta_n^{-\frac 1 \alpha})$, and $p_1(2\delta_n^{-\frac 1 \alpha}) \geq c_1 \delta_n^{\frac{d+\alpha}{\alpha}}$ for some $c_1>0$ by \eqref{heatkernelbds2}. Using this and \eqref{asymp_mass}, we obtain that for $c_2 = c_1\cdot c_{\ref{lemma_heatkernellowerbd}}>0$, for $\omega \in A_{2\kappa}$ and $n \geq N(\omega)$,
\begin{align}
P(X_t(x) = 0 \, | \, \cF_{t-\delta_n})(\omega) &\leq \exp \left( -c_2 \delta_n^{-\frac 1 \beta + \frac{d+\alpha}{\alpha}} X_{1-\delta_n}(B(x,1)) \right)  \nonumber
\\ &\leq \exp \left( -c_2 \kappa  \delta_n^{-\frac 1 \beta + \frac{d+\alpha}{\alpha}}\right). \nonumber
\end{align} 
In view of \eqref{asymp_mass}, we have:
\begin{align}\label{e_condprobexpsmall} 
&\text{For $P_{X_0}^X$-a.e. $\omega \in A_{2\kappa}$, for all $n \geq N(\omega)$, we have} \nonumber
\\& \hspace{16 mm} P_{X_0}^X( X_t(x) = 0 \, | \, \cF_{t-\delta_n})(\omega) \leq \exp \left( -c_2 \kappa \, \delta_n^{-\frac 1 \beta + \frac{d+\alpha}{\alpha}}\right).
\end{align}
First suppose that $\beta < \beta^*(\alpha) = \frac{\alpha}{d+\alpha}$. In this case, the exponent of $\delta_n$ in \eqref{e_condprobexpsmall} is negative and so the right hand side of \eqref{e_condprobexpsmall} converges to $0$ as $n \to \infty$. From \eqref{asymp_mass}, \eqref{e_condprobexpsmall} and the tower property, we obtain that
\begin{align}
&P_{X_0}^X(\{X_t(x) = 0\} \cap A_{2\kappa}) \nonumber
\\ & \hspace{8 mm}= E^X_{X_0}( P( \{X_t(x) = 0\} \cap A_{2\kappa}\, | \, \cF_{t-\delta_n})) \nonumber
\\ &\hspace{8 mm}\leq P_{X_0}^X(A_{2\kappa})\,\exp \left( -c_0 \kappa \, \delta_n^{- \frac 1 \beta + \frac{d+\alpha}{\alpha} } \right)  \nonumber
\\ &\hspace{ 18 mm} + P_{X_0}^X(A_{2\kappa} \cap \{N(\omega) > n\}) \nonumber 
\\ &\hspace{8 mm } \to 0\,\, \text{    as $n \to \infty$}. \nonumber
\end{align}
We have therefore shown that
\[ P_{X_0}^X( \{X_t(x) = 0 \} \cap A_{2\kappa}) = 0.\]
This suffices to prove the result, so the proof is complete for $\beta < \frac{\alpha}{d+ \alpha}$. 

Now suppose that $\beta = \frac{\alpha}{d+\alpha}$. Here we use a martingale argument. First, we observe that in this case, \eqref{e_condprobexpsmall} implies there exists $c(\kappa) <1$ such that for $\omega \in A_{2\kappa}$ and $n \geq N(\omega)$,
\begin{equation}
P(X_t(x) = 0 \, | \, \cF_{t-\delta_n})(\omega) \leq c(\kappa). \nonumber
\end{equation}
In particular,
\begin{equation} \label{e_crit_lwrbd}
P(X_t(x) > 0 \, | \, \cF_{t-\delta_n})(\omega) \geq 1 - c(\kappa) > 0
\end{equation}
for $n \geq N(\omega)$. The process $P(X_t(x) > 0 \, | \, \cF_{t-\delta_n})$ is a martingale with respect to the increasing sequence of $\sigma$-algebras $\{\cF_{t-\delta_n}\}_{n=1}^\infty$. By the martingale convergence theorem it follows that
\begin{equation} \label{e_mtgconvergence}
\lim_{n \to \infty} P(X_t(x) > 0 \, | \, \cF_{t-\delta_n})(\omega) = P(X_t(x) > 0 \, | \, \cF_{t^-})(\omega)
\end{equation} 
for $P_{X_0}^X$-a.e. $\omega$, where $\cF_{t-\delta_n} \uparrow \cF_{t^-} := \sigma(X_s : 0 \leq s<t)$. To complete the result, it suffices to show that the right hand side of the above is equal to $1(X_t(x) > 0)(\omega)$ almost surely. By \eqref{e_density_liminf} we have
\begin{align}
X_t(x) &= \liminf_{\epsilon \downarrow 0} \frac{X_t(B(x,\epsilon))}{|B(x,\epsilon)|} = \liminf_{\epsilon \downarrow 0}\frac{X_{t^-}(B(x,\epsilon))}{|B(x,\epsilon)|} \hspace{2 mm} \text{ a.s.}, \nonumber
\end{align}
where $X_{t^-}(B(x,\epsilon)) = \lim_{s \uparrow t}X_s(B(x,\epsilon))$ exists and equals $X_t(B(x,\epsilon))$ a.s. for all $\epsilon > 0$ by \linebreak Lemma~\ref{lemma_contfix}. In a slight abuse of notation, let us denote by $X_{t^-}(x)$ the quantity on the right hand side of the above. Then $X_{t^-}(x) = X_t(x)$ almost surely and $X_{t^-}(x)$ is $\cF_{t^-}$-measurable. We therefore have, for $P^X_{X_0}$-a.e. $\omega$,
\begin{align}
 P(X_t(x) > 0 \, | \, \cF_{t^-})(\omega) &= P(X_{t^-}(x) > 0 \, | \, \cF_{t^-})(\omega)  \nonumber
\\ &= 1(X_{t^-}(x) > 0 )(\omega) \nonumber
\\  &= 1(X_t(x) > 0 )(\omega). \nonumber
\end{align}
Hence \eqref{e_mtgconvergence} implies that
\begin{equation}
\lim_{n \to \infty} P(X_t(x) > 0 \, | \, \cF_{t-\delta_n}) = 1(X_t(x) > 0)\hspace{3 mm} \text{$P_{X_0}^X$-a.s.} \nonumber
\end{equation}

By \eqref{e_crit_lwrbd}, it follows that, on the event $A_{2\kappa}$, the left hand side of the above is greater than or equal to $1- c(\kappa)$ for $n \geq N(\omega)$, and hence we have
\begin{equation}
1(X_t(x) > 0) \geq 1 - c(\kappa) >0 \nonumber
\end{equation}
almost surely on $A_{2\kappa}$. Because $1(X_t(x) > 0 ) \in \{0,1\}$, this implies that it must in fact be equal to $1$. In other words, $P_{X_0}^X(\{X_t(x) = 0\} \cap A_{2\kappa}) = 0$ for all $\kappa>0$, which proves the result. Hence the proof is complete for the case $\beta = \frac{\alpha}{d + \alpha}$, and we are done.

Having shown that $P^X_{X_0}( X_t(x) > 0 \, | \, X_t \neq 0) = 1$, the result under $\N_0$ then follows by \eqref{clustertransfer}. In particular, we obtain that $\N_0(X_t(x) > 0) = \N_0( X_t \neq 0)$, and hence that $u^{\infty}_t(x) = U_t$. \end{proof}

\section{Strict positivity of the density} \label{s_pos}
In this section we prove Theorem~\ref{thm_positiveEverwhere}, which states that the density is strictly positive under certain conditions in the continuous case. In particular, in dimension one ($d=1$) with $\alpha > 1+\beta$ (continuity) and $\beta < \beta^*(\alpha)$ (strong instantaneous propagation), we have
\[X_t(x) > 0 \text{ for all $x \in \R$ almost surely on $\{X_t \neq 0\}$}\]
at a fixed time $t>0$. The proof of the result hinges in part on the following result, which gives an exponential rate of decay for the left tail of the density conditional on non-negligible nearby mass. Its proof shares many ideas with the proof of Proposition~\ref{prop_pointflat}. 

The following holds for general dimensions $d \in \N$. We denote $B_R = \{ x \in \R^d : |x| \leq R\}$.
\begin{lemma} \label{lemma_expbd}
Let $\beta < \frac{\alpha}{d+\alpha}$. Let $R\geq 1$ and $t \in (0,1]$. There is a constant $c_{\ref{lemma_expbd}}>0$ which depends only on $(\alpha,\beta,d)$ such that for any $X_0 \in \cM_F(\R^d)$,
\begin{equation} 
P_{X_0}^X \left( X_t(x) \leq t^{\frac{\alpha - \beta}{\alpha \beta}} \right) \leq e \cdot \exp \left(-c_{\ref{lemma_expbd}} \frac{X_0(B_R)}{R^{d+\alpha}} \,t^{-q}\right) \nonumber
\end{equation}
for all $x \in B_R$, where $q = \frac 1 \beta - \frac{d+\alpha}{\alpha} > 0$.
\end{lemma}
\begin{proof}
Fix $t \in(0,1]$, $R\geq 1$ and $X_0 \in \cM_F(\R^d)$. Let $\lambda = t^{-\frac{\alpha - \beta}{\alpha\beta}}$. One can verify directly that
\[ 1(X_t( x)\leq \lambda^{-1} ) \leq \exp \left( 1 - \lambda X_t(x)\right).\]
Using the above and applying \eqref{e_pointLap}, we have
\begin{align} \label{eexp_aux1}
P_{X_0}^X \left( X_t(x) \leq \lambda^{-1} \right) &\leq e \cdot E_{X_0}^X \left( \exp\left(-\lambda X_t(x) \right) \right)\nonumber
\\ &= e \cdot \exp \left(- \int u^\lambda(t,y-x) X_0(dy) \right) \nonumber
\\ &\leq e \cdot \exp \left(- \int_{B_R} u^\lambda(t,y-x) X_0(dy) \right).
\end{align}
In the last line we simply disregard the mass of $\mu$ outside of $B_R$. Next we obtain a lower bound on the integrand in the above. From \eqref{e_scalepoint}, we have
\begin{align}
u^\lambda(t,y-x) &= t^{-\frac 1 \beta} u^{\lambda \cdot t^{\frac{\alpha - \beta}{\alpha \beta}}}(1, t^{-\frac 1 \alpha}(y-x)) \nonumber
\\ &= t^{-\frac 1 \beta} u^{1}(1, t^{-\frac 1 \alpha}(y-x)) \nonumber
\\ &\geq c_{\ref{lemma_heatkernellowerbd}} t^{-\frac 1 \beta} p_1( t^{-\frac 1 \alpha}(y-x)). \nonumber
\end{align}
The second line uses the fact that $\lambda t^{\frac{\alpha - \beta}{\alpha \beta}} = 1$ and the third follows from Lemma~\ref{lemma_heatkernellowerbd}. Finally, for all $x,y \in B_R$, we use \eqref{heatkernelbds2} and the above to obtain that
\begin{equation}
u^\lambda(t,y-x) \geq c \,\frac{t^{-\frac 1 \beta + \frac{d+ \alpha}{\alpha}}}{R^{d+\alpha}} \hspace{5 mm}\forall \, x,y \in B_R \nonumber
\end{equation}
for a constant $c >0$ depending only on $(\alpha,\beta,d )$. Using the above in \eqref{eexp_aux1}, we obtain that
\begin{align}
P_{X_0}^X \left( X_t(x) \leq \lambda^{-1} \right) &\leq   e \cdot \exp \left(- c X_0(B_R) \,\frac{t^{-\frac 1 \beta + \frac{d+ \alpha}{\alpha}}}{R^{d+\alpha}} \right). \nonumber
\end{align}
Since $\lambda = t^{-\frac{\alpha - \beta}{\alpha\beta}}$, this completes the proof. \end{proof}

Besides the above, the other main ingredient in the proof of Theorem~\ref{thm_positiveEverwhere} is H\"older continuity of $X_t(x)$, which we discussed in Section~\ref{s_density}. In particular we will use \eqref{e_localmodulus}. As can be seen from the proof, the actual index of H\"older continuity is irrelevant. Any positive index works.

\begin{proof}[Proof of Theorem~\ref{thm_positiveEverwhere}] Let $d=1$ and $\alpha > 1+\beta$. Let $X_0 \in \cM_F(\R^d)$. By scaling it is sufficient to consider the time $t=1$. We will show that $X_1(x) > 0$ for all $x \in [-R,R]$ $P^X_{X_0}$-a.s. for every $R \geq 1$, and hence that $X_1(x) > 0 $ for all $x \in \R$.

Fix $R \in \N$. As in the proof of Proposition~\ref{prop_pointflat}, we will use instantaneous propagation. In particular, by Theorem A we have
\[ P^X_{X_0} (X_1([-R,R]) > 0 \, | \, X_1 \neq 0) = 1 .\]
Let $E^{R,\kappa} = \{X_1([-R,R]) \geq \kappa\}$. By the above, it suffices to show that
\begin{equation} \label{pos_probsufficient}
X_1(x)>0 \,\text{  for all $x \in [-R,R]\,$   a.s. on $E^{R,\kappa}$}  
\end{equation}
for all $\kappa >0$. We fix $\kappa > 0$ and consider the event $E^{R,2\kappa}$. For a sequence $\{\delta_n\}_{n \in \N} = \{2^{-\gamma n}\}_{n \in \N}$, with $\gamma > 0$ to be specified later, we define events $B_n^{R,\kappa}$ by
\begin{equation} \label{def_Bn}
B_n^{R,\kappa} = \{X_{1-\delta_n} ([-R,R]) \geq \kappa  \}. \nonumber
\end{equation}
By Lemma~\ref{lemma_contfix}, $\lim_{s \uparrow 1} X_s([-R,R]) = X_1([-R,R]) \geq 2\kappa$ a.s. on $E^{R,2\kappa}$, and hence for a.e. $\omega \in E^{R,2\kappa}$ there is $s_0(\omega)<1$ such that $X_{s}([-R,R]) \geq \kappa$ for all $s \in [ s_0(\omega),1]$. Hence $B_n^{R,\kappa}$ occurs for sufficiently large $n$, that is,
\begin{equation} \label{events_E_B}
 E^{R,2\kappa} \subseteq \{B_n^{R,\kappa} \text{ eventually}\},
\end{equation}
where
\begin{equation} 
\{B_n^{R,\kappa} \text{ eventually}\}= \bigcup_{N=1}^\infty \bigcap_{n=N}^\infty B_n^{R,\kappa}. \nonumber
\end{equation}
For $n \in\N$, let $\Lambda_n$ denote the set of dyadic lattice points at scale $2^{-n}$, i.e. $\Lambda_n = 2^{-n} \Z$. Recall that $R \in \N$. We then let $\Lambda_n^R = [-R,R] \cap \Lambda_n$, that is
\[\Lambda_n = \{-R + k2^{-n} : k=0,1,\hdots,2R2^n\}.\]
Next, we define
\begin{equation} \label{events_Fn}
F_n = F_n(R) = \{ X_1(x) > \delta_n^{\frac{\alpha - \beta}{\alpha\beta}}\,\,\forall x \in \Lambda_n^R\}.
\end{equation}
The first step of our proof is to show that 
\begin{equation} \label{prob_io_0}
P_{X_0}^X(E^{R,2\kappa} \cap \{F_n^c \text{ i.o.}\}) = 0,
\end{equation}
where \emph{i.o.} is short for \emph{infinitely often}, meaning
\[\{F_n^c \text{ i.o.}\} = \bigcap_{N=1}^\infty \bigcup_{n=N}^\infty F_n^c. \]
We now show that \eqref{prob_io_0} holds. By \eqref{events_E_B}, we have
\begin{align}
P_{X_0}^X (E^{R,2\kappa} \cap \{F_n^c \text{ i.o.}\}) &\leq P^X_{X_0}(\{B_n^{R,\kappa} \text{ eventually}\} \cap \{F_n^c \text{ i.o.}\}) \nonumber
\end{align}
Suppose that $\omega$ is in the event on the right hand side of the above; then (i) there is $N(\omega)$ such that $\omega \in B_n^{R,\kappa}$ for all $n \geq N(\omega)$, and (ii) for any $m \in \N$, there is $n > m$ so that $\omega \in F_n^c$. Together, (i) and (ii) imply that for any $m \geq N(\omega)$, there is $n > m$ such that $\omega \in B_n^{R,\kappa} \cap F_n^c$. That is, the above event is a sub-event of $\{B_n^{R,\kappa} \cap F_n^c \, \text{ i.o.}\}$. Hence
\begin{align} \label{limsupprobbd}
P_{X_0}^X (E^{R,2\kappa} \cap \{F_n^c \text{ i.o.}\}) &\leq P^X_{X_0} (\{B_n^{R,\kappa} \cap F_n^c \text{ i.o.}\}) \nonumber
\\ &=  P^X_{X_0} \bigg(\bigcap_{N=1}^\infty \, \bigcup_{n=N}^\infty B_n^{R,\kappa} \cap F_n^c \bigg) \nonumber
\\ &=\lim_{N \to \infty} P^X_{X_0} \bigg(\bigcup_{n=N}^\infty B_n^{R,\kappa} \cap F_n^c \bigg) \nonumber
\\ &\leq \lim_{N \to \infty} \sum_{n=N}^\infty P^X_{X_0} ( B_n^{R,\kappa} \cap F_n^c ).
\end{align}
We bound the probabilities arising in the final term using Lemma~\ref{lemma_expbd}. We condition on $\cF_{1-\delta_n}$ and note that $B_n^{R,\kappa} \in \cF_{1-\delta_n}$. By the Markov property, we have
\begin{align}
P ( B_n^{R,\kappa} \cap F_n^c \, | \, \cF_{1-\delta_n}) &=  1(B_n^{R,\kappa}) P^X_{X_{1-\delta_n}} (F_n^c). \nonumber
\\ &\leq \sum_{x \in \Lambda^R_n} 1(B_n^{R,\kappa}) P^X_{X_{1-\delta_n}} \Big(X(\delta_n,x) \leq \delta_n^{\frac{\alpha - \beta}{\alpha \beta}}\Big). \nonumber
\end{align}
By the definition of $B_n^{R,\kappa}$, in the above we need only compute the probability for such $X_{1-\delta_n}$ as satisfy $X_{1-\delta_n}([-R,R]) \geq \kappa$, in which case we can apply Lemma~\ref{lemma_expbd} for each $x \in \Lambda_n^R$. From this we obtain (recall that $|\Lambda_n^R| = 2R2^n$)
\begin{equation}  
P( B_n^{R,\kappa} \cap F_n^c \, | \, \cF_{1-\delta_n}) \leq C_1(R) \,2^n \exp \left(-c_{\ref{lemma_expbd}} \frac{\kappa}{R^{1+\alpha}} \,\delta_n^{-q} \right), \nonumber
\end{equation}
where $C_1(R) = 2Re$. Recall that we have chosen $\delta_n = 2^{-\gamma n}$, and so, substituting the above into \eqref{limsupprobbd}, we obtain
\begin{align}
P_{X_0}^X (E^{R,2\kappa} \cap \{F_n^c \text{ i.o.}\}) &\leq \lim_{N \to \infty} \sum_{n=N}^\infty C_1(R)\,2^n \exp \left(-c_{\ref{lemma_expbd}} \frac{\kappa}{R^{1+\alpha}} \,2^{\gamma q n} \right) \nonumber
\\ &= 0. \nonumber
\end{align} 
Thus we have shown that \eqref{prob_io_0} holds, implying that $\{F_n \text{ eventually}\}$ occurs a.s. on $E^{R,2\kappa}$. Recalling the definition of $F_n$ from \eqref{events_Fn} and that $\delta_n = 2^{-\gamma n}$, it therefore holds that
\begin{align} \label{lattice_pos}
&\text{For a.e. $\omega \in E^{R,2\kappa}$, there is $N(\omega) \in \N$ such that for all $n \geq N(\omega)$,} \nonumber \\
& \hspace{ 30 mm} \text{$X_1(x) > 2^{-\gamma \frac{\alpha - \beta}{\alpha \beta} n }$ for each $x \in \Lambda_n^R$.}
\end{align}
Next we use the H\"older continuity of $X_t(\cdot)$. Let $0 < \eta < \frac{\alpha}{1+\beta} - 1$. By \eqref{e_localmodulus} with $K = [-R,R]$, for a random constant $C_2(R) = C_2(R, \eta,\omega)>0$,
\begin{equation} \label{modulus_R}
|X_1(x_1) - X_1(x_2)| \leq C_2(R)|x_1 - x_2|^\eta \,\, \text{  for all $x_1,x_2 \in [-R,R]$.}
\end{equation}
Having chosen $\eta$, we can now choose a corresponding value of $\gamma$. Let
$\gamma = \frac{\alpha \beta}{\alpha - \beta} \frac{\eta}{2}$. Then by \eqref{lattice_pos}, for $n$ sufficiently large,
\begin{equation} \label{lattice_pos2}
X_1(x) > 2^{-\frac{\eta}{2} n} \,\,\, \text{ for each $x \in \Lambda_n^R$.}
\end{equation}
Let $y \in [-R,R]$. We define $[y]_{n} = \min \{ x \in \Lambda_n^R : x \geq y \}$. That is, if $y \not \in \Lambda_n^R$, then $[y]_{n}$ the point in $\Lambda_n^R$ nearest to $y$ on the right; if $ y \in \Lambda_n^R$ , then $[y]_{n} = y$. Note that $|y - [y]_{n}| < 2^{-n}$ for all $y \in [-R,R]$ by the definition of $\Lambda_n^R$. Hence by \eqref{modulus_R},
\begin{align} \label{lattice_approx}
\sup_{y \in[-R,R]} |X_1(y) - X_1([y]_{n})| \leq C_2(R) \,2^{-\eta n }.
\end{align}
By the triangle inequality, for $y \in [-R,R]$,
\begin{align}
X_1(y) \geq X_1([y]_n) - |X_1(y) - X_1([y]_{n})|. \nonumber
\end{align}
Note that $\{[y]_n : y \in [-R,R]\} = \Lambda_n^R$. Taking the infimum of the above over $[-R,R]$ and applying \eqref{lattice_approx}, we obtain that
\begin{align}
\inf_{y \in [-R,R]} X_1(y) &\geq \left(\inf_{y \in [-R,R]} X_1( [y]_n) \right) - \sup_{y \in [-R,R]}|X_1(y) - X_1([y]_{n})| \nonumber
\\ &\geq \inf_{x \in \Lambda_n^R} X_1( x) - C_2(R) \,2^{-\eta n}. \nonumber
\end{align}
By \eqref{lattice_pos2}, for all sufficiently large $n$ we therefore have
\begin{align}
\inf_{y \in [-R,R]} X_1(y) &\geq  2^{-\frac{\eta}{2}n} - C_2(R) \,2^{-\eta n} \nonumber
\\ &= 2^{-\frac \eta 2 n}\big( 1 - C_2(R) 2^{-\frac \eta 2 n} \big). \nonumber
\end{align}
By taking $n$ to be large enough in comparison to $C_2(R)$, the right hand side is positive. This proves that the density is strictly positive on $[-R,R]$ a.s. on $E^{R,2\kappa}$. Hence \eqref{pos_probsufficient} holds and the proof is complete. \end{proof}

\section{Almost sure charging of (F1)-$s$ measures when $\beta \leq \beta^*(\alpha,s)$}\label{s_flat}
In this section we prove that, under some conditions on $\alpha$ and $\beta$, $\mu(X_t) > 0$ almost surely on $\{X_t \neq 0 \}$ for certain measures $\mu$, which is equivalent to $\N_x(\mu(X_t) >0) = u^{\infty \mu}_t = U_t$. More precisely, this section contains the proof of Theorem~\ref{thm_flatFrost}(a). We recall the Frostman condition (F1) for a measure: for $s \in [0,d]$, $\mu \in \cM_F(\R^d)$ satisfies (F1)-$s$ if
\begin{itemize}
\item[] (F1)-$s$ \hspace{2 mm} For some constant $\overline{C}$, for all $x \in \R^d$ and $r>0$,
\begin{equation}
\mu(B(x,r)) \leq \overline{C} r^s. \nonumber
\end{equation}
\end{itemize}

Theorem~\ref{thm_flatFrost}(a) states that if $\mu \in \cM_F(\R^d)$ satisfies (F1)-$s$ and $\beta \leq \beta^*(\alpha,s) = \frac{\alpha}{(d-s) + \alpha}$, then $\mu(X_t) > 0$ almost surely on $\{X_t \neq 0 \}$ and, equivalently, $u^{\infty \mu}_t = U_t$.

Without loss of generality, we can assume that $\text{supp}(\mu)$ is bounded. Indeed, if (F1)-$s$ holds for $\mu$, then it also holds for the restriction of $\mu$ to a bounded set. Furthermore, if $\mu'$ denotes this restriction, then $u^{\lambda \mu'}_t \leq u^{\lambda \mu}_t$ for $\lambda > 0$ by the comparison principle, and hence it suffices to show that $u^{\infty \mu'}_t = U_t$. We will further assume that $\text{supp}(\mu) \subseteq B_1 = \{x \in \R^d : |x| \leq 1\}$. This is allowable because, by translation invariance of \eqref{e_pde}, $u^{\infty \mu}_t$ is flat if and only if $u^{\infty \mu_z}_t$ is flat, where $\mu_z$ is the translate of $\mu$ by $z \in \R^d$. We can therefore translate $\mu$ so that it has positive mass in $B_1$, then discard the mass outside $B_1$ by the previous argument.

We set the following standing assumption: for the remainder of this section, let $\mu \in \cM_F(\R^d)$ satisfy (F1)-$s$ for some $s \in [0,d]$ and $\mu(B_1^c) = 0$. Without loss of generality we will suppose that $\mu(B_1) = \mu(\R^d) = 1$.

The proof of Theorem~\ref{thm_flatFrost}(a) uses a similar argument to the proof of Proposition~\ref{prop_pointflat}. Recall that the bound $u^\infty(1,x) \geq c p_1(x)$ from Lemma~\ref{lemma_heatkernellowerbd} played a critical role in that result. We use a similar bound here, which however is adapted to $u^{\infty \mu}_t$ for $\mu$ satisfying (F1)-$s$. By monotonicity of $\lambda \to u^{\lambda \mu}_t$, we have the trivial bound that $u^{\infty \mu}_t \geq u^{\lambda \mu }_t$ for all $\lambda >0$. As we will also be using scaling properties of these solutions, which involve rescaling the initial measure (see \eqref{e_scale}), the critical scale turns out to be $u^{r^s \mu(\cdot / r)}(1,x)$, where we recall that for $r>0$, $\mu(\cdot / r)$ is the measure given by $\mu(A/r) = \int 1_A(rx) d\mu(x)$. If $\mu$ has support $\cS$, then the support of $\mu(\cdot / r)$ is $r\cS = \{rx : x \in \cS\}$. The next result, which is analogous to Lemma~\ref{lemma_heatkernellowerbd}, gives a lower bound for $u^{r^s \mu(\cdot / r)}(1,x)$.
\begin{lemma} \label{lemma_heatlwrbd_Frost} Let $\mu \in \cM_F(\R^d)$ satisfy (F1)-$s$ for some $s \in [0,d]$. Then there is a constant $c_{\ref{lemma_heatlwrbd_Frost}} = c_{\ref{lemma_heatlwrbd_Frost}}(\mu,\alpha,\beta,d) > 0$ such that for all $r\geq 1$,
\begin{equation}
u^{r^s \mu(\cdot / r)}(1,x) \geq c_{\ref{lemma_heatlwrbd_Frost}}\,r^s S_1(\mu(\cdot/r))(x). \nonumber
\end{equation}
\end{lemma}
The proof of Lemma~\ref{lemma_heatlwrbd_Frost} requires the following boundedness result. 

\begin{lemma} \label{lemma_uniformheatbd_Frost}
Let $\mu \in \cM_F(\R^d)$ satisfy (F1)-$s$ for $s \in [0,d]$ and fix $\alpha \in (0,2)$. Then there is a constant $C_{\ref{lemma_uniformheatbd_Frost}} = C_{\ref{lemma_uniformheatbd_Frost}}(\mu,\alpha,d) >0$ such that for all $t > 0$,
\begin{equation}
\sup_{r \geq 1} \sup_{y \in \R^d} S_t (r^s \mu(\cdot  / r))(y) \leq C_{\ref{lemma_uniformheatbd_Frost}} \,t^{-\frac{(d-s)}{\alpha}}. \nonumber
\end{equation}
\end{lemma}
\begin{proof}
Fix $t,r>0$. For $y \in \R^d$, we have
\begin{align} \label{e_unifheatbd_aux1}
S_t (r^s \mu(\cdot / r))(y) &= r^s \int p_t(\tilde{z}-y) d\mu(\tilde{z}/r) \nonumber
\\ &= r^s \int p_t(rz-y) d\mu(z) \nonumber
\\ &= r^s \int_0^\infty \mu \left( \left\{z : p_t(rz-y) \geq k \right\} \right) dk, \nonumber
\end{align}
where the last line uses Fubini's theorem. By \eqref{heatkernelbds2}, it follows that 
\[\{z : p_t(rz-y) \geq k\} \subseteq \{z : C_{\ref{heatkernelbds}} \frac{t}{|rz-y|^{d+\alpha}} \geq k\} \cap \{ C_{\ref{heatkernelbds}} t^{-\frac d \alpha} \geq k\}. \]
Note that the first set on the right hand side is equal to $B(y/r, c (k^{-1}t)^{\frac{1}{d+\alpha}}r^{-1})$ with $c = C_{\ref{heatkernelbds}}^{\frac{1}{d+\alpha}}$. Using this and the fact that $\mu$ satisfies (F1)-$s$, we have
\begin{align}
S_t (r^s \mu(\cdot / r))(y) &\leq r^s \int_0^{C_{\ref{heatkernelbds}} t^{-\frac d \alpha}} \mu \left( B(y/r, c (k^{-1}t)^{\frac{1}{d+\alpha}}r^{-1}) \right) dk \nonumber 
\\ &\leq \overline{C} r^s r^{-s} t^{\frac{s}{d+\alpha}}\int_0^{C_{\ref{heatkernelbds}} t^{-\frac d \alpha}} k^{-\frac{s}{d+\alpha}} dk \nonumber 
\\ &\leq Ct^{-\frac{(d-s)}{\alpha}}, \nonumber
\end{align}
where we recall that $s \leq d$ and so $C$ depends only on $\overline{C}$, $d$ and $\alpha$.
\end{proof}

\begin{proof}[Proof of Lemma~\ref{lemma_heatlwrbd_Frost}] As in the proof of Lemma~\ref{lemma_heatkernellowerbd}, we will use the Feynman-Kac formula from Section~\ref{s_FK}. We cannot apply Lemma~\ref{lemma_FK} directly to $\frac{\partial}{\partial \lambda} u^{\lambda r^s \mu(\cdot / r)}$ because $r^s \mu(\cdot / r)$ is not a function, so for $\epsilon>0$ we define
\begin{equation} \label{e_mollify}
\psi_{\epsilon,r} = S_\epsilon (r^s \mu(\cdot/r)) = r^s  (\mu(\cdot / r) * p_\epsilon).
\end{equation}
Then $\psi_{\epsilon,r}$ is a smooth, bounded, non-negative function, so by Lemma~\ref{lemma_FK}(b), $w^{\lambda \psi_{\epsilon,r}}(t,x) = \frac{\partial}{\partial \lambda} u^{\lambda \psi_{\epsilon,r}}(t,x)$ exists and
\begin{equation} \label{e_FK_smooth}
w^{\lambda \psi_{\epsilon,r}}(t,x) = E^W_x \left( \psi_{\epsilon,r}(W_t) \exp \left( -(1+\beta) \int_0^t u^{\lambda \psi_{\epsilon,r}}(t-\tau,W_\tau)^\beta d\tau \right) \right).
\end{equation}
By \eqref{e_heatcomparison}, we have
\[ u^{\lambda \psi_{\epsilon,r}}_\tau \leq S_\tau( \lambda \psi_{\epsilon,r}) = \lambda S_\tau( \psi_{\epsilon,r})\]
for $\tau > 0$. Hence by \eqref{e_FK_smooth} with $t=1$ we have
\begin{align} 
w^{\lambda \psi_{\epsilon,r}}(1,x) &\geq E^W_x \left( \psi_{\epsilon,r}(W_1) \exp \left( -(1+\beta)\lambda^\beta \int_0^1  [S_{1-\tau}(\psi_{\epsilon,r})(W_\tau)]^\beta d\tau \right) \right) \nonumber
\\ &= E^W_x \left( \psi_{\epsilon,r}(W_1) \exp \left( -(1+\beta)\lambda^\beta \int_0^1  [S_{1-\tau+\epsilon}(r^s \mu(\cdot / r))(W_\tau)]^\beta d\tau\right) \right), \nonumber
\end{align}
where to obtain the final expression we have used \eqref{e_mollify} and the semigroup property. By Lemma~\ref{lemma_uniformheatbd_Frost},
\[ [S_{1-\tau+\epsilon}(r^s \mu(\cdot / r))(W_\tau)]^\beta \leq C_{\ref{lemma_uniformheatbd_Frost}}^\beta (1-\tau+\epsilon)^{-\frac{(d-s)\beta}{\alpha}},\]
and in particular there is a constant $C>0$ such that for $0 < \epsilon \leq 1$, 
\[ [S_{1-\tau+\epsilon}(r^s \mu(\cdot / r))(W_\tau)]^\beta \leq C (1-\tau+\epsilon)^{-\frac{d\beta}{\alpha}}.\]
Using this bound and changing variables, we obtain that
\begin{equation} \label{e_FK_derivbd}
w^{\lambda \psi_{\epsilon,r}}(1,x) \geq E^W_x (\psi_{\epsilon,r}(W_1)) \exp \left( -C(1+\beta)\lambda^\beta \int_0^1   (u+\epsilon)^{-\frac{ d\beta}{\alpha}} du \right). \nonumber
\end{equation}
Since $\beta < \frac \alpha d$ the integral remains bounded as $\epsilon \downarrow 0$. Hence for a new constant $C>0$, we obtain that for all $0< \epsilon \leq 1$,
\begin{equation} 
w^{\lambda \psi_{\epsilon,r}}(1,x) \geq E^W_x (\psi_{\epsilon,r}(W_1)) \exp \left( -C \lambda^\beta \right).
\end{equation}
We now integrate over $\lambda$ to obtain a lower bound for $u^{\psi_{\epsilon,r}}(1,x)$. Since $w^{\lambda \psi_{\epsilon,r}}(1,x) = \frac{\partial}{\partial \lambda} u^{\lambda \psi_{\epsilon,r}}(1,x)$, by \eqref{e_FK_FTC} and \eqref{e_FK_derivbd} we have
\begin{align} \label{e_heatunifbd_ineq}
u^{\psi_{\epsilon,r}}(1,x) &= \int_0^1 w^{\lambda \psi_{\epsilon,r}}(1,x)\, d\lambda \nonumber
\\ &\geq E^W_x (\psi_{\epsilon,r}(W_1)) \int_0^1 \exp \left( -C \lambda^\beta \right) d\lambda \nonumber
\\ &\geq c_0\, E^W_x (\psi_{\epsilon,r}(W_1)) 
\end{align}
for a constant $c_0>0$. It remains to show that the left and right hand sides of the above converge to the desired quantities when $\epsilon \downarrow 0$. It follows from Lemma~\ref{lemma_stability} that
\begin{equation} \label{e_heatunifbd_limit1}
\lim_{\epsilon \to 0} u^{\psi_{\epsilon,r}}(1,x)  = u^{r^s \mu(\cdot / r)}(1,x).
\end{equation}
Turning to the right hand side of \eqref{e_heatunifbd_ineq}, we first observe that $E^W_x(\psi_{\epsilon,r}(W_1)) = S_1\psi_{\epsilon,r}(x)$. Thus we may use the Dominated Convergence Theorem to see that 
\begin{equation}\label{e_heatunifbd_limit2}
\lim_{\epsilon \to 0} E^W_x(\psi_{\epsilon,r}) = S_1 (r^s \mu(\cdot/r))(x).
\end{equation}
Letting $\epsilon \downarrow 0$ in \eqref{e_heatunifbd_ineq}, from \eqref{e_heatunifbd_limit1} and \eqref{e_heatunifbd_limit2} we obtain that
\begin{equation} 
u^{r^s \mu(\cdot / r)}(1,x) \geq c_0 \,r^s S_1(\mu(\cdot/r))(x) \nonumber
\end{equation}
for all $x \in \R^d$, which completes the proof. \end{proof}

We now have all the tools we need to prove Theorem~\ref{thm_flatFrost}(a). Part (b) is proved in Section~\ref{s_trace}. 

\begin{proof}[Proof of Theorem~\ref{thm_flatFrost}(a)]
Fix $X_0 \in \cM_F(\R^d)$. Let $\mu \in \cM_F(\R^d)$ satisfy (F1)-$s$ as well as our assumptions that $\mu(1) = 1$ and $\mu(B_1^c) = 0$. By Lemma~\ref{lemma_measure_duality}, in particular using \eqref{e_Lapmeasure2} with $u^{\lambda \mu}_t$ and taking $\lambda \to \infty$, we obtain
\begin{equation} \label{e_LapProbFrost}
P^X_{X_0}(\mu(X_t) = 0 ) = \exp ( - X_0(u^{\infty \mu}_t) ).
\end{equation}
Let $\delta_n = 2^{-n}$. Assume that $n$ is large enough so that $\delta_n < t$, and consider the conditional probability $P(\mu(X_t) = 0 \, | \, \cF_{t-\delta_n})$. Applying the Markov property and using \eqref{e_LapProbFrost}, we obtain that
\begin{align} \label{e_probmuzero1}
P(\mu(X_t) = 0 \, | \, \cF_{t-\delta_n}) &= P^X_{X_{t-\delta_n}}(\mu(X_{\delta_n}
) = 0 ) \nonumber
\\&= \exp \left( - \int u^{\infty \mu}(\delta_n,x) dX_{t-\delta_n}(x)\right) \nonumber
\\ &\leq \exp \left( - \int_{B_2} u^{\infty \mu}(\delta_n,x) dX_{t-\delta_n}(x)\right),
\end{align}
where $B_2 = \{x \in \R^d: |x| \leq 2 \}$. Using monotonicity of  $\lambda \to u^{\lambda \mu}(\delta_n,x)$ and the scaling relationship \eqref{e_scalet}, we have
\begin{align} \label{e_uinf_lwrbd1}
u^{\infty \mu}(\delta_n,x) &\geq u^{\delta_n^{-\frac s\alpha -\frac{\alpha - \beta}{\alpha \beta}} \mu}(\delta_n,x) \nonumber
\\ &= \delta_n^{-\frac 1 \beta} u^{\delta_n^{-\frac s \alpha} \mu(\cdot / \delta_n^{-\frac 1\alpha})}(1,\delta_n^{-\frac 1 \alpha} x) \nonumber
\\ &\geq  c_{\ref{lemma_heatlwrbd_Frost}}\,\delta_n^{-\frac 1 \beta  -\frac s \alpha} S_1(\mu(\cdot/\delta_n^{-\frac 1 \alpha}))(\delta_n^{-\frac 1 \alpha}x).
\end{align}
The final inequality follows from Lemma~\ref{lemma_heatlwrbd_Frost}. We expand the semigroup term in the above as a convolution with $p_1$. After a change of variables, we have
\begin{align} \label{e_uinf_lwrbd2}
S_1(\mu(\cdot/\delta_n^{-\frac 1 \alpha}))(\delta_n^{- \frac 1 \alpha}x) &= \int p_1(\delta_n^{-\frac 1 \alpha}(x - y)) \,d\mu(y) \nonumber
\\ &\geq \mu(1) \, p_1( \delta_n^{-\frac 1 \alpha} d(x, \cS)),
\end{align}
where $\cS = \text{supp}(\mu)$ and $d(x,\cS) = \inf_{y \in \cS} |x-y|$, and we recall that for $\rho > 0$, $p_1(\rho)$ denotes $p_1(|z|)$ with $|z| = \rho$. Because $\mu$ is supported on $B_1$, for any $x \in B_2$ we have $d(x,\cS) \leq 3$. In particular, using this in \eqref{e_uinf_lwrbd2} and substituting it into \eqref{e_uinf_lwrbd1}, we obtain
\begin{equation}
u^{\infty \mu}(\delta_n,x) \geq c_{\ref{lemma_heatlwrbd_Frost}}\,\delta_n^{-\frac 1 \beta - \frac s \alpha} p_1( 3 \delta_n^{-\frac 1 \alpha}) \hspace{6 mm} \text{ for all } x \in B_2, \nonumber
\end{equation}
where $\mu(1)$ does not appear because it equals one. Using \eqref{heatkernelbds2} to bound $p_1$ below, we conclude that for a constant $c_1>0$,
\begin{align}\label{e_uinf_lwrbd3}
u^{\infty \mu}(\delta_n,x) &\geq c_1 \delta_n^{-\frac 1 \beta -\frac s \alpha}(\delta_n^{-\frac 1 \alpha})^{-(d+\alpha)} \nonumber
\\ &= c_1 \delta_n^{-q} \hspace{6 mm} \text{ for all } x \in B_2,
\end{align}
where $q := \frac 1 \beta - \frac{d-s + \alpha}{\alpha}$. Using \eqref{e_uinf_lwrbd3} in \eqref{e_probmuzero1}, we obtain the following:
\begin{align} \label{e_frostprobexp}
P(\mu(X_t) = 0 \, | \, \cF_{t-\delta_n}) &\leq \exp \left( - c_1 X_{t-\delta_n}(B_2) \,\delta_n^{-q}\right). 
\end{align}
From this point, the proof is identical to that of Proposition~\ref{prop_pointflat}. By instantaneous propagation, $X_t(B_2) > 0$ almost surely on $\{X_t \neq 0 \}$. One considers the event $A_{2\kappa} = \{X_t(B_2) \geq 2\kappa\}$ for $\kappa >0$ and notes that $X_{t-\delta_n}(B_2) \geq \kappa$ eventually a.s. on $A_{2\kappa}$. This leads to a statement analogous to \eqref{e_condprobexpsmall}. One then finishes the proof in the same way: by direct computation when $\beta < \beta^*(\alpha,s)$ and using martingale convergence when $\beta = \beta^*(\alpha,s)$. This completes the proof that $P^X_{X_0}(\mu(X_t) = 0 \, | \, X_t > 0)$. The result under $\N_x$ follows from \eqref{clustertransfer}, which implies that
\[ \N_x(\mu(X_t) = 0 \,| \, X_t \neq 0) = 0.\]
In particular, we have $\N_x( \mu(X_t) > 0 ) = \N_x( X_t \neq 0 ) = U_t$. Since $u^{\infty \mu}_t(x) = \N_x( \mu(X_t) > 0 )$, this proves the last claim.
\end{proof}

\section{Decay of $\N_x(\mu(X_t)> 0 )$ for (F2)-$s$ measures when $\beta > \beta^*(\alpha,s)$} \label{s_nonflat}

This section is concerned with establishing conditions under which $\N_x(\mu(X_t)>0) = u^{\infty \mu}_t(x)$ is non-flat (and hence $\mu(X_t) = 0$ with positive probability on $\{X_t \neq 0 \}$) and quantifying its asymptotic behaviour under these conditions. The main result we prove is Theorem~\ref{thm_nonflat}. The proofs are analytic and we pose it as an open problem to prove the same results using probabilistic arguments.

We will show that $u^{\infty \mu}_t = \N_x(\mu(X_t) > 0 )$ is non-flat when $\beta > \beta^*(\alpha,s) = \frac{\alpha}{d-s+\alpha}$, where $s \in [0,d]$ and $\mu$ has compact support and satisfies (F2)-$s$, which we recall is the condition that
\begin{itemize}
\item[](F2)-$s$ \hspace{2 mm} For some constant $\underline{C}>0$, for all $x \in \text{supp}(\mu)$ and $r\in(0,1]$,
\begin{equation}
\mu(B(x,r)) \geq \underline{C} r^s. \nonumber
\end{equation}
\end{itemize}

The method we use is to show the existence certain \textit{barrier functions} for the equation \eqref{e_pde}. A function $h: Q \to \R^+$ is a barrier function for $\mu$ if it is a super-solution to \eqref{e_pde} on $Q$ that explodes on $\text{supp}(\mu)$ with order $t^{-\frac 1 \beta}$ and vanishes on $\text{supp}(\mu)^c$ as $t\downarrow 0$. Our method is based on, and adapted from, a similar argument in \cite{CVW2016}. 

First, we define $W:\R^+ \to \R^+$ by 
\begin{equation} \label{e_Wdef}
W(r) = \frac{\log(e+r^2)}{1+r^{d+\alpha}}.
\end{equation}
We also introduce $V:\R^d \to \R^+$, given by
\begin{equation} \label{e_Vdef}
V(x) = W(|x|) = \frac{\log(e+|x|^2)}{1+|x|^{d+\alpha}}.
\end{equation}
For $(t,x) \in Q$, we then define $w_t(x)$ by
\begin{equation} \label{e_wtdef}
w_t(x) = t^{-\frac 1 \beta}(1 + t^{-\frac s \alpha}) V(t^{-\frac 1 \alpha}x).
\end{equation}
Finally, for $k > 0$ and $\mu \in \cM_F(\R^d)$, let $h_k(t,x)$ be given by
\begin{equation} \label{e_hlambdadef}
h_k(t,x) = k (w_t * \mu) (x).
\end{equation}
Note that $w_t \in C^{1,2}(Q)$, the space of functions which are once continuously differentiable in time and twice continuously differentiable in space. Consequently, we also have that $h_k \in C^{1,2}(Q)$. Recall that $\beta^*(\alpha,s) = \frac{\alpha}{(d-s) + \alpha}$. In what follows, we restrict to $s \in[0,\alpha)$, since this is required to have $\beta^*(\alpha,s) < \frac \alpha d$.

For closed $\cS \subset \R^d$, recall that $\cM_F(\cS)$ is the space of measures $\mu \in \cM_F(\R^d)$ with $\text{supp}(\mu) \subseteq \cS$, and that $d(x,\cS) = \inf_{y \in \cS} |x-y|$.

\begin{proposition} \label{prop_linearenvelope}
Suppose that $\mu \in \cM_F(\R^d)$ satisfies (F2)-$s$ for some $s \in [0,\alpha)$ and has compact support $\cS \subset \R^d$. Let $\beta^*(\alpha,s) < \beta < \frac \alpha d$. \\
(a) There exists $\Lambda_0 > 0$ such that if $k \geq \Lambda_0$, $h_k$ is a (strong) supersolution to \eqref{e_pde} on $Q$, in the sense that for all $(t,x) \in Q$,
\begin{equation}
(\partial_t - \Delta_\alpha)h_k(t,x) + h_k(t,x)^{1+\beta} \geq 0.
\end{equation}
(b) For $x \in \cS$ and $t>0$,
\begin{equation} \label{e_prop_linEnv_bd1}
h_k(t,x) \geq  c_{\ref{e_prop_linEnv_bd1}} k t^{-\frac 1 \beta},
\end{equation} 
where $c_{\ref{e_prop_linEnv_bd1}} = \underline{C} \cdot c_0$ and $c_0>0$ depends only on $(\alpha,d)$. For all $(t,x) \in Q$, 
\begin{equation} \label{e_prop_linEnv_bd2}
h_k(t,x) \leq k \mu(1) [t^{-\frac 1 \beta - \frac s \alpha} \vee t^{-\frac 1 \beta}] W(t^{-\frac 1 \alpha}d(x,\cS)).
\end{equation}
In particular, $\lim_{t\to 0} h_k(t,x)= \infty$ for $x \in \cS$, and $h_k(t,\cdot)$ vanishes uniformly on $\{x: d(x,\cS) \geq \rho\}$ as $t\downarrow 0$ for all $\rho >0$.\\
(c) For $(t,x) \in Q$ we have
\begin{equation} \label{e_prop_linEnv_bd3}
h_k(t,x) \geq  k c_{\ref{e_prop_linEnv_bd3}}  t^{-\frac 1\beta} W(t^{-\frac 1 \alpha}d(x,\cS)),
\end{equation}
where $c_{\ref{e_prop_linEnv_bd3}} = \underline{C}\cdot c_1$ and $c_1>0$ depends only on $(\alpha,d)$.\\
(d) For any $\nu \in \cM_F(\cS)$, if $k \geq \Lambda_0$ then $h_k(t,x) \geq u^\nu(t,x)$ on $Q$.
\end{proposition}

This proposition is the main result underlying Theorem~\ref{thm_nonflat} (as well as Theorem~\ref{thm_inittrace}(b)). Before proving it, we comment on the technique. By and large, our method is adapted from the argument used by Chen, Veron and Wang in \cite{CVW2016} to prove the result we called Theorem B(b) in the introduction. Our barrier function is modelled after theirs and we make use of some of their intermediate results. Define $\tilde{w}_t$ by
\begin{equation}
\tilde{w}_t(x) = t^{-\frac 1 \beta} W(t^{-\frac 1 \alpha}|x|). \nonumber
\end{equation}
In \cite{CVW2016}, it is shown that $k \tilde{w}_t(x)$ is a supersolution to \eqref{e_pde} for sufficiently large $k$. This barrier function is what the authors use to prove that $\lim_{\lambda \to \infty} u^{\lambda \delta_0}(t,x)$ is non-flat when $\beta > \frac{\alpha}{d + \alpha}$. Part of their proof was a detailed analysis of $-\Delta_\alpha V$. In equation (5.11) of \cite{CVW2016}, the following bound is established: there is a constant $c_1>0$ such that for $x \in \R^d$ with $|x| \geq 2$,
\begin{equation} \label{e_VfracLapbd}
-\Delta_\alpha V(x) \geq - \frac{c_1}{1+|x|^{d+ \alpha}}.
\end{equation}
This bound is critical to their argument and it is equally critical in ours which follows. As can be seen from \eqref{e_hlambdadef}, the function $h_k(t,x)$ is essentially $\tilde{w_t}$ spread out over $\cS = \text{supp}(\mu)$ via a convolution with $\mu$, with an additional power of $t$ to locally normalize mass of $\mu$ when $t \leq 1$. By spreading out $\tilde{w}_t$ over $\cS$, we construct a supersolution which is singular on $\cS$ as $t \downarrow 0$.

We make a few observations about the functions we have introduced. The function $W(z)$ is not globally decreasing for positive $z$, and correspondingly $V$ and $w_t$ are not globally radially decreasing. However, for any $d\geq 1$ and $\alpha \in (0,2)$, $W$ attains its maximum value at some $r_0 \in [0,1)$ and $W(r)$ is decreasing for $r \geq r_0$. Furthermore, one can verify that $\min_{r \in [0,1]} W(r) = W(1)$, and so for all $d\geq 1$ and $\alpha \in (0,2)$, $W$ is weakly decreasing in the sense that
\begin{equation}\label{e_Wweakdec1}
\min_{r' \in [0,r]} W(r') =W(r) \,\, \text{ for all $r \geq 1$.} 
\end{equation} 
$V$ inherits this as a form of weak radial decreasing, i.e.
\begin{equation} \label{e_Wweakdec0}
\min_{x \in B(0,R)} V(x) = W(R) \,\, \text{ for all $R \geq 1$.}
\end{equation} 
Finally, one can show that there is a constant $c_{\ref{e_Wweakdec2}} >0$ such that 
\begin{equation} \label{e_Wweakdec2}
\text{For all } R\geq 0,\,\,\, \sup_{|x| \geq R} V(x) \leq c_{\ref{e_Wweakdec2}} W(R).
\end{equation}

\begin{proof}[Proof of Proposition~\ref{prop_linearenvelope}(a)] First let us consider the time derivative of $w_t(x)$. Expanding directly using \eqref{e_Vdef} and \eqref{e_wtdef}, for $z = t^{-\frac 1 \alpha}|x|$ we have
\begin{align} \label{e_dt1}
\partial_t(w_t(x)) &= -\frac 1 \beta t^{-\frac 1 \beta - 1} W(t^{-\frac 1 \alpha}|x|) + \left(-\frac 1 \beta - \frac s \alpha \right)t^{-\frac 1 \beta - \frac s \alpha - 1} W(t^{-\frac 1 \alpha}|x|)  \nonumber
\\ &\hspace{10 mm} - \frac 1 \alpha t^{-\frac 1 \beta - 1}(1 + t^{-\frac s \alpha}) (t^{\frac 1 \alpha}|x|)W'(t^{-\frac 1 \alpha}|x|) \nonumber
\\ &= t^{-\frac 1 \beta - \frac s \alpha - 1} \left[ \left(-\frac 1 \beta - \frac s \alpha \right)W(z) - \frac 1 \alpha zW'(z) \right] \nonumber
\\ &\hspace{10 mm} +t^{-\frac 1 \beta - 1} \left[-\frac 1 \beta W(z) - \frac 1 \alpha zW'(z) \right].
\end{align}
Computing $W'$ directly, we obtain that 
\begin{align}
W'(z) = \frac{2z}{e+z^2} \frac{1}{1+z^{d+\alpha}} - (d+\alpha) \frac{\log(e+z^2) z^{d+\alpha - 1}}{(1+z^{d+\alpha})^2} \nonumber
\end{align}
From \eqref{e_dt1}, it follows that
\begin{align}
\partial_t(w_t(x)) &= t^{-\frac 1 \beta - \frac s \alpha - 1} W(z) \left[ \left(-\frac 1 \beta - \frac s \alpha\right) -\frac 1 \alpha \frac{2z^2(e+z^2)^{-1}}{\log(e+z^2)} + \frac{d+\alpha}{\alpha} \frac{z^{d+\alpha}}{1+z^{d+\alpha}} \right] \nonumber
\\ &\hspace{10 mm}+ t^{-\frac 1 \beta - 1} W(z) \left[-\frac 1 \beta -\frac 1 \alpha \frac{2z^2(e+z^2)^{-1}}{\log(e+z^2)} + \frac{d+\alpha}{\alpha} \frac{z^{d+\alpha}}{1+z^{d+\alpha}} \right] \nonumber
\\ &= t^{-\frac 1 \beta - \frac s \alpha - 1}W(z)f_1(z) + t^{-\frac 1 \beta - 1}W(z)f_2(z). \label{e_dt11}
\end{align}
Consider $f_1(z)$, the square bracketed quantity in the first line. The second term in $f_1(z)$ vanishes as $z \to \infty$, and the third converges to $\frac{d+\alpha}{\alpha}$. The assumption $\beta > \frac{\alpha}{d-s+\alpha}$ is equivalent to $\frac 1 \beta + \frac s \alpha < \frac{d+\alpha}{\alpha}$; it follows that $f_1(z)$ is positive for sufficiently large $z$. Moreover, there are constants $R_0 > 0$ and $c_2 > 0$ such that the for $z \geq R_0$, $f_1(z) \geq c_2$. Since $f_2(z) > f_1(z)$ for all $z \geq 0$, we also have $f_2(z) \geq c_2$ for $z \geq R_0$, and from \eqref{e_dt11} we have the following: for all $(t,x)$ satisfying $t^{-\frac 1 \alpha}|x| \geq R_0$,
\begin{align} \label{e_dt2}
\partial_t(w_t(x)) \geq c_2 t^{-\frac 1 \beta - 1}(1+t^{-\frac s \alpha}) W(t^{-\frac 1 \alpha}|x|).
\end{align}
Next we consider $\Delta_\alpha w_t$. By the scaling of the $\alpha$-stable process, if $g^\lambda(x) = g(\lambda x)$, then $\Delta_\alpha g^\lambda (x) = \lambda^\alpha (\Delta_\alpha g)(\lambda x)$. Using this and \eqref{e_wtdef}, it follows that 
\begin{equation} \label{e_fracLap0} 
\Delta_\alpha w_t(x) = t^{-\frac 1 \beta -1} (1 + t^{ - \frac s \alpha }) (\Delta_\alpha V)(t^{-\frac 1 \alpha}x). 
\end{equation}
Using the above and \eqref{e_VfracLapbd}, for all $(t,x)$ such that $t^{-\frac 1 \alpha}|x| \geq 2$,
\begin{equation} \label{e_fracLap1}
-\Delta_\alpha w_t(x) \geq - c_1 t^{-\frac 1 \beta - 1}(1+t^{-\frac s \alpha}) \frac{1}{1+|t^{-\frac 1 \alpha} x|^{d+ \alpha}}.
\end{equation}
This bound allows a direct comparison with $\partial_t(w_t(x))$. In particular, by \eqref{e_dt2} and \eqref{e_fracLap1}, (and recalling the definition of $W$ from \eqref{e_Wdef}) we have, for $|t^{-\frac 1 \alpha}x| \geq 2 \vee R_0$,
\begin{equation}
(\partial_t -\Delta_\alpha) w_t(x) \geq \frac{t^{-\frac 1 \beta  - 1}(1+t^{-\frac s \alpha})}{1+|t^{-\frac 1 \alpha} x|^{d+ \alpha}} \left[ c_2 \log(e+|t^{-\frac 1 \alpha}x|^2) - c_1 \right]. \nonumber
\end{equation}
It follows that for some $R_1 \geq R_0 \vee 2$, 
\begin{equation} \label{e_dtfracLapbd1}
(\partial_t -\Delta_\alpha) w_t(x) \geq 0 \,\, \text{ for all $(t,x)$ satisfying $|t^{-\frac 1 \alpha}x| \geq R_1$.}
\end{equation}
Now consider $h_k(t,x)$. We can take differentiation under the integral in \eqref{e_hlambdadef} to obtain
\begin{align} \label{e_dtfracLapbd2}
(\partial_t -\Delta_\alpha)h_k(t,x) = k \int (\partial_t - \Delta_\alpha)w_t(x-y) d\mu(y).
\end{align}
Recall that $d(x,\cS)$ denotes the distance from $x \in \R^d$ to the set $\cS$. If $t^{-\frac 1 \alpha} d(x,\cS) \geq R_1$, from \eqref{e_dtfracLapbd1} the integrand in \eqref{e_dtfracLapbd2} is positive for all $y \in \cS$, i.e. for all $y \in \text{supp}(\mu)$, and hence
\begin{equation} \label{e_dtfracLapbd22}
(\partial_t -\Delta_\alpha) h_k(t,x) \geq 0 \,\, \text{ for all $(t,x)$ satisfying $t^{-\frac 1 \alpha} d(x,\cS) \geq R_1$.} 
\end{equation}
The condition on $(t,x)$ that $t^{-\frac 1 \alpha} d(x,\cS) \geq R_1$ will be important, so we introduce
\[Q^{\geq R_1} =  \{(t,x) \in Q: t^{-\frac 1 \alpha}d(x,\cS) \geq R_1\}.\]
The statement \eqref{e_dtfracLapbd22} then reads that
\begin{equation} \label{e_dtfracLapbd3}
(\partial_t - \Delta_\alpha) h_k(t,x) \geq 0 \,\, \text{ for all $(t,x) \in Q^{\geq R_1}$.} 
\end{equation}
We also introduce
\[Q^{<R_1} = \{(t,x) \in Q: t^{-\frac 1 \alpha}d(x,\cS) < R_1\},\]
and now consider the behaviour of $(\partial_t - \Delta_\alpha) h_k$ on $Q^{< R_1}$. We can apply \eqref{e_dtfracLapbd1} to the integrand in \eqref{e_dtfracLapbd2} to obtain that
\begin{align} \label{e_dtfracLapbd4}
(\partial_t - \Delta_\alpha)h_k(t,x) \geq k \int_{B(0,t^{\frac 1 \alpha} R_1)} (\partial_t - \Delta_\alpha)w_t(y) d\mu(y-x).
\end{align}
By \eqref{e_dt1} and \eqref{e_fracLap0}, for $u = t^{-\frac 1 \alpha}x$ we have
\begin{align}
(\partial_t - \Delta_\alpha)w_t(y) &= t^{-\frac 1 \beta -\frac s \alpha -1} \left[ \left(-\frac 1 \beta - \frac s \alpha \right)W(|u|) - \frac 1 \alpha |u|W'(|u|) - \Delta_\alpha V(u)\right],  \nonumber
\\ &\hspace{4 mm} +t^{-\frac 1 \beta -1} \left[-\frac 1 \beta W(|u|) - \frac 1 \alpha |u|W'(|u|) - \Delta_\alpha V(u)\right]. \nonumber
\end{align}
Since all the terms in above are continuous functions of $u$ and $|u| \leq R_1$ on $\{y : |y| \leq t^{\frac 1 \alpha}R_1\}$, the square-bracketed terms above are bounded on this set. Thus
\begin{equation} 
\sup_{y:|y|\leq t^{\frac 1 \alpha}R_1} |(\partial_t - \Delta_\alpha w_t(y)| \leq c_3 t^{-\frac 1 \beta - 1} [t^{-\frac s \alpha} \vee 1] \nonumber
\end{equation} 
for some $0< c_3 < \infty$. Using this in \eqref{e_dtfracLapbd4} we obtain that
\begin{align} \label{e_smallvalworstcase}
(\partial_t - \Delta_\alpha)h_k(t,x) &\geq - k c_3 t^{-\frac 1 \beta - 1} \left[ \frac{\mu(B(x,t^{\frac 1 \alpha}R_1))}{t^{\frac s \alpha} \wedge 1} \right]
\end{align}
for all $(t,x) \in Q$. We now must show that the non-linear term in \eqref{e_pde}, given by $h_k(t,x)^{1+\beta}$, is sufficiently large on $Q^{<R_1}$ so that $h_k(t,x)$ is a super-solution to \eqref{e_pde} even in the case of the worst-case bound given in \eqref{e_smallvalworstcase}. From \eqref{e_wtdef} and \eqref{e_hlambdadef}, we have
\begin{align}
\frac{1}{k^{1+\beta}}h_k(t,x)^{1+\beta} &= \left[\int w_t(x-y) d\mu(y) \right]^{1+\beta} \nonumber
\\ &= \left[t^{-\frac 1 \beta}(1+ t^{ - \frac s \alpha}) \int V(t^{-\frac 1 \alpha}(x-y)) d\mu(y) \right]^{1+\beta} \nonumber
\\ &\geq  \left[ (t^{-\frac 1 \beta  - \frac s \alpha} \vee t^{-\frac 1 \beta} )\int_{B(x,t^{\frac 1 \alpha}(R_1+1))} V(t^{-\frac 1 \alpha}(x-y)) d\mu(y) \right]^{1+\beta} \nonumber
\end{align}
where the final inequality holds because $V \geq 0$. Restricted to $B(x,t^{\frac 1 \alpha}(R_1+1))$, $V(t^{-\frac 1 \alpha}(x-y))$ is bounded below by $W(R_1 + 1)$ by \eqref{e_Wweakdec0}. This implies that
\begin{align} \label{e_smallvalNL} 
\frac{1}{k^{1+\beta}} h_k(t,x)^{1+\beta} &\geq c_4 t^{-\frac 1 \beta - 1} \left[ \frac{\mu(B(x,t^{\frac 1 \alpha}(R_1 + 1)))}{t^{\frac s \alpha} \wedge 1} \right]^{1+\beta} 
\end{align}
on $Q^{<R_1}$, where $c_4 = W(R+1)^{1+\beta}$. Combining \eqref{e_smallvalworstcase} and \eqref{e_smallvalNL}, we obtain that for $x \in Q^{<R_1}_T$,
\begin{align} \label{e_smallvalshowdown1}
&(\partial_t - \Delta_\alpha) h_k(t,x) + h_k(t,x)^{1+\beta} \nonumber
\\ &\hspace{10 mm} \geq k t^{-\frac 1 \beta - 1} \left( c_4 k^\beta \left[ \frac{\mu(B(x,t^{\frac 1 \alpha}(R_1 + 1)))}{t^{\frac s \alpha} \wedge 1} \right]^{1+\beta}   - c_3 \left[ \frac{\mu(B(x,t^{\frac 1 \alpha}R_1))}{t^{\frac s \alpha} \wedge 1} \right] \right) \nonumber
\\ &\hspace{10 mm} \geq k t^{-\frac 1 \beta - 1} \left[ \frac{\mu(B(x,t^{\frac 1 \alpha}R_1))}{t^{\frac s \alpha} \wedge 1} \right] \left( c_4 k^\beta \left[ \frac{\mu(B(x,t^{\frac 1 \alpha}(R_1 + 1)))}{t^{\frac s \alpha} \wedge 1} \right]^{\beta}   - c_3\right). 
\end{align}
Since $(t,x) \in Q^{<R_1}$, there must be a point $y_0 \in \cS$ such that $B(y_0, t^{\frac 1 \alpha}) \subset B(x,t^{\frac 1 \alpha}(R_1 + 1))$. In particular, since $\mu$ satisfies (F2)-$s$ we have
\begin{align} 
\mu(B(x,t^{\frac 1 \alpha}(R_1 + 1))) \geq \mu(B(y_0,t^{\frac 1 \alpha})) \geq \underline{C} [t^{\frac s \alpha} \wedge 1]. \nonumber
\end{align}
The minimum above appears since for $y_0 \in \cS$ and $t \geq 1$, $\mu(B(y_0, t^{\frac 1 \alpha})) \geq \mu(B(y_0,1)) \geq \underline{C}$. Using this in \eqref{e_smallvalshowdown1}, we obtain 
\begin{align} 
(\partial_t - \Delta_\alpha) h_k(t,x) - h_k(t,x)^{1+\beta} \geq k t^{-\frac 1 \beta - 1} \left[ \frac{\mu(B(x,t^{\frac 1 \alpha}R_1))}{t^{\frac s \alpha} \wedge 1} \right]  \left( c_4' k^\beta - c_3 \right) \nonumber
\end{align}  
for all $(t,x) \in Q^{<R_1}$, where $c_4' = \underline{C}^\beta c_4$. For sufficiently large $k$ the above is non-negative, and hence $h_k(t,x)$ is a super-solution to \eqref{e_pde} on $Q^{<R_1}$. On the other hand, because $h_k(t,x)^{1+\beta} \geq 0$, \eqref{e_dtfracLapbd3} implies that $h_k(t,x)$ is a super-solution on $Q^{\geq R_1}$. Thus we have shown that for sufficiently large $k$, $h_k(t,x)$ is a super-solution to \eqref{e_pde} on $Q^{<R_1} \cup Q^{\geq R_1} = Q$.
\end{proof}

\begin{proof}[Proof of Proposition~\ref{prop_linearenvelope}(b)-(d)]
We first show part (c). For fixed $x$, since $\cS$ is closed, there is a point $y_0 \in \cS$ such that $|x-y_0| = d(x,\cS)$. Hence from \eqref{e_hlambdadef},
\begin{align*}
h_k(t,x) &\geq k \int_{B(y_0,t^{\frac 1 \alpha})} w_t(x-y) d\mu(y)
\\ & = k t^{-\frac 1 \beta}(1+t^{ - \frac s \alpha}) \int_{B(y_0,t^{\frac 1 \alpha})} W(t^{-\frac 1 \alpha}|x-y|) d\mu(y)
\\ &\geq k [t^{-\frac 1 \beta} \vee t^
{-\frac 1 \beta - \frac s \alpha}] W(t^{-\frac 1 \alpha}d(x,\cS) + 1) \mu(B(y_0,t^{\frac 1 \alpha})),
\end{align*}
where the last line has used the triangle inequality and \eqref{e_Wweakdec1}. Since $\mu$ satisfies (F2)-$s$, $\mu(B(y_0,t^{\frac 1 \alpha})) \geq \underline{C}t^{\frac s \alpha}$ for $t \leq 1$ and $\mu(B(y_0,t^{\frac 1 \alpha})) \geq \underline{C}$ for $t > 1$. This implies that above is bounded below by $k \underline{C}t^{-\frac 1 \beta } W(t^{-\frac 1 \alpha}d(x,\cS) + 1)$. The result then follows from \eqref{e_Wweakdec2}.

Next we prove part (b). The claim for $x \in \cS$ follows from part (c). Now let $x \in \R^d$. Applying \eqref{e_Wweakdec2}, we obtain that for every $y \in \cS$,
\[ w_t(x-y) \leq c_{\ref{e_Wweakdec2}} [t^{-\frac 1 \beta - \frac s \alpha} \vee t^{-\frac 1 \beta}] W(t^{-\frac 1 \alpha}d(x,\cS)). \]
One then uses this bound directly in the convolution defining $h_k(t,x)$, i.e. \eqref{e_hlambdadef}, to obtain \eqref{e_prop_linEnv_bd2}. Using formula \eqref{e_Wdef} for $W$ and the fact that $\frac 1 \beta + \frac s \alpha < \frac{d+\alpha}{\alpha}$, the uniform convergence as $t\downarrow 0$ follows.

We now prove (d). Suppose that $\nu \in \cM_F(\cS)$. (This includes the case $\nu = \mu$.) In order to show that $u^\nu(t,x) \leq h_k(t,x)$, we will need to consider a sequence of solutions corresponding to a sequence of approximations of $\nu$. Let $Z>0$ be the normalizing constant such that
\[Z^{-1} \int V(x) dx = 1.\]
Let $\tilde{V}(\cdot) = Z^{-1}V(\cdot)$ and define a sequence of $C^2$ mollifiers by $\phi_n(x) = n^{d} \tilde{V}(nx)$ for $n \geq 1$. Then it is immediate that $\nu * \phi_n \to \nu$ in the weak sense of measures as $n \to \infty$. Let $u_n(t,x) = u^{\nu * \phi_n}(t,x)$. By Lemma~\ref{lemma_stability}(a),
\begin{equation} \label{mollifystab}
\lim_{n \to \infty} u_n(t,x) = u^{\nu}(t,x)
\end{equation}
for all $(t,x) \in Q$. For $x \in \R^d$,
\begin{align} 
\nu * \phi_n(x) &= \frac 1 Z n^d \int W(n|y-x|)d\nu(y) \nonumber
\\ &\leq \frac 1 Z \nu(1) n^d \sup_{y \in \cS} W(n|x-y|) \nonumber
\end{align}
Hence from \eqref{e_Wweakdec2} we obtain that
\begin{equation} \label{comparisonprinc1}
\nu * \phi_n(x) \leq \frac{c_{\ref{e_Wweakdec2}}\nu(1)}{Z}  [n^d W(nd(x,\cS))].
\end{equation}
Let $t_n = n^{-\alpha}$. By \eqref{e_prop_linEnv_bd3}, we have
\begin{align} \label{comparisonprinc2}
h_k(t_n,x) &\geq c_{\ref{e_prop_linEnv_bd3}} k n^{\frac \alpha \beta} W(nd(x,\cS)) \nonumber
\\ &= c_{\ref{e_prop_linEnv_bd3}}k  n^{\frac \alpha \beta - d} [n^d W(nd(x,\cS))]
\end{align}
By part (a), if $k \geq \Lambda_0$, then $(t,x) \to h_{k}(t_n + t,x)$ is a super-solution to \eqref{e_pde} on $Q$ with initial value $h_{k}(t_n,\cdot)$. Since $\frac \alpha \beta > d$, it follows from \eqref{comparisonprinc1} and \eqref{comparisonprinc2} that, for sufficiently large $n$,
\[\nu * \phi_n(x) \leq h_k(t_n,x).\]
Since $u_n$ has initial data $\nu * \phi_n$ and $(t,x) \to h_{k}(t_n + t,x)$ is a super-solution, the above and the comparison principle imply that
\begin{equation}
u_n(t,x) \leq h_k(t_n+t,x) \nonumber
\end{equation}
for all $(t,x) \in Q$. Taking $n \to \infty$ on both sides and using \eqref{mollifystab}, we obtain that
\[u^\nu(t,x) \leq h_k(t,x)\]
for all $(t,x) \in Q$.
\end{proof}

The final ingredient in the proof of Theorem~\ref{thm_nonflat} is the following lower bound on $u^{\infty \mu}_t(x)$. We recall that $u^\infty_t(x) = \lim_{\lambda \to \infty} u^{\lambda \delta_0}_t(x)$.
\begin{lemma} \label{lemma_umuinf_lwrdbd} Let $\mu \in \cM_F(\R^d)$. Then for every $z \in \text{supp}(\mu)$,
\begin{equation}
u^{\infty \mu}_t(x) \geq  u^\infty_t(x-z) \geq c_{\ref{lemma_umuinf_lwrdbd}} t^{-\frac 1 \beta} p_1(t^{-\frac 1 \alpha}(x-z)) \nonumber
\end{equation}
for all $(t,x) \in Q$, where $c_{\ref{lemma_umuinf_lwrdbd}}>0$ depends only on $(\alpha,\beta,d)$.
\end{lemma}
We defer the proof to the end of the section.

\begin{proof}[Proof of Theorem~\ref{thm_nonflat}] Fix $s \in [0,\alpha)$ and let $\mu$ satisfy (F2)-$s$ and have compact support $\cS$. Let $\beta^*(\alpha,s) < \beta < \frac \alpha d$. We begin with part (b). We need to show that the upper bound in \eqref{e_thmnonflat_upperbound} holds, i.e. for some constant $C>0$,
\begin{equation} 
u^{\infty \mu}_t(x) \leq C [t^{-\frac 1 \beta - \frac s \alpha} \vee t^{-\frac 1 \beta}] W(t^{-\frac 1 \alpha}d(x,\cS)) \nonumber
\end{equation}
on $Q$. Let $\Lambda_0>0$ be as in Proposition~\ref{prop_linearenvelope}(a). Then by Proposition~\ref{prop_linearenvelope}(d), $u^{\lambda \mu}_t(x) \leq h_{\Lambda_0}(t,x)$ for every $\lambda > 0$, and hence $u^{\infty \mu}_t(x) = \lim_{\lambda \to \infty} u^{\lambda \mu}_t(x) \leq h_{\Lambda_0}(t,x)$. The above bound then follows from Proposition~\ref{prop_linearenvelope}(b) with $C = \mu(1) \Lambda_0$, which proves the upper bound in \eqref{e_thmnonflat_upperbound}. The lower bound in \eqref{e_thmnonflat_upperbound} follows from Lemma~\ref{lemma_umuinf_lwrdbd}.

Next we prove part (a). From the upper bound in \eqref{e_thmnonflat_upperbound}, it is clear that for $t>0$, $u^{\infty \mu}_t(x) \to 0$ as $d(x,\cS) \to \infty$. We can thus take $x_0 \in \R^d$ so that $u^{\infty \mu}_t(x_0) < U_t$, which is equivalent to \linebreak$\N_{x_0}(\mu(X_t) = 0 \, | \, X_t \neq 0 ) > 0$. By \eqref{clustercanon}, we have $P^X_{\delta_{x_0}}(\mu(X_t) = 0 \, | \, X_t \neq 0 ) >0$. Appealing to \eqref{mutual_absolute_continuity}, it follows that \[P^X_{X_0}(\mu(X_t) = 0 \, | \, X_t \neq 0 ) >0\] for every $X_0 \in \cM_F(\R^d)$. In particular, we can take $X_0 = \delta_x$ and use \eqref{clustercanon} again to see that $\N_x(\mu(X_t) = 0 \, | \, X_t \neq 0)>0$ for every $x \in \R^d$. This completes the proof of (a). 

To see that part (d) holds, recall that Proposition~\ref{prop_linearenvelope}(d) applies to any $\nu \in \cM_F(\cS)$. The upper bound in \eqref{e_thmnonflat_upperbound_nu} then follows by the same argument used to prove the upper bound in \eqref{e_thmnonflat_upperbound} above. The argument used to prove (a) can then be used to prove that the claims from (a) hold when $\nu$ is replaced with $\mu$. Finally, as in the proof of part (b), the lower bound in \eqref{e_thmnonflat_upperbound_nu} follows from Lemma~\ref{lemma_umuinf_lwrdbd}.

It remains to show part (c). Recall from \eqref{e_scaletinf} that
\begin{equation}
u_t^{\infty \mu}(x) = t^{-\frac 1 \beta}u^{\infty \mu(\cdot / t^{-\frac 1 \alpha})}_1(t^{-\frac 1 \alpha} x).  \nonumber
\end{equation}
If $\mu$ satisfies (F1)-$s$, for $t \leq 1$ we apply Lemma~\ref{lemma_heatlwrbd_Frost} to the right hand side of the above to obtain 
\begin{align} 
u_t^{\infty \mu}(x) &\geq c_{\ref{lemma_heatlwrbd_Frost}} t^{-\frac 1 \beta - \frac s \alpha} \int p_1(y-t^{-\frac 1 \alpha}x) d\mu \big( y /t^{-\frac 1 \alpha} \big) \nonumber
\\ &= c_{\ref{lemma_heatlwrbd_Frost}} t^{-\frac 1 \beta - \frac s \alpha} \int p_1(t^{-\frac 1 \alpha}(y-x)) d\mu(y). \nonumber
\end{align}
For every $y \in \cS$, $|y-x| \leq d(x,\cS) + \text{diam}(\cS)$, so from the above, for $x \in \R^d$ and $t \leq 1$ we have
\begin{equation}
u_t^{\infty \mu}(x)  \geq c_{\ref{lemma_heatlwrbd_Frost}} \mu(1) t^{-\frac 1 \beta - \frac s \alpha}p_1(t^{-\frac 1 \alpha}(d(x,\cS) + \text{diam}(\cS))). \nonumber
\end{equation}
The lower bound on $p_1$ from \eqref{heatkernelbds} then implies \eqref{e_thmnonflat_lowerbound}. \end{proof}

\begin{proof}[Proof of Lemma~\ref{lemma_umuinf_lwrdbd}] Let $\mu \in \cM_F(\R^d)$ and $z \in \text{supp}(\mu)$. Then $\mu(B(z,\rho)) > 0$ for every $\rho > 0$. For $k >0$ and $\rho >0$, we have
\[ \liminf_{t \to 0} \int_{B(z,\rho)} u^{k \mu}_t(x)\, dx \geq k \mu(B(z,\rho /2)).\]
Since $u^{\infty \mu}_t \geq u^{k \mu}_t$ for all $k >0$, it follows that 
\begin{equation} \label{u_muinf_sing}
\lim_{t \to 0} \int_{B(z,\rho)} u^{\infty \mu}_t(x)\, dx = + \infty \,\,\, \text{  for all $z \in \text{supp}(\mu)$ and $\rho > 0$.}
\end{equation}
We now fix $\lambda > 0$ and $z \in \text{supp}(\mu)$. By \eqref{u_muinf_sing}, there exists $(t_n,\rho_n)$ such that $t_n, \rho_n \to 0$ and
\begin{equation}
\int_{B(z,\rho_n)} u^{\infty \mu}_{t_n}(x)\,dx = \lambda \nonumber
\end{equation}
for all $n \geq 1$. Let $\phi_n(x) = u^{\infty \mu}_{t_n}(x) 1_{B(z,\rho_n)}(x)$. Note that $\phi_n \to \lambda \delta_0(\cdot - z)$ in the weak sense of measures as $n \to \infty$, so by Lemma~\ref{lemma_stability}(a), $ \lim_{n \to \infty} u^{\phi_n}_t(x) = u^{\lambda}_t(x-z)$. By \eqref{e_Lapcanon}, we have
\begin{equation} \label{muinflwr1}
u^{\phi_n}_t(x) = \N_x( 1 - \exp( X_t(\phi_n))).
\end{equation}
On the other hand, we remark that for every $k >0$, by \eqref{e_Lapcanonmeasure2} and the Markov property,
\begin{align} 
u^{k \mu}_{t+t_n}(x) &= \N_x(1 - \exp (k \mu(X_{t+t_n})) \nonumber
\\ &= \N_x(1 - \exp(X_t(u^{k \mu}_{t_n})). \nonumber
\end{align}
Taking $k \to \infty$, we obtain
\begin{equation} \label{muinflwr2}
u^{\infty \mu}_{t+t_n}(x) = \N_x(1 - \exp(X_t(u^{\infty \mu}_{t_n})).
\end{equation}
By the definition of $\phi_n$, we have $\phi_n \leq u^{\infty \mu}_{t_n}$ for all $n \in \N$. Hence, from \eqref{muinflwr1} and \eqref{muinflwr2}, we have
\[ u^{\phi_n}_t(x) \leq u^{\infty \mu}_{t + t_n}(x)\]
for all $(t,x) \in Q$ and all $n \in \N$. As noted above, the left hand side converges to $u_t^\lambda(x-z)$ as $n \to \infty$. Taking $n \to \infty$, we obtain that
\[ u^\lambda_t(x-z) \leq u^{\infty \mu}_t(x).\]
Taking $\lambda \to \infty$, we obtain $u^\infty_t(x-z) \leq u^{\infty \mu}_t(x)$. This proves the first inequality in the lemma, and the second then follows from \eqref{e_uinf_lwrbound}. \end{proof}


\section{The initial trace problem} \label{s_trace}
We now apply the results of Sections \ref{s_flat} and \ref{s_nonflat} to the initial trace theory for \eqref{e_pde}. Recall that the initial trace of a solution $u_t(x)$ to \eqref{e_pde} was defined in \eqref{e_inittrace_measure} and \eqref{e_inittrace_singular}. We restate the definition here for convenience. A pair $(\cS,\nu)$ with $\cS \subset \R^d$ a closed set and $\nu$ a Radon measure with $\nu(\cS) = 0$ is the \textit{initial trace} of $u_t(x)$ if:

\begin{itemize} 
\item For all $\xi \in  C_c(\cS^c)$.	
\begin{equation} \label{e_inittrace_measure2}
\lim_{t \to 0} \int \xi(x) u_t(x) dx = \int \xi d\nu.
\end{equation}
\item For every $z \in \cS$ and $\rho > 0$,
\begin{equation} \label{e_inittrace_singular2}
\lim_{t \to 0} \int_{B(z,\rho)} u_t(x) dx = + \infty.
\end{equation}
\end{itemize}
Our contribution is to the problem of determining when a solution with a given initial trace exists. We consider weak solutions. First, recall that we have defined a weak solution to \eqref{e_pde} in Definition~\ref{def_weaksol}. Although we only consider the problem for initial traces in which the regular component (i.e. the Radon measure) is null, our definition applies for general initial traces.

\begin{definition} \label{def_weaksoltrace}
For closed $\cS \subset \R^d$ and a Radon measure $\nu$ satisfying $\nu(\cS) = 0$, we say that $u:Q \to [0,\infty)$ is a weak solution to the initial trace problem with initial trace $(\cS,\nu)$ if:
\begin{itemize}
\item $u$ is a weak solution to \eqref{e_pde} in the sense of Definition~\ref{def_weaksol}.
\item The initial trace $(\cS,\nu)$ is attained in the sense that \eqref{e_inittrace_measure2} and \eqref{e_inittrace_singular2} hold.
\end{itemize}
\end{definition}

Our main result about the initial trace problem is Theorem~\ref{thm_inittrace}, which has two parts: non-existence and existence. Most of the work has already been carried out in Sections~\ref{s_flat} and \ref{s_nonflat}. Both proofs require the following lemma.

\begin{lemma} \label{lemma_limsol} For $\mu \in \cM_F(\R^d)$, $u^{\infty \mu}_t(x) = \lim_{\lambda \to \infty} u^{\lambda \mu}_t(x)$ is a weak solution to \eqref{e_pde}. Furthermore, for all $\rho > 0$ and $z \in \text{supp}(\mu)$,
\begin{equation} \label{e_lemma_limsol_sing}
\lim_{t \to 0} \int_{B(z,\rho)} u^{\infty \mu}_t(x) \, dx = +\infty,
\end{equation}
and hence the singular set of $u_t$ contains $\text{supp}(\mu)$.
\end{lemma}
\begin{proof}
Let $\mu \in \cM_F(\R^d)$. We must show that $u^{\infty \mu}(t,x)$ satisfies \eqref{e_weaksolibp_compact} for every $\xi \in C_c^{1,2}(Q)$. For $\lambda > 0$, $u^{\lambda \mu}(t,x)$ is a solution to the problem \eqref{e_pde_measureproblem} with $u_0 = \lambda \mu$, so by definition we have
\begin{align} 
&\int_{Q} (u^{\lambda \mu}(t,x)[-\partial_t \xi(t,x) - \Delta_\alpha \xi(t,x)]) + u^{\lambda \mu}(t,x)^{1+\beta} \xi(t,x) \,dxdt = 0. \nonumber
\end{align}
Since $\xi$ has compact support, by \eqref{e_uinf_bd} the bound $u^{\lambda \mu}(t,x) \leq u^{\infty \mu}(t,x) \leq U_t$ allows us to apply Dominated Convergence and conclude that $u^{\infty \mu}(t,x)$ satisfies \eqref{e_weaksolibp_compact} for all $\xi \in C^{1,2}_c(Q)$. Similarly, $u^{\infty \mu}(t,x) \leq U_t$ implies that $u^{\infty \mu}(t,x)$ is bounded on $(\epsilon,\infty) \times \R^d$ for all $\epsilon > 0$ and hence $u^{\infty \mu} \in L^{1+\beta}_{\text{loc}}(Q)$. To see the $u^{\infty \mu}_t(x)$ is continuous, we note that for any $t_0 >0$, $(t,x) \to u^{\infty \mu}(t_0 + t,x)$ is a weak solution to \eqref{e_pde} which is globally bounded by $U_{t_0}$. In particular, (recall Remark~\ref{remark_integbounded}) it is a solution to the integral equation \eqref{e_evol} with initial data $u_{t_0}^{\infty \mu} \in \cB^+_b$ and hence is continuous. It follows that $u^{\infty \mu}_t(x)$ is continuous on $[t_0,\infty) \times \R^d$ for all $t_0>0$ and hence on $Q$. Thus $u^{\infty \mu}_t(x)$ is a weak solution to \eqref{e_pde}.

The fact that \eqref{e_lemma_limsol_sing} holds has already been shown in the proof of Lemma~\ref{lemma_umuinf_lwrdbd}; see \eqref{u_muinf_sing}.
\end{proof}

Our main result concerning flatness and non-existence is the following. We recall from \eqref{def_calU} the space $\cU$ of positive functions on $Q$ bounded above by $U_t$, and that we have restricted our attention to solutions in $\cU$. 

\begin{theorem} \label{thm_nonexistence}
Suppose that $\cS \subset \R^d$ is closed and supports a measure $\mu \in \cM_F(\R^d)$ for which $\lim_{\lambda \to \infty}u^{\lambda \mu}_t = U_t$. If $\cS$ is contained in the singular set of a solution $u$ to \eqref{e_pde} in $\, \cU$, then $u_t = U_t$. In particular, if $\cS \neq \R^d$ there is no solution to \eqref{e_pde} in $\, \cU$ with singular set $\cS$.
\end{theorem}

We obtain Theorem~\ref{thm_flatFrost}(b) and Theorem~\ref{thm_inittrace}(a) as corollaries as follows.

\begin{proof}[Proof of Theorem~\ref{thm_flatFrost}(b)]
Let $s \in [0,d]$, $\beta \leq \beta^*(\alpha,s)$ and suppose that $\mu\in \cM_F(\R^d)$ satisfies (F1)-$s$. Then $u^{\infty \mu}_t(x) = U_t$ by Theorem~\ref{thm_flatFrost}(a). Hence $\cS = \text{supp}(\mu)$ supports a measure $\mu$ for which $\lim_{\lambda \to \infty} u^{\lambda \mu}_t(x) = U_t$. Now let $\nu \in \cM_F(\R^d)$ be any measure such that $\text{supp}(\mu) \subseteq \text{supp}(\nu)$. By Lemma~\ref{lemma_limsol}, $u^{\infty \nu}_t$ is a weak solution to \eqref{e_pde} whose singular set contains $\text{supp}(\nu)$, which itself contains $\cS$. The result follows from Theorem~\ref{thm_nonexistence}.
\end{proof}

\begin{proof}[Proof of Theorem~\ref{thm_inittrace}(a)] If $\cS \subset \R^d$ is closed with $\cH^{d_{\text{sat}}} (\cS) > 0$, then by Frostman's Lemma there exists $\mu \in \cM_F(\cS)$ which satisfies (F1)-$d_{\text{sat}}$. By Theorem~\ref{thm_flatFrost}(a), $u^{\infty \mu}_t = U_t$, and the result follows from Theorem~\ref{thm_nonexistence}. \end{proof}

The proof of Theorem~\ref{thm_nonexistence} requires the following pointwise estimate for solutions with a given singular set. A more general version of this result proved for classical solutions in \cite{CV2019}, where it was called Theorem C. We include a short proof for the sake of completeness. The argument is essentially the same as the argument in the proof of Lemma~\ref{lemma_umuinf_lwrdbd}. Recall that $u^\infty_t(x) = \lim_{\lambda \to \infty} u^{\lambda \delta_0}_t(x)$.
\begin{proposition}  \label{prop_pointlowerbound}
Suppose that $u(t,x)$ is a weak solution to \eqref{e_pde} in $\,\cU$ whose singular set contains $\cS \subset \R^d$. Then for every $z \in \cS$, 
\[u(t,x) \geq u^\infty(t,x-z) \geq c_{\ref{prop_pointlowerbound}} t^{-\frac 1 \beta}p_1(t^{-\frac 1 \alpha}(x-z))\]
for all $(t,x) \in Q$, where $c_{\ref{prop_pointlowerbound}} > 0$ depends only on $(\alpha,\beta,d)$.
\end{proposition}
\begin{proof} Suppose that $z \in \cS$, the singular set of $u(t,x)$. Let $\lambda > 0$. Since \eqref{e_inittrace_singular2} holds, there must be sequences $(t_n)_{n \geq 1}$ and $(\rho_n)_{n \geq 1}$ such that $t_n, \rho_n \to 0$ and 
\begin{equation}
\int_{B(z,\rho_n)} u(t_n,x) \, dx = \lambda. \nonumber
\end{equation}
Let $\phi_n(x) = 1_{B(z,\rho_n)}(x) u(t_n,x)$. We then have
\[ u(t_n,x) \geq \phi_n(x)\]
for all $x \in \R^d$. Since $u(t_n,\cdot)$ and $\phi_n$ are both bounded (because $u \in \cU$), it follows from the comparison principle (recall Remark~\ref{remark_comp_prin}) that
\begin{equation} \label{e_tracemono}
u(t_n + t,x) \geq u^{\phi_n}(t,x)
\end{equation}
for all $(t,x) \in Q$. Note that $\phi_n \to \lambda\delta_z$ in $\cM_F(\R^d)$, so by Lemma~\ref{lemma_stability}(a) and translation invariance, $u^{\phi_n}(t,x) \to u^\lambda(t,x-z)$. (Recall that $u^\lambda_t(x) = u^{\lambda \delta_0}(t,x)$.) Of course, $\lim_{n \to \infty} u(t_n + t,x) = u(t,x)$, and so taking $n \to \infty$ in \eqref{e_tracemono} we obtain
\[ u(t,x) \geq u^\lambda(t,x-z).\]
Since this holds for all $\lambda >0$, it follows that $u(t,x) \geq u^{\infty}(t,x-z)$. The second inequality in the result then follows from \eqref{e_uinf_lwrbound}.
\end{proof}

\begin{proof}[Proof of Theorem~\ref{thm_nonexistence}]
Let $\cS$ and $\mu$ be as in the statement of the theorem. Suppose that $u_t(x)$ is a solution to \eqref{e_pde} in $\cU$ whose singular set contains $\cS$. By Proposition~\ref{prop_pointlowerbound}, for some constant $c>0$ we have
\begin{align} \label{e_nonex1}
u_t(x) &\geq \sup_{z \in \cS} c t^{-\frac 1 \beta} p_1(t^{-\frac 1 \alpha}(x-z)) \nonumber
\\  &= c t^{-\frac 1 \beta} p_1(t^{-\frac 1 \alpha}d(x,\cS)), 
\end{align}
where the second equality holds because $\cS$ is closed and $p_1$ is continuous and radially decreasing. By assumption, $\cS$ supports a finite measure $\mu$ such that $\lim_{\lambda \to \infty} u^{\lambda \mu}_t = U_t$. Fix $\lambda > 0$. By \eqref{e_heatcomparison},
\begin{align} \label{e_nonex2}
u^{\lambda \mu}_t(x) &\leq \lambda S_t \mu (x)\nonumber
\\ & \leq \lambda \mu(1) p_t(d(x,\cS)) \nonumber
\\ &= \lambda \mu(1) t^{-\frac d \alpha} p_1(t^{-\frac 1 \alpha}d(x,\cS)).
\end{align}
Since $\beta < \frac \alpha d$, \eqref{e_nonex1} and \eqref{e_nonex2} imply that there exists $t_0(\lambda) > 0$ such that 
\begin{equation}
u^{\lambda \mu}_t(x) \leq u_t(x) \,\, \text{  for all } x \in \R^d \text{ and } t \in (0, t_0(\lambda)].
\end{equation}
Applying the comparison principle (as in Remark~\ref{remark_comp_prin}) at time $t_0(\lambda)$, it follows that 
\begin{equation}
u^{\lambda \mu}_t(x) \leq u_t(x) \,\, \text{  for all } (t,x) \in Q.
\end{equation}
The above holds for all $\lambda > 0$. Since $\lim_{\lambda \to \infty} u^{\lambda \mu}_t = U_t$ and $u \in \cU$, it follows that $u_t(x) = U_t$.\end{proof}


It remains to prove our existence result, Theorem~\ref{thm_inittrace}(b), which states that if $\cS \subset \R^d$ is compact and $\cS = \text{supp}(\mu)$ for $\mu \in \cM_F(\R^d)$ satisfying (F2)-$s$ for some $s < d_{\text{sat}}$, then there exists a weak solution of \eqref{e_pde} with initial trace $(\cS,0)$ if $\beta^*(\alpha,s) < \beta < \frac \alpha d$. This solution is $u^{\infty \mu}_t$.

\begin{proof}[Proof of Theorem~\ref{thm_inittrace}(b)] Let $\mu \in \cM_F(\R^d)$ satisfy (F2)-$s$ and have compact support $\cS$. By Lemma~\ref{lemma_limsol}, $u^{\infty \mu}_t(x)$ is a weak solution to \eqref{e_pde} and $\cS$ is contained in the singular set of $u^{\infty \mu}_t$. By \eqref{e_thmnonflat_upperbound}, $u^{\infty \mu}_t(x)$ vanishes uniformly on $\{ x : d(x,\cS) \geq \epsilon \}$ as $t \downarrow 0$ for any $\epsilon > 0$. It follows that for any $\xi \in C_c(\cS^c)$, 
\[ \lim_{t\to 0} \int \xi(x) u^{\infty \mu}_t(x) dx = 0.\] 
Hence the singular set of $u^{\infty \mu}_t$ is no larger than $\cS$ and \eqref{e_inittrace_measure} holds with measure $\nu = 0$, which implies that $u^{\infty \mu}_t$ has initial trace $(\cS,0)$. \end{proof}

\vspace{0.5 cm}

\noindent \textbf{Acknowledgements.} This work is part of the author's PhD dissertation that the University of British Columbia. The author thanks his supervisor, Ed Perkins, for many useful discussions, including suggestions which led to the proof of Theorem~\ref{thm_positiveEverwhere}, and for giving close readings of the manuscript. The author also thanks Leonid Mytnik for encouraging him to read \cite{CVW2016}, which was the starting point for this work.

\end{document}